\documentclass[english,11pt]{amsart}
\usepackage[top=1in,bottom=1in,left=1.25in,right=1.25in]{geometry}                
\usepackage{graphicx}
\usepackage{amssymb}
\usepackage{epstopdf}
\usepackage{stmaryrd}
\usepackage{amsthm}
\usepackage{mathrsfs}
\usepackage{txfonts}
\usepackage[T1]{fontenc}
\usepackage[latin9]{inputenc}
\usepackage{amsfonts}
\usepackage{babel}
\usepackage{amsmath}
\usepackage{color}%
\usepackage{hyperref}
\usepackage{indentfirst}
\title[Weak solutions of the SLLG Equations with non-zero anisotrophy energy]{Weak solutions of the Stochastic Landau-Lifschitz-Gilbert Equations with non-zero anisotrophy energy}
\author{Zdzis{\l}aw Brze\'zniak and  Liang Li}

\input colordvi

\newcommand\delc[1]{}

\newcommand\rH{\mathrm{H}}

\newcommand\lb{\langle}
\newcommand\rb{\rangle}
\newcommand\bis{{\prime\prime}}

\makeatletter
\numberwithin{equation}{section}
\numberwithin{figure}{section}
\@ifundefined{theoremstyle}{\usepackage{amsthm}}{}
\theoremstyle{plain}
\newtheorem{thm}{Theorem}[section]
\theoremstyle{remark}

\newtheorem*{acknowledgement*}{Acknowledgement}
\theoremstyle{plain}

\theoremstyle{plain}

\newcounter{casectr}

\theoremstyle{remark}

\theoremstyle{remark}

\theoremstyle{definition}

\theoremstyle{plain}

\theoremstyle{plain}
\newtheorem{cor}[thm]{Corollary}
\theoremstyle{plain}

\theoremstyle{definition}

\theoremstyle{definition}
\newtheorem{defn}[thm]{Definition}
\theoremstyle{definition}
\theoremstyle{definition}

\theoremstyle{plain}

\theoremstyle{plain}
\newtheorem{lem}[thm]{Lemma}
\theoremstyle{remark}
\newtheorem{notation}[thm]{Notation}
\theoremstyle{definition}

\theoremstyle{plain}
\newtheorem{prop}[thm]{Proposition}
\theoremstyle{remark}
\newtheorem*{rem*}{Remark}
\theoremstyle{remark}
\newtheorem{rem}[thm]{Remark}
\theoremstyle{remark}

\theoremstyle{plain}
\newtheorem{assumption}[thm]{Assumption}

\numberwithin{equation}{section}
\makeatother

\newcommand{\ud}{\,\mathrm{d}}

\newcommand{\lmd}{\lambda}

\newcommand{\R}{\mathbb{R}}

\newcommand{\eps}{\varepsilon}

\newcommand{\esup}{\mathop{\mathrm{ess}\sup}}

\newcommand{\lspan}{\mathrm{linspan}}
\newcommand{\llangle}{\left\langle}
\newcommand{\rrangle}{\right\rangle}

\newcommand\rv{\mathrm{v}}

\newcommand{\Ltwo}{\mathbb{L}^2}

\begin{document}

\email{zdzislaw.brzezniak@york.ac.uk, liangli@amss.ac.cn}
\begin{abstract}
We study a stochastic Landau-Lifschitz-Gilbert Equation (c) with non-zero anisotrophy energy and multidimensional noise. We prove the existence  and some regularities of weak solution  proved. Our paper is motivated by finite-dimensional study of stochastic LLGEs or general stochasric differental equations with constraints studied by Kohn et al \cite{Kohn_2005} and Leli{\`e}vre et al \cite{LeBris_2008}.
\end{abstract}

\keywords{stochastic partial differential equations, ferromagnetism, anisotrophy, heat flow}

\maketitle

\tableofcontents

\section{Introduction}
The ferromagnetism theory was first studied by Wei{\ss} in 1907 and then further developed by Landau and Lifshitz \cite{Landau} and Gilbert \cite{Gilbert}. According to  their theory there is a characteristic of the material called the Curie's temperature, whence below this critical temperature, large ferromagnetic bodies  would break up into small uniformly magnetized regions  separated by thin transition layers. The small uniformly magnetized regions are called Wei{\ss} domains and the transition layers are called Bloch walls. This fact is taken into account by imposing the following constraint:
\begin{equation}
  |u(t,x)|_{\R^3}=|u_0|_{\R^3}.
\end{equation}
Moreover the magnetization in a domain $D\subset \R^3$ at time $t>0$ given by $u(t,x)\in \R^3$ satisfies the following Landau-Lifschitz equation:
\begin{equation}\label{eq:LLeq}
  \frac{\ud u(t,x)}{\ud t}=\lmd_1u(t,x)\times \rho(t,x)-\lmd_2u(t,x)\times (u(t,x)\times \rho(t,x)).
\end{equation}
The $\rho$ in the equation \eqref{eq:LLeq} is called the effective magnetic field and defined by
\begin{equation}\label{eq:effmag}
  \rho=-\nabla_u\mathcal{E},
\end{equation}
where the $\mathcal{E}$ is the so called total electro-magnetic energy which composed by anisotropy energy, exchange energy and electronic energy.

In order to describe phase transitions between different equilibrium states induced by thermal fluctuations of the effective magnetic field $\rho$, Brze{\'z}niak and Goldys and Jegaraj \cite{ZB&TJ} introduced the Gaussian noise into the Landau-Lifschitz-Gilbert (LLG) equation to perturb $\rho$ and then the stochastic Landau-Lifschitz-Gilbert (SLLG) equation have the following form:
\begin{equation}\label{eq:effmags}
\ud u(t)=\left(\lmd_1u(t)\times\rho(t)-\lmd_2u(t)\times(u(t)\times\rho(t))\right)\ud t+(u(t)\times h)\circ \ud W(t),
\end{equation}
where $h\in L^\infty(D;\R^3)$.
Their total energy contains only the exchange energy $\frac{1}{2}\|\nabla u\|_{{\mathbb{L}^2}}$, and hence their equation has the following form:
\begin{equation}\label{eq:zbeq}
  \left\{\begin{array}{l}
  \ud u(t)=(\lmd_1u(t)\times \Delta u(t)-\lmd_2u(t)\times(u(t)\times \Delta u(t)))\ud t+(u(t)\times h)\circ\ud W(t),\\
  \frac{\partial u}{\partial n}(t,x)=0,\qquad t>0,x\in \partial D,\\
  u(0,x)=u_0(x),\qquad x\in D.\end{array}\right.
\end{equation}
They concluded the existence of the weak solution of \eqref{eq:zbeq} and also proved some regularities of the solution.

There is also some research about the numerical schemes of equation \eqref{eq:zbeq}, such as Ba{\v n}as, Brze{\'z}niak, and Prohl \cite{LB&ZB&AP},
 Ba{\v n}as, Brze{\'z}niak, Neklyudov, and Prohl \cite{LB&ZB&AP&MN}, Ba{\v n}as, Brze{\'z}niak, Neklyudov, and Prohl \cite{LB&ZB&AP&MNbk},  Goldys, Le, and Tran \cite{Goldys&Le} and Alouges, de Bouard and Hocquet \cite{Alouges_2014}. The last paper differs from all previous papers as it deals with so called Gilbert form of the LLGEs, see \cite{Gilbert} and  \cite{Alouges_1992} for some related deterministic results.

In the present  paper we consider the SLLG equation with  the total energy $\mathcal{E}$ consisting of the exchange and anisotropy energies and hence it  defined as:
\[\mathcal{E}(u)=\mathcal{E}_{an}(u)+\mathcal{E}_{ex}(u)=\int_D\left(\phi(u(x))+\frac{1}{2}|\nabla u(x)|^2\right)\ud x,\]
where $\mathcal{E}_{an}(u):=\int_D\phi(u(x))\ud x$ stands for the anisotropy energy and $\mathcal{E}_{ex}(u):=\frac{1}{2}\int_D|\nabla u(x)|^2\ud x$ stands for the exchange energy.

Our study is motivated by finite-dimensional study of stochastic LLGEs or general stochastic differential equations with constraints studied by Kohn et al \cite{Kohn_2005} and Leli{\`e}vre et al \cite{LeBris_2008}. An essential feature of the model studies in \cite{Kohn_2005} was the presence of anisotropy energy (while the exchange energy was absent). So far none of the papers, apart from \cite{ZB&BG&TJ} which treats only one-dimensional domains,  on the stochastic LLGEs considered nonzero anisotropy energy. Therefore there is a need to fill this literature gap and that is what we have achieved in the current work. 

The  main novelty  of the current paper lies in being able to  study of LLGEs with energy including    the
anisotropy energy. As we have mentioned earlier, both the    papers by the first named authour, Goldys and Jegaraj and by    Alouges,  De Bouard and Hocquet, treat  purely exchange energy. \delc{Secondly, we showed that it is possible to have the noisy part of the Hamiltonian in both nonlinear terms, the one corresponding to gyroscopic motion and the one corresponding to dissipation of energy.}
Our success  was possible because  we have been able to find uniform {\em a priori} estimates for the suitable
Galerkin approximations  of the full problem. This was in turn possible because our Galerkin approximation could be seen as an equation of a finite dimensional Hilbert space $H_n$ of the form similar to the full equations.

So the SLLG equation we are going to study in this paper has the form:
\begin{equation}\label{eq:7eqintr}\displaystyle
\left\{\begin{array}{ll}
\ud u(t)&=\Big[\lmd_1u(t)\times\left(\Delta u(t)-\nabla \phi\big(u(t)\big)\right)\\
\displaystyle
&-\lmd_2u(t)\times\Big(u(t)\times\left(\Delta u(t)
\;\;-\nabla \phi\big(u(t)\big)\right)\Big)\Big]\ud t +\displaystyle\sum_{j=1}^N\big(u(t)\times h_j\big)\circ\ud W_j(t),\\
\left.\frac{\partial u}{\partial n}\right|_\Gamma&=0,\\
u(0)&=u_0,\end{array}\right.
\end{equation}
where  $h_j\in L^\infty(D;\R^3)\cap \mathbb{W}^{1,3}$,  for $j=1,\cdots, N$ and  some $N\in \mathbb{N}$; see Assumption \ref{as:all4}.

Let me describe on a heuristic level the idea of the proof. For this let us denote by $M$ the set of all functions $u\in \rH=L^2(D;\mathbb{R}^3)$ such that $u (x) \in \mathbb{S}^2$ for a.a. $x\in D$, where $\mathbb{S}^2$ is the unit sphere in $\mathbb{R}^3$. Since for  $u \in H^2(D;\mathbb{R}^3)\cap M$ the $\rH$-orthogonal projection on $T_u M$ is equal to the map $\rH\ni z \mapsto - u \times \Bigl( u\times z \Bigr)\in T_u M $, and
$\Delta u -\nabla \phi(u)$ is equal to $-\nabla_{\rH} \mathcal{E}(u)$,  the $-\rH$ gradient of the total energy $\mathcal{E}$, the second deterministic term on the RHS of \eqref{eq:7eqintr} (modulo $\lambda_2$), is equal to $ -\nabla_{M} \mathcal{E}(u)$, the -gradient of the total energy $\mathcal{E}$ with respect to the riemannian structure of $M$ inherited from $\rH$. Similarly, the first deterministic term on the RHS of \eqref{eq:7eqintr} (modulo $\lambda_1$) is equal to $-u \times \bigl( -\nabla_{M} \mathcal{E}(u)\bigr)$ and hence is   perpendicular to $ \nabla_{M} \mathcal{E}(u)$.
Note also that for each $j$, $M \ni u \mapsto u \times h_j \in T_uM$, so that the function $u \times h_j$ could be seen as a (tangent!) vector field $v_j$ on $M$. Therefore, the the first equation of the system \eqref{eq:7eqintr} could be written in the following geometric form
\begin{eqnarray}\label{eq:7eqintr'}
\qquad\ud u(t)&=&\Big[ \lmd_1 u \times \bigl( \nabla_{M} \mathcal{E}(u)\bigr)
-\lmd_2  \nabla_{M} \mathcal{E}(u)+ \frac12 \sum_{j=1}^N  \bigl( u \times h_j\bigr) \times h_j    \Big] \ud t +\sum_{j=1}^N\big(u(t)\times h_j\big)\ud W_j(t).
\end{eqnarray}
Thus, on a purely heuristics level, applying the It\^o Lemma to the function $\mathcal{E}$ and a solution $u$ to \eqref{eq:7eqintr} we get
\begin{eqnarray}\nonumber
   d \mathcal{E}(u(t))&=& \lambda_1 \lb \nabla_{M} \mathcal{E}(u), u \times \bigl( \nabla_{M} \mathcal{E}(u)\bigr) \rb\, dt
   -\lambda_2 \lb \nabla_{M} \mathcal{E}(u), \nabla_{M} \mathcal{E}(u) \rb\, dt\\
    &+& \frac12 \sum_{j=1}^N \lb \nabla_{M} \mathcal{E}(u), \bigl( u \times h_j\bigr) \times h_j \rb \, dt + \sum_{j=1}^N \lb \nabla_{M} \mathcal{E}(u), u \times h_j \rb \, dW_j
    \nonumber \\
   &+& \frac12 \sum_{j=1}^N  \lb \nabla_{M}^2 \mathcal{E}(u) \bigl( u \times h_j\bigr),  u \times h_j \rb \, dt
\nonumber   \\
\nonumber
&=&
   -\lambda_2 \vert  \nabla_{M} \mathcal{E}(u)\vert^2 \, dt\\
 \nonumber   &+& \frac12 \sum_{j=1}^N \lb \nabla_{M} \mathcal{E}(u), \bigl( u \times h_j\bigr) \times h_j \rb \, dt + \sum_{j=1}^N \lb \nabla_{M} \mathcal{E}(u), u \times h_j \rb \, dW_j\\
   &+& \frac12 \sum_{j=1}^N  \lb \nabla_{M}^2 \mathcal{E}(u) \bigl( u \times h_j\bigr),  u \times h_j \rb \, dt
 \label{eqn-explanation01}
   \end{eqnarray}
The above equality could be seen as an {\em a priori} estimate but there are two  problems. Firstly,   we do not have a solution and secondly, even if we had it, it might not be strong or regular enough for the applicability of the It\^o Lemma. A standard procedure is to approximate the full equation by some simpler problems. In the paper \cite{ZB&TJ} we used Galerkin approximation, in a series of works with Banas,  Prohl and Neklyudov culminating in a monograph \cite{LB&ZB&AP&MNbk}, we used the finite element approximation. Here  We follow the same method as used in Brze{\'z}niak, Goldys and Jegaraj's paper \cite{ZB&TJ} but with one important addition. We introduce the finite dimensional Hilbert spaces $\rH_n$ as in \cite{ZB&TJ} but then consider an analog of equation \eqref{eq:7eqintr'} on the space $\rH_n$ with $M$ replaced by $M_n=M \cap \rH$ and the energy function $\mathcal{E}$ replaced by $\mathcal{E}_n$, the restriction of the former to $\rH_n$. Thus, we consider
\begin{equation}\label{eq:7eqintr''}
\ud u_n(t)=\Big[ \lmd_1 u_n \times \bigl( \nabla_{M_n} \mathcal{E}_n(u_n)\bigr)
-\lmd_2  \nabla_{M_n} \mathcal{E}_n(u_n)+ \frac12 \sum_{j=1}^N  \pi_n \Bigl(  \bigl( \pi_n( u_n \times h_j) \bigr) \times h_j \Bigr)    \Big] \ud t +\sum_{j=1}^N\pi_n \big(u_n(t)\times h_j\big)\ud W_j.
\end{equation}
where $\pi_n : \rH \to \rH_n$ is the orthogonal projection. Equation \eqref{eq:7eqintr''} is nothing else but equation \eqref{eq:dun} or \eqref{eq:dun2}. Now, the above problem is an SDE in a finite dimensional space $\rH_n$ and hence it has a unique local maximal solution $u_n$. Applying the, now correct, It\^o lemma to process $u_n$ and the function $\mathcal{E}_n$ we get an analog of identity
\begin{eqnarray}\nonumber
   d \mathcal{E}(u_n(t))&+&
   \lambda_2 \vert  \nabla_{M_n} \mathcal{E}_n(u_n)\vert^2 \, dt\\
 \nonumber   &=& \;\;\;\;\;\frac12 \sum_{j=1}^N \lb \nabla_{M_n} \mathcal{E}(u_n), \pi_n \Bigl(  \bigl( \pi_n( u_n \times h_j) \bigr) \times h_j \Bigr) \rb \, dt + \sum_{j=1}^N \lb \nabla_{M_n} \mathcal{E}_n(u_n), \pi_n\bigl( u_n \times h_j\bigr) \rb \, dW_j\\
   &+& \frac12 \sum_{j=1}^N  \lb \nabla_{M_n}^2 \mathcal{E}_n(u_n) \bigl( \pi_n\bigl( u_n \times h_j\bigr) \bigr),  \pi_n\bigl( u_n \times h_j\bigr) \rb \, dt
 \label{eqn-explanation02}
   \end{eqnarray}
Since $\mathcal{E}_n$ is the restriction of $\mathcal{E}$ to $\rH_n$, $\nabla_{M_n} \mathcal{E}_n(z)=\nabla_{} \mathcal{E}(z)\circ i_n$, $z\in M_n$, where $i_n:\rH\to \rH$ is the natural embedding.  Similarly, $\nabla_{M_n}^2 \mathcal{E}_n(z)=\pi_n \nabla_{M}^2 \mathcal{E}(z) i_n$.

We follow the same method as used in Brze{\'z}niak, Goldys and Jegaraj's paper \cite{ZB&TJ} to proved the existence of the weak solution of \eqref{eq:7eqintr} and get some similar regularities of the weak solution (but not uniqueness).

In particular, our results give an alternative proof of the existence result from Brze{\'z}niak, Goldys and Jegaraj's paper \cite{ZB&BG&TJ}, where large deviations principle for stochastic LLG equation on a $1$-dimensional domain has been studied.

This paper is organized as follows. In Section 2 we introduce the notations and formulate the main result on the existence of the weak solution of the Equation \eqref{eq:7eqintr} as well as some regularities. In Section 3 we introduce the Galerkin approximation and prove the existence of the global solutions $\{u_n\}$ of the approximate equation of \eqref{eq:7eqintr}, which are in $n$ dimensional spaces, where $n\in\mathbb{N}$. In Section 4 we prove the global solutions of the approximate equations in finite dimensional spaces satisfy some {\em a priori} estimates. In Section 5, we use the {\em a priori} estimates to show the laws of the $\{u_n\}$ are tight on a suitable space. In Section 6, we use the tightness results and the Skorohod's Theorem to construct a new probability space and some processes $\{u_n^{\prime}\}$ which have the same laws as $\{u_n\}$. By the Skorohod's Theorem, we also get a limit process $u^{\prime}$ of $\{u_n^{\prime}\}$. And we show some properties that $u^{\prime}$ satisfies. In Section 7, we use two steps to show that $u^{\prime}$ constructed before is a weak solution of the Equation \eqref{eq:7eqintr}. In Section 8, we prove some regularities of $u^{\prime}$ and so finish the proof of the main Theorem which stated in Section 2.

Let us finish the introduction by remarking that all our results are formulated for $D\subset \R^d$, $d=3$, but they are also valid for $d=1$ or $d=2$.

\begin{rem*}
  This paper is from a part of the Ph.D. thesis at the University of York in UK of the second named author .
\end{rem*}
\subsection*{Acknowledgments}
  The authors would like to anonymous referees  for a careful reading of the
  manuscript and pointing out many errors and making many useful suggestions. These have lead to an improvement of the paper.

\section{Notation and the formulation of the main result}
\begin{notation}
  Let us denote the classical spaces:
  $$\mathbb{L}^p:=L^p(D;\R^3) \textrm{ or }L^p(D;\R^{3\times 3}),$$
  $$\mathbb{W}^{k,p}:=W^{k,p}(D;\R^3),\;\mathbb{H}^{k}:=H^{k}(D;\R^3),\;\textrm{and }\mathbb{V}:=\mathbb{W}^{1,2}.$$
The dual brackets between a space $X$ and its dual $X^*$ will be denoted ${}_{X^*} \langle\cdot, \cdot\rangle_X$. A scalar
product in Hilbert space $H$ will be denoted $ \langle\cdot, \cdot\rangle_H$ and its associated norm $\| \cdot \|_H.$
\end{notation}

\begin{assumption}\label{as:all4}
  Let $D$ be an open and bounded domain in $\R^3$ with $C^2$ boundary $\Gamma:=\partial D.$ $n$ is the outward normal vector on $\Gamma$.
   $\lmd_1\in\R$, $\lmd_2>0$, $  h_j\in \mathbb{L}^\infty\cap \mathbb{W}^{1,3}$, for $j=1,\ldots,N$, $u_0\in {\mathbb{V}}$. $\phi:\R^3\longrightarrow \R^+\cup\{0\}$ is in $C^4$ and $\phi$,  $\nabla \phi$, $\phi^\bis$ and $\phi^{(3)}$ are bounded. $\nabla \phi$ is also globally Lipschitz.
    Moreover, we also assume that we have a filtered probability space $(\Omega,\mathcal{F},\mathbb{F}=(\mathcal{F}_t)_{t\geq0},\mathbb{P})$, and this probability space satisfies the so called usual conditions:
   \begin{trivlist}
     \item[(i)] $\mathbb{P}$ is complete on $(\Omega,\mathcal{F})$,
     \item[(ii)] for each $t\geq0$, $\mathcal{F}_t$ contains all $(\mathcal{F},\mathbb{P})$-null sets,
     \item[(iii)] the filtration $(\mathcal{F}_t)_{t\geq0}$ is right-continuous.
   \end{trivlist}
   We also assume that $(W(t))_{t\geq0}=\big((W_j)_{j=1}^N(t)\big)_{t\geq0}$ is a $\R^N$-valued, $(\mathcal{F}_t)_{t\geq0}$-adapted Wiener process defined on $(\Omega,\mathcal{F},(\mathcal{F}_t)_{t\geq0},\mathbb{P})$.
\end{assumption}

In  this paper we are going to study the following equation.
\begin{equation}\label{eq:7eq}
\left\{\begin{array}{ll}
\ud u(t)&=\Big[\lmd_1u(t)\times\left(\Delta u(t)-\nabla \phi\big(u(t)\big)\right)\\&-\lmd_2u(t)\times\Big(u(t)\times\left(\Delta u(t)
\;\;-\nabla \phi\big(u(t)\big)\right)\Big)\Big]\ud t +\displaystyle \sum_{j=1}^N\big(u(t)\times h_j\big)\circ\ud W_j(t)\\
\left.\frac{\partial u}{\partial n}\right|_\Gamma&=0\\
u(0)&=u_0\end{array}\right.
\end{equation}

\begin{rem}
Since   $\phi:\R^3\longrightarrow \R\in C^4$,  for every $x\in \R^3$ the Frechet derivative  $\ud_x\phi=\nabla \phi(x):\R^3\longrightarrow \R$ is linear, and hence by the Riesz Lemma,  there exists a vector $\nabla \phi(x)\in \R^3$ such that
\[ \lb \nabla \phi(x),y\rb_{\R^3}=\ud_x \phi (y),\;\;\; y\in \R^3.\]
\end{rem}

\begin{defn}[Solution of \eqref{eq:7eq}]\label{defn:sol}
A weak solution of \eqref{eq:7eq} is system consisting of a filtered probability space $(\Omega^{\prime},\mathcal{F}^{\prime},\mathbb{F}^{\prime},\mathbb{P}^{\prime})$,
an $N$-dimensional $\mathbb{F}^{\prime}$-Wiener process
$W^{\prime}=(W_j^{\prime})_{j=1}^N$ and an $\mathbb{F}^{\prime}$-progressively measurable process \[u^{\prime}=(u_i')_{i=1}^3:\Omega^{\prime}\times[0,T]\longrightarrow {\mathbb{V}}\cap\mathbb{L}^\infty\]
such that  for all $\psi\in C_0^\infty(D;\R^3)$, $t\in[0,T]$, we have, $\mathbb{P}^{\prime}$-a.s.,
\begin{eqnarray}
 \langle u^{\prime}(t),\psi\rangle_{{\mathbb{L}^2}}&=&\langle u_0,\psi\rangle_{{\mathbb{L}^2}}-\lmd_1\int_0^t\langle\nabla u^{\prime}(s),\nabla \psi\times u^{\prime}(s)\rangle_{\Ltwo}\ud s\nonumber\\
 &&+\lmd_1\int_0^t\langle u^{\prime}(s)\times \nabla\phi(u^{\prime}(s)),\psi\rangle_{\mathbb{L}^2} \ud s\nonumber\\
  &&-\lmd_2\int_0^t\langle \nabla u^{\prime}(s),\nabla (u^{\prime}\times\psi)(s)\times u^{\prime}(s)\rangle_{\Ltwo}\ud s\label{eq:defsol}\\
  &&+\lmd_2\int_0^t\langle u^{\prime}(s)\times (u^{\prime}(s)\times \nabla \phi(u^{\prime}(s)),\psi\rangle_{\mathbb{L}^2} \ud s\nonumber\\
  &&+\sum_{j=1}^N\int_0^t\langle u^{\prime}(s)\times h_j,\psi\rangle_{\mathbb{L}^2}\circ\ud W_j^{\prime}(s).\nonumber
\end{eqnarray}

Next we will formulate the main result of this paper:
\begin{thm}\label{thm:sum4}
Under the assumptions listed in Assumption \ref{as:all4}
  , i.e. a system consisting of a filtered probability space $(\Omega^{\prime},\mathcal{F}^{\prime},\mathbb{F}^{\prime},\mathbb{P}^{\prime})$,
and $N$-dimensional $\mathbb{F}^{\prime}$-Wiener process
$W^{\prime}=(W_j^{\prime})_{j=1}^N$.
  \begin{trivlist}
    \item[(i)] There exists a weak solution of \eqref{eq:7eq}.
    \item[(ii)]
    \[\mathbb{E}\int_0^T\left\|u'(t)\times\Delta u'(t)-u'(t)\times\nabla\phi(u'(t))\right\|^2_{\mathbb{L}^2}\ud t<\infty.\]
    \item[(iii)] For every $t\in[0,T]$,  in $L^2(\Omega^{\prime};{\mathbb{L}^2})$,
  \begin{eqnarray*}
    u^{\prime}(t)&=&u_0+\lmd_1\int_0^t\left(u^{\prime}\times\Delta u^{\prime}-u^{\prime}\times\nabla\phi(u^{\prime})\right)(s)\ud s\\
    &&-\lmd_2\int_0^tu^{\prime}(s)\times\left(u^{\prime}\times\Delta u^{\prime}-u^{\prime}\times\nabla\phi(u^{\prime})\right)(s)\ud s\nonumber\\
    &&+\sum_{j=1}^N\int_0^t(u^{\prime}(s)\times h_j)\circ \ud W_j^{\prime}(s);\nonumber
  \end{eqnarray*}
    \item[(iv)]
  \begin{equation*}
  |u^{\prime}(t,x)|_{\R^3}=1, \quad\textrm{for Lebesgue a.e. } (t,x)\in [0,T]\times D \textrm{ and }\mathbb{P}^{\prime}-a.s..
  \end{equation*}
    \item[(v)] For every $\alpha\in (0,\frac{1}{2})$,
  \begin{equation*}
  u^{\prime}\in C^\alpha([0,T];{\mathbb{L}^2}),\qquad \mathbb{P}^{\prime}-a.s..
  \end{equation*}
  \end{trivlist}
\end{thm}
\end{defn}

\begin{rem}
  The notation $u^{\prime}\times \Delta u^{\prime}$ used in Theorem \ref{thm:sum4} will be defined in the Notation \ref{notation:uxDeltau}.\\
  The notation $u^{\prime}\times ( u^{\prime}\times \Delta u^{\prime})$ used in Theorem \ref{thm:sum4} will be defined in the Notation \ref{notation:uxuxDeltau}.
\end{rem}

\begin{rem}
  Our results are for the Laplace operator with Neumann boundary conditions. Without any difficult work one could prove the same result for the Laplace operator on a compact manifold without boundary. In particular, for Laplace operator with periodic boundary condition.
\end{rem}

\section{Galerkin approximation}
Let us define $A:=-\Delta$ as the $-$Laplace operator in $D$ acting on $\R^3$-valued functions with Neumann boundary condition:
\[D(A)=\left\{u\in \mathbb{H}^2:\frac{\partial u}{\partial n}\Big|_{\partial D}=0\right\}\subset {\mathbb{L}^2}. \]
$A$ is self-adjoint, so by (\cite[Thm 1, p. 335]{Evans}), there exists an orthonormal basis (which are eigenvectors of $A$)  $\{e_k\}_{k=1}^\infty$ of ${\mathbb{L}^2}$, such that $e_k\in C^\infty(\bar{D})$ for all $k=1,2,\ldots,$. We set $H_n=\lspan\{e_1,e_2,\ldots,e_n\}$ and let $\pi_n$ denote the orthogonal projection from ${\mathbb{L}^2}$ to $H_n$. We also note that ${\mathbb{V}}=D(A^\frac{1}{2})$ and define $A_1:=I+A$, then ${\mathbb{V}}=D(A_{1}^\frac{1}{2})$ and $\|u\|_{\mathbb{V}}=\|A_1^\frac{1}{2}u\|_{\mathbb{L}^2}$ for $u\in {\mathbb{V}}$.\\
We also have the following definition and properties relate to the operator $A$, which will be frequently used later:

\begin{defn}[Fractional power spaces of $A_1=I+A$]
  For any nonnegative real number $\beta$ we define the Hilbert space $X^\beta:=D(A_1^\beta)$, which is the domain of the fractional power operator $A_1^\beta$. And the dual of $X^\beta$ is denoted by $X^{-\beta}$. See \cite{ZB&TJ}.
\end{defn}

We have the following property about the relations of $X^\gamma$ and $H^{2\gamma}$.
\begin{prop}\label{prop:DAH} With $A_1=I+A$ as above we have, see \cite[4.3.3]{Triebel},
  \[X^\gamma=D(A_1^\gamma)=\left\{\begin{array}{ll}
  \left\{u\in \mathbb{H}^{2\gamma}:\frac{\partial u}{\partial n}\Big|_{\partial D}=0\right\},&\mbox{ if }\;2\gamma>\frac{3}{2},\\
  \mathbb{H}^{2\gamma},&\mbox{ if }\; 2\gamma<\frac{3}{2}.\end{array}\right.\]
\end{prop}

\begin{prop}\label{prop:stokes on neu}
Let $D$ be a bounded open domain in $\R^3$ with $C^2$ boundary, $u\in \mathbb{H}^2$, $\rv\in \mathbb{V}$, and $\left.\frac{\partial u}{\partial n}\right|_{\partial D}=0$ then we have
  \[\langle Au,\rv\rangle_{{\mathbb{L}^2}}=\int_D\lb \nabla u(x),\nabla \rv(x)\rb_{\R^{3\times 3}}\ud x.\]
\end{prop}

\begin{prop}\label{prop:uxAuweak}
  If $\rv\in {\mathbb{V}}$ and $u\in D(A)$, then
  \begin{equation}
    \int_D\llangle u(x)\times Au(x),Au(x)\rrangle_{\R^3} \ud x=0.
  \end{equation}
  \begin{equation}
    \int_D\llangle u(x)\times(u(x)\times Au(x)),Au(x)\rrangle_{\R^3} \ud x=-\int_D|u(x)\times Au(x)|^2\ud x.
  \end{equation}
  \begin{equation}\label{eq:2.9}
    \int_D\llangle u(x)\times Au(x),\rv(x)\rrangle_{\R^3} \ud x=\sum_{i=1}^3\int_D\left\langle \frac{\partial u}{\partial x_i}(x),\frac{\partial \rv}{\partial x_i}(x)\times u(x)\right\rangle_{\R^3}\ud x.
  \end{equation}
  \begin{equation}\label{eq:2.91}
    \int_D\langle u(x)\times(u(x)\times Au(x)),\rv(x)\rangle_{\R^3} \ud x=\sum_{i=1}^3\int_D\left\langle \frac{\partial u}{\partial x_i}(x),\frac{\partial (\rv\times u)}{\partial x_i}(x)\times u(x)\right\rangle_{\R^3}\ud x.
  \end{equation}
\end{prop}
\begin{proof}[Proof of \eqref{eq:2.9} and \eqref{eq:2.91}]
The equality \eqref{eq:2.9} follows from Brze{\'z}niak, Goldys and Jegaraj's paper \cite{ZB&TJ}. And since $\lb u\times (u\times Au),\rv\rb=\lb u\times Au,\rv\times u\rb$ and if $u\in D(A)$, $\rv\in {\mathbb{V}}$, then $\rv \times u\in {\mathbb{V}}$, \eqref{eq:2.91} follows from \eqref{eq:2.9}.
\end{proof}

We consider the following equation in $H_n$ ($H_n\subset D(A)$) with all the assumptions in Assumption \ref{as:all4}:
\begin{equation}\label{eq:dun}
\left\{\begin{array}{ll}
\ud u_n(t)&=-\pi_n\Big\{\lmd_1u_n(t)\times\Big[A u_n(t)+\pi_n \big(\nabla \phi\big(u_n(t)\big) \big)\Big]
\\&\hspace{0.5truecm}\lefteqn{-\lmd_2u_n(t)\times\Big(u_n(t)\times
\Big[A u_n(t) +\pi_n\big( \nabla \phi\big(u_n(t)\big)\big)\Big]\Big)\Big\}\ud t}
\\&+\sum_{j=1}^N\pi_n \Big[u_n(t)\times h_j\Big]\circ\ud W_j(t),\quad t\geq0,\\
u_n(0)&=\pi_nu_0.\end{array}\right.
\end{equation}
Let us point out that \eqref{eq:dun} is a suitable projection of \eqref{eq:7eq} onto the space $H_n$.

Let us define the following maps:
  \begin{eqnarray}
   F^1_n&:&H_n\ni u\longmapsto-\pi_n(u\times A u)\in H_n, \label{eq:F^1_n}\\
   F_n^2&:&H_n\ni u\longmapsto-\pi_n\big(u\times(u\times A u)\big)\in H_n, \label{eq:F^2_n}\\
   F_n^3&:&H_n\ni u\longmapsto- \pi_n\left(u\times\pi_n (\nabla \phi( u))\right)\in H_n, \label{eq:F^3_n}\\
   F_n^4&:&H_n\ni u\longmapsto- \pi_n\Big(u\times\left(u\times\pi_n (\nabla \phi(u))\right)\Big)\in H_n, \label{eq:F^4_n}\\
   \;\;\,G_{jn}&:& H_n\ni u\longmapsto\pi_n(u\times h_j)\in H_n, \; h_j\in \mathbb{L}^\infty\cap \mathbb{W}^{1,3}\label{eq:G_n},\; j=1,\ldots,N.
 \end{eqnarray}

Since $A$ restrict to $H_n$ is linear and bounded (with values in $H_n$) and since $H_n\subset D(A)\subset \mathbb{L}^\infty$, we infer that $G_{jn}$ and $F_n^1,F_n^2,F_n^3,F_n^4$ are well defined maps from $H_n$ to $H_n$.

The problem \eqref{eq:dun} can be written in a more compact way
\begin{equation}\label{eq:dun2}
\left\{\begin{array}{ll}
\ud u_n(t)&= \lmd_1 \left(   F_n^1\big(u_n(t)\big) +   F_n^3 \big(u_n(t)\big)\right)\, \ud t
-\lmd_2 \;\left( F_n^2\big(u_n(t)\big) +   F_n^4 \big(u_n(t)\big)\right)\, \ud t
\\&+\frac12 \sum_{j=1}^N G^2_{jn}\big(u_n(t)\big)\, \ud t + \sum_{j=1}^N G_{jn}\big(u_n(t)\big)\ud W_j(t),\\
u_n(0)&=\pi_nu_0.\end{array}\right.
\end{equation}

\begin{rem}
  In the Equations \eqref{eq:7eq} and \eqref{eq:dun}, we use the Stratonovich differential and in the Equation \eqref{eq:dun2} we use the It\^{o}  differential, the following equality relates the two differentals: for the map $G:{\mathbb{L}^2}\ni u\mapsto u\times h\in {\mathbb{L}^2}$,
  \[(Gu)\circ \ud W(t)=\frac{1}{2}G'(u)[G(u)]\ud t+G(u)\ud W(t),\qquad u\in {\mathbb{L}^2}.\]
\end{rem}

\begin{rem}
 As the equality \eqref{eq:effmag}, we have
 \[-\nabla_{H_n}\mathcal{E}(u_n)=Au_n+\pi_n\nabla \phi(u_n),\]
 so with the ``$\pi_n$''s in the equation \eqref{eq:dun}, our approximation keeps as much as possible the structure of the equation \eqref{eq:7eq}, and consequently we will get the {\em a priori} estimates.
\end{rem}

Now we start to solve the Equation \eqref{eq:dun2}.

\begin{lem}\label{lem:Lips for F3F4}
The maps  $F_n^i$, $i=1,2,3,4$  are Lipschitz on  balls, that is, for every $R>0$  there exists a constant $C=C(n,R)>0$ such that whenever $x,y\in H_n$ and $\|x\|_{\mathbb{L}^2} \leq R$, $\|y\|_{\mathbb{L}^2}\leq R$, we have
  \[\left\|F_n^i(x)-F_n^i(y)\right\|_{\mathbb{L}^2}\leq C\|x-y\|_{\mathbb{L}^2}.\]
  The map  $G_{jn}$ is linear and
 \begin{equation}\label{eq:Gnineq}
   \|G_{jn}u\|_{H_n}\leq \|u\|_{\mathbb{L}^2}\|h_j\|_{\mathbb{L}^\infty}, \qquad u\in H_n.
 \end{equation}
\end{lem}
\begin{proof}
  Let us notice that the maps
   \begin{eqnarray*}
     H_n\ni u&\longmapsto& Au\in H_n\quad\textrm{and}\\
     H_n\ni u&\longmapsto& \pi_n(\nabla \phi(u))\in H_n
   \end{eqnarray*}
   are locally bounded and globally Lipschitz. And if the map $\psi:H_n\longrightarrow H_n$ is locally bounded and locally Lipschitz, then the map
   \[H_n\ni u\longmapsto u\times \psi(u)\in {\mathbb{L}^2}\]
   is also locally bounded and locally Lipschitz. Hence the maps  $F_n^i$, $i=1,2,3,4$  are locally Lipschitz.\\
   The result about $G_{jn}$ is obvious. This completes the proof of Lemma \ref{lem:Lips for F3F4}.
\end{proof}

Since the linear operator $\pi_n:H_n\longrightarrow H_n$ is self-adjoint and by the formula $(a\times b, b)_{\R^3}=0$, we infer that
\begin{lem}\label{lem:1sidelingro}
\[G_{jn}^*=-G_{jn}\]
   Moreover for $i=1,2,3,4$ and $u\in H_n$, we have
  \[\llangle F_n^i(u),u\rrangle_{\mathbb{L}^2}=0.\]
\end{lem}

\begin{cor}\cite{SA&ZB&JW}
  The Equation \eqref{eq:dun} has a unique global solution $u_n:[0,T]\longrightarrow H_n$.
\end{cor}
\begin{proof}
  By the Lemma \ref{lem:Lips for F3F4} and Lemma \ref{lem:1sidelingro}, the coefficients $F_n^i$, $i=1,2,3,4$ and $G_{jn}$ are locally Lipschitz and one side linear growth. Hence by a result in \cite{SA&ZB&JW}, the Equation \eqref{eq:dun} has a unique global solution $u_n:[0,T]\longrightarrow H_n$.
\end{proof}

Let us define functions $F_n$ and $\hat{F}_n:H_n\longrightarrow H_n$ by
\[F_n=\lmd_1(F_n^1+F_n^3)-\lmd_2(F_n^2+F_n^4), \mbox{ and }\hat{F}_n=F_n+\frac{1}{2}\sum_{j=1}^NG_{jn}^2.\]
 Then the problem \eqref{eq:dun} (or \eqref{eq:dun2}) can be written in the following compact way
\begin{equation}\label{eq:dun3}
\ud u_n(t)=\hat{F}_n\big(u_n(t)\big)\ud t+\sum_{j=1}^NG_{jn}\big(u_n(t)\big)\ud W_j(t).
\end{equation}

\section{{\em A priori} estimates}
In this section we will get some properties of the solution of Equation \eqref{eq:dun} especially some {\em a priori} estimates.
\begin{thm}
Assume that $n\in \mathbb{N}$. Let $u_n$ be the solution of the Equation \eqref{eq:dun} which was constructed earlier. Then for every $t\in[0,T]$,
  \begin{equation}\label{eq:3.10phi}
  \|u_n(t)\|_{\mathbb{L}^2}=\|u_n(0)\|_{\mathbb{L}^2},\qquad a.s..
  \end{equation}
\end{thm}
\begin{proof}
Let us consider a function $\psi:H_n\ni u\mapsto\frac{1}{2}\|u\|_H^2\in \R$.
Since $\psi$ is a homogeneous polynomial of degree $2$, $\psi$ is of $C^\infty$. Moreover we have
  \[\psi'(u)(g)=\langle u,g\rangle_{\mathbb{L}^2},\quad\textrm{and}\quad\psi^\bis(u)(g,k)=\langle k,g\rangle_{\mathbb{L}^2}.\]
  By the It\^o Lemma and Lemma \ref{lem:1sidelingro}, we have
  \begin{eqnarray*}
    \frac{1}{2}\ud \|u_n(t)\|_H^2&=&\left(\llangle u_n(t),\hat{F}_n\big(u_n(t)\big)\rrangle_{\mathbb{L}^2}+\frac{1}{2}\sum_{j=1}^N\llangle G_{jn}\big(u_n(t)\big),G_{jn}\big(u_n(t)\big)\rrangle_{\mathbb{L}^2}\right)\ud t\\
    &&+\sum_{j=1}^N\llangle u_n(t),G_{jn}\big(u_n(t)\big)\rrangle_{\mathbb{L}^2}\ud W_j(t)\\
    &=&\frac{1}{2}\sum_{j=1}^N\left(\llangle u_n(t),G_{jn}^2\big(u_n(t)\big)\rrangle_{\mathbb{L}^2}+\sum_{j=1}^N\left\|G_{jn}\big(u_n(t)\big)\right\|_{\mathbb{L}^2}^2\right)\ud t+0\ud W_j(t)\\
    &=&0
  \end{eqnarray*}
  Hence for $t\in[0,T]$,
  \[\|u_n(t)\|_{\mathbb{L}^2}=\|u_n(0)\|_{\mathbb{L}^2},\qquad a.s..\]
\end{proof}

\begin{lem}\label{lem:8.14}
  Let us define a function $\Phi:H_n\longrightarrow\R$ by
   \begin{equation}\label{eq:4Phi}
  \Phi(u):=\frac{1}{2}\int_D\|\nabla u(x)\|^2\ud x+\int_D\phi\big(u(x)\big)\ud x,\qquad u\in H_n.
  \end{equation}
  Then $\Phi\in C^2(H_n)$ and for $u,g,k\in H_n$,
    \begin{eqnarray}\label{eqn-Phi^prime}
  \ud_u\Phi(g)&=&\Phi^{\prime}(u)(g)=\langle \nabla u,\nabla g\rangle_{\Ltwo}+\int_D \lb \nabla \phi\big(u(x)\big), g(x) \rb_{\R^3}\ud x
  \\
  \nonumber
  &=&
  \langle A u, g\rangle_{\mathbb{L}^2}+\int_D \lb \nabla \phi\big(u(x)\big), g(x) \rb_{\R^3}\ud x,
  \\
  \label{eqn-Phi^bis}
  \Phi^\bis(u)(g,k)&=&\langle \nabla g,\nabla k\rangle_{\Ltwo}+\int_D\phi^\bis\big(u(x)\big)\big(g(x),k(x)\big)\ud x.
    \end{eqnarray}
\end{lem}
\begin{proof}
Let us introduce auxiliary real functions $\Phi_0$ and $\Phi_1$ by:
  \[\Phi_0(u):= \int_D\phi\big(u(x)\big)\ud x,\qquad \Phi_1(u):=\frac{1}{2}\|\nabla u\|^2_{\Ltwo}, \qquad u\in H_n.\]
It is enough to prove the results of $\Phi_0$ and $\Phi_1$.\\
The result about $\Phi_0$ is obvious and the result of $\Phi_1$ follows from the mean value Theorem in integral form, see \cite{ZB&KE}, for related result.
\end{proof}

\begin{prop}
  There exist constants $a,b,a_1,b_1>0$ such that for all $n\in\mathbb{N}$,
  \begin{equation}\label{eq:3.19}
     \big\|\nabla G_{jn}u\big\|_{\Ltwo}^2\leq a\big\|\nabla u\big\|_{\Ltwo}^2+b,\quad u\in H_n,
  \end{equation}
  and
  \begin{equation}\label{eq:3.20}
\big\|\nabla G_{jn}^2u\big\|^2_{\Ltwo}\leq a_1\big\|\nabla u\big\|_{\Ltwo}^2+b_1,\quad u\in H_n.
  \end{equation}
\end{prop}
\begin{proof}[Proof of \eqref{eq:3.19}]
Since $A_1$ is self-adjoint and $A_1\geq A$, we have
  \begin{eqnarray*}
    &&\big\|\nabla G_{jn}u\big\|_{\Ltwo}^2=\left(AG_{jn}(u),G_{jn}(u)\right)_{\mathbb{L}^2}\leq (A_1G_{jn}(u),G_{jn}(u))_{\mathbb{L}^2}\\
    &=&\left\|A_1^\frac{1}{2}\pi_n(u\times h_j)\right\|_{\mathbb{L}^2}^2=\left\|\pi_nA_1^\frac{1}{2}(u\times h_j)\right\|_{\mathbb{L}^2}^2\leq \left\|A_1^\frac{1}{2}(u\times h_j)\right\|_{\mathbb{L}^2}^2\\
    &=&\left\|(u\times h_j)\right\|_{\mathbb{V}}^2\leq N\left(\|u\times h_j\|_{\mathbb{L}^2}^2+\|\nabla(u\times h_j)\|_{\Ltwo}^2\right)\\
    &\leq&\left[\|h_j\|_{\mathbb{L}^\infty}^2\left(\|u\|_{\mathbb{L}^2}^2+2\|\nabla u\|_{\Ltwo}^2\right)+2\|\nabla h_j\|_{\mathbb{L}^3}^2\|u\|_{\mathbb{L}^6}^2\right].
  \end{eqnarray*}
  Next since $\mathbb{L}^6\hookrightarrow {\mathbb{V}}$ and by equality \eqref{eq:3.10phi} $\|u_n(s)\|_{\mathbb{L}^2}\leq \|u_0\|_{\mathbb{L}^2}$, we infer that
  \[\big\|\nabla G_{jn}u\big\|_{\Ltwo}^2\leq a\big\|\nabla u\big\|_{\Ltwo}^2+b,\]
  for some constants $a$ and $b$ which only depend on $\|h_j\|_{\mathbb{L}^\infty}$, $\|\nabla h_j\|_{\mathbb{L}^3}$ and $\|u_0\|_{\mathbb{L}^2}$, but not on $n$.
\end{proof}
\begin{proof}[Proof of \eqref{eq:3.20}]
The estimate \eqref{eq:3.20} followed from double application of \eqref{eq:3.19}.
\end{proof}

\begin{rem}
  The previous results will be  used to prove the following fundamental {\em a priori} estimates on the sequence $\{u_n\}$ of the solution of Equation \eqref{eq:dun}.
\end{rem}
\begin{thm}\label{thm:8.9}
Assume that $p\geq 1$, $\beta>\frac{1}{4}$. Then there exists a constant $C>0$, such that for all $n\in\mathbb{N}$,
  \begin{equation}\label{eq:3.11phi}
    \mathbb{E}\sup_{r\in[0,t]}\left\{\big\|\nabla u_n(r)\big\|_{{\mathbb{L}^2}}^{2}+\int_D\phi\big(u_n(r,x)\big)\ud x\right\}^p\leq C,\qquad t\in[0,T],
  \end{equation}
  \begin{equation}\label{eq:3.12phi}
    \mathbb{E}\left[\left(\int_0^T\big\|u_n(t)\times\left(\Delta u_n(t)-\pi_n \nabla \phi\big(u_n(t)\big)\right)\big\|^2_{\mathbb{L}^2}\ud t\right)^p\right]\leq C,
  \end{equation}
  \begin{equation}\label{eq:3.13phi}
    \mathbb{E}\left[\left(\int_0^T\big\|u_n(t)\times\big(u_n(t)\times\left(\Delta u_n(t)-\pi_n \nabla \phi\big(u_n(t)\big)\right)\big)\big\|^2_{\mathbb{L}^\frac{3}{2}(D)}\ud t\right)^p\right]\leq C,
  \end{equation}
  \begin{equation}\label{eq:3.14phi}
    \mathbb{E}\int_0^T\left\|\pi_n\left(u_n(t)\times\big(u_n(t)\times\left(\Delta u_n(t)-\pi_n \nabla \phi\big(u_n(t)\big)\right)\big)\right)\right\|^2_{X^{-\beta}}\ud t\leq C.
  \end{equation}
\end{thm}

\begin{proof}[Proof of \eqref{eq:3.11phi} and \eqref{eq:3.12phi}]
Let us define a function $\Phi$ same as in the Equation \eqref{eq:4Phi}. Then
by the It\^o Lemma,
  \begin{eqnarray}
    &&\varPhi\big(u_n(t)\big)-\varPhi\big(u_n(0)\big)\nonumber\\
    &=&\int_0^t\left(\varPhi'\big(u_n(s)\big)\hat{F}_n\big(u_n(s)\big)+\frac{1}{2}\sum_{j=1}^N\varPhi^\bis\big(u_n(s)\big)G_{jn}\big(u_n(s)\big)^2\right)\ud s\label{eq:Phiunt-0}\\
    &&+\sum_{j=1}^N\int_0^t\varPhi'\big(u_n(s)\big)G_{jn}\big(u_n(s)\big)\ud W_j(s),\qquad t\in[0,T].\nonumber
  \end{eqnarray}

Then we consider each term on the RHS of the Equation \eqref{eq:Phiunt-0}, and we can prove that
\begin{eqnarray}\nonumber
\varPhi'\big(u\big)\hat{F}_n\big(u\big)=&-&
\lmd_2  \big\| u\times\big( \Delta u -\pi_n (\nabla \phi( u) \big)\big\|^2_{\mathbb{L}^2} \\
&-&\frac12 \sum_{j=1}^N \langle \Delta u - \pi_n\big( \nabla \phi( u)\big), \pi_n(u\times h_j) \times h_j \rb_{\mathbb{L}^2}
\label{eqn-Phi'hatF_n}
\end{eqnarray}
\begin{eqnarray}
\varPhi'\big(u\big)[{G}_{jn}\big(u\big)]&=& -\langle \Delta u,  u\times h_j \rb_{\mathbb{L}^2}
+  \lb \nabla \phi( u), \pi_n( u\times h_j ) \rb_{\mathbb{L}^2}\label{eq:Phi'Gn},
\end{eqnarray}
 and
\begin{eqnarray}
\label{eq:Phi''Gn2}&&\hspace{1cm}\varPhi^\bis(u)[G_{jn}(u)^2]\\
&&= \| \nabla \pi_n(u\times h_j) \|^2_{\Ltwo}+\int_D\phi^\bis\big(u(x)\big)\big(\pi_n(u\times h_j)(x),\pi_n(u\times h_j)(x)\big)\ud x.\nonumber
\end{eqnarray}

Therefore by Equations \eqref{eq:4Phi}, \eqref{eqn-Phi'hatF_n}, \eqref{eq:Phi'Gn} and \eqref{eq:Phi''Gn2}, the Equation \eqref{eq:Phiunt-0} transforms to:
  \begin{eqnarray}
    &&\hspace{3cm}\frac{1}{2}\|\nabla u_n(t)\|_{\Ltwo}^2+\frac{1}{2}\int_D\phi(u_n(t,x))\ud x\label{eq:Phiunt-01}\\
    &&+\lmd_2\int_0^t\|u_n(s)\times(\Delta u_n(s)-\pi_n\nabla \phi(u_n(s)))\|_{\mathbb{L}^2}^2\ud s\nonumber\\
    &=&\frac{1}{2}\|\nabla u_n(0)\|_{\Ltwo}^2+\frac{1}{2}\int_D\phi(u_n(0)(x))\ud x-\frac{1}{2}\sum_{j=1}^N\int_0^t\langle \Delta u_n(s),\pi_n(u_n(s)\times h_j)\times h_j\rangle_{\mathbb{L}^2}\ud s\nonumber\\
    &&+\frac{1}{2}\sum_{j=1}^N\int_0^t\langle\nabla \phi(u_n(s)),\pi_n(u_n(s)\times h_j)\times h_j\rangle_{\mathbb{L}^2}\ud s+\frac{1}{2}\sum_{j=1}^N\int_0^t\|\nabla\pi_n(u_n(s)\times h_j)\|_{{\mathbb{L}^2}}^2\ud s\nonumber\\
    &&+\frac{1}{2}\sum_{j=1}^N\int_0^t\int_D\phi^{\prime\prime}(u_n(s)(x))(\pi_n(u_n(s)\times h_j)(x),\pi_n(u_n(s)\times h_j)(x))\ud x\ud s\nonumber\\
    &&-\sum_{j=1}^N\int_0^t\langle\Delta u_n(s),u_n(s)\times h_j\rangle_{\mathbb{L}^2}\ud W_j(s)+\sum_{j=1}^N\int_0^t\langle\nabla \phi(u_n(s),\pi_n(u_n(s)\times h_j)\rangle_{\mathbb{L}^2}\ud W_j(s).\nonumber
  \end{eqnarray}

Next we will get estimates for some terms on the right hand side of Equation \eqref{eq:Phiunt-01}.\\
For the first term on the right hand side of Equation \eqref{eq:Phiunt-01}, we have
\begin{equation}\label{eq:Phiunt-01r1}
  \|\nabla u_n(0)\|_{\Ltwo}^2=\|\nabla\pi_nu_0\|_{\Ltwo}^2\leq\|\pi_nu_0\|_{\mathbb{V}}^2=\|A_1^\frac{1}{2}\pi_nu_0\|_{\mathbb{L}^2}^2=\|\pi_nA_1^\frac{1}{2}u_0\|_{\mathbb{L}^2}^2\leq\|A_1^\frac{1}{2}u_0\|_{\mathbb{L}^2}^2=\|u_0\|_{\mathbb{V}}^2.
\end{equation}

By our assumption, $\phi$ is bounded, so there is a constant $C_\phi>0$, such that for the second term on the right hand side of Equation \eqref{eq:Phiunt-01}, we have
\begin{equation}\label{eq:Phiunt-01r2}
\left|\int_D\phi\big(u_n(0,x)\big)\ud x\right|\leq C_\phi m(D).
\end{equation}

For the third  term on the right hand side of Equation \eqref{eq:Phiunt-01}, by \eqref{eq:3.20} and Cauchy-Schwartz inequality, we have
  \begin{eqnarray}
  &&\hspace{1cm}\left|\langle \Delta u_n(s),\pi_n(u_n(s)\times h_j)\times h_j\rangle_{\mathbb{L}^2}\right|=\left|\langle \nabla u_n(s),\nabla G_n^2u_n(s)\rangle_{\Ltwo}\right|\label{eq:Phiunt-01r3}\\
  &\leq&\|\nabla u_n(s)\|_{\Ltwo}\sqrt{a_1\|\nabla u_n(s)\|_{\Ltwo}^2+b_1}\leq\sqrt{a_1}\|\nabla u_n(s)\|_{\Ltwo}^2+\frac{b_1}{2\sqrt{a_1}}.\nonumber
  \end{eqnarray}

For the fourth term on the right hand side of Equation \eqref{eq:Phiunt-01}, by the equality \eqref{eq:3.10phi} and Cauchy-Schwartz inequality, we have
\begin{equation}\label{eq:Phiunt-01r4}
  \langle\nabla \phi(u_n(s)),\pi_n(u_n(s)\times h_j)\times h_j\rangle_{\mathbb{L}^2}\leq C_{\nabla \phi}m(D)\|u_0\|_{\mathbb{L}^2}\|h_j\|_{\mathbb{L}^\infty}^2.
\end{equation}

For the fifth term on the right hand side of Equation \eqref{eq:Phiunt-01}, by \eqref{eq:3.19}, we have
\begin{equation}\label{eq:Phiunt-01r5}
 \left\|\nabla \pi_n(u_n(s)\times h_j)(s)\right\|^2_{\Ltwo}=\left\|\nabla G_{jn}(u_n(s))\right\|^2_{\Ltwo}\leq a\|\nabla u_n(s)\|^2_{\Ltwo}+b.
\end{equation}

For the sixth term on the right hand side of Equation \eqref{eq:Phiunt-01}, we have
  \begin{eqnarray}
    &&\int_D\left|\phi^\bis\big(u_n(s,x)\big)\left(\pi_n(u_n(s)\times h_j)(x),\pi_n(u_n(s)\times h_j)(x)\right) \right|\ud x\nonumber\\
    &\leq&C_{\phi^\bis}\int_D\left|\pi_n(u_n(s)\times h_j)(x))\right|^{2}\ud x\leq C_{\phi^\bis}\|h_j\|_{\mathbb{L}^\infty}^{2}\|u_0\|_{\mathbb{L}^2}^{2}.\label{eq:Phiunt-01r6}
  \end{eqnarray}

Then by the equalities \eqref{eq:Phiunt-01}-\eqref{eq:Phiunt-01r6}, there exists a constant $C_2>0$ such that for all $n\in \mathbb{N}$, $t\in [0,T]$ and $\mathbb{P}$-almost surely:
  \begin{eqnarray}
    &&\left\|\nabla u_n(t)\right\|^2_{\Ltwo}+\int_D\phi\big(u_n(t,x)\big)\ud x+2\lmd_2\int_0^t\|u_n(s)\times(\Delta u_n(s)-\pi_n\nabla \phi(u_n(s)))\|^2_{\mathbb{L}^2}\ud s\nonumber\\
    &&\hspace{0.5truecm}\lefteqn{\leq C_2\int_0^t\|\nabla u_n(s)\|^2_{\Ltwo}\ud s+C_2+2\sum_{j=1}^N\int_0^t\left\langle\nabla u_n(s),\nabla G_{jn}\big(u_n(s)\big)\right\rangle_{\Ltwo}\ud W_j(s)}\nonumber\\
    &&\hspace{2cm}+\sum_{j=1}^N\int_0^t\llangle \nabla \phi\big(u_n(s)\big),G_{jn}\big(u_n(s)\big)\rrangle_{\mathbb{L}^2}\ud W_j(s).\label{eq:20.5}
  \end{eqnarray}

  Hence for $p\geq 1$,
  \begin{eqnarray*}
    &&\mathbb{E}\sup_{r\in[0,t]}\left\{\left\|\nabla u_n(r)\right\|^2_{\Ltwo}+\int_D\phi\big(u_n(r,x)\big)\ud x+2\lmd_2\int_0^r\|u_n(s)\times(\Delta u_n(s)-\pi_n\nabla \phi(u_n(s)))\|^2_{\mathbb{L}^2}\ud s\right\}^p\\
    &\leq&4^{p-1}C_2^pt^{p-1}\mathbb{E}\left(\int_0^t\|\nabla u_n(s)\|^{2p}_{\Ltwo}\ud s\right)\\
    &&+4^{p-1}2\mathbb{E}\sup_{r\in[0,t]}\left|\sum_{j=1}^N\int_0^r\llangle \nabla u_n(s),\nabla G_{jn}\big(u_n(s)\big)\rrangle_{\Ltwo}\ud W_s\right|^p\\
    &&+4^{p-1}\mathbb{E}\sup_{r\in[0,t]}\left|\sum_{j=1}^N\int_0^r\llangle \nabla \phi\big( u_n(s)\big), G_{jn}\big(u_n(s)\big)\rrangle_{\mathbb{L}^2}\ud W_s\right|^p+4^{p-1}C_2^p.
  \end{eqnarray*}

By the Burkholder-Davis-Gundy inequality, there exists a constant $K>0$ such that for all $n\in \mathbb{N}$,
\[\mathbb{E}\sup_{r\in[0,t]}\left|\sum_{j=1}^N\int_0^r\llangle \nabla u_n(s),\nabla G_{jn}\big(u_n(s)\big)\rrangle_{\Ltwo}\ud W_s\right|^p\leq K\mathbb{E}\left|\sum_{j=1}^N\int_0^t\llangle \nabla u_n(s),\nabla G_{jn}\big(u_n(s)\big)\rrangle_{\Ltwo}^2\ud s\right|^\frac{p}{2},\]
\[\mathbb{E}\sup_{r\in[0,t]}\left|\sum_{j=1}^N\int_0^r\llangle \nabla \phi\big( u_n(s)\big), G_{jn}\big(u_n(s)\big)\rrangle_{\mathbb{L}^2}\ud W_s\right|^p\leq K\mathbb{E}\left|\sum_{j=1}^N\int_0^t\llangle \nabla \phi\big( u_n(s)\big), G_{jn}\big(u_n(s)\big)\rrangle_{\mathbb{L}^2}^2\ud s\right|^\frac{p}{2}.\]

By the  inequality \eqref{eq:3.19} we get, for any $\eps>0$,
  \begin{eqnarray*}
    &&\mathbb{E}\left|\sum_{j=1}^N\int_0^t\llangle \nabla u_n(s),\nabla G_{jn}\big(u_n(s)\big)\rrangle_{\Ltwo}^2\ud s\right|^\frac{p}{2}\leq \mathbb{E}\left[\sup_{r\in [0,t]}\|\nabla u_n(r)\|_{\Ltwo}^p\left(\sum_{j=1}^N\int_0^t\|\nabla G_{jn}(u_n(s))\|^2_{\Ltwo}\ud s\right)^\frac{p}{2}\right]\\
    &\leq&\mathbb{E}\left[\eps\sup_{r\in[0,t]}\|\nabla u_n(r)\|_{\Ltwo}^{2p}+\frac{4}{\eps}\left(\sum_{j=1}^N\int_0^t\|\nabla G_{jn}(u_n(s))\|_{\Ltwo}^2\ud s\right)^p\right]\\
    &\leq&\eps\mathbb{E}\left(\sup_{r\in[0,t]}\|\nabla u_n(r)\|_{\Ltwo}^{2p}\right)+\frac{4}{\eps}(2t)^{p-1}a^pN^p\mathbb{E}\left(\int_0^t\|\nabla u_n(s)\|_{\Ltwo}^{2p}\ud s\right)+\frac{4}{\eps}2^{p-1}(bt)^pN^p.
  \end{eqnarray*}

And

  \begin{eqnarray*}
    &&\hspace{-1truecm}\lefteqn{\mathbb{E}\left|\sum_{j=1}^N\int_0^t\llangle\nabla \phi\big(u_n(s)\big),G_{jn}\big(u_n(s)\big)\rrangle_{\Ltwo}^2\ud s\right|^\frac{p}{2}}
    \\
    &\leq& \mathbb{E}\left[\sup_{r\in [0,t]}\|\nabla \phi(u_n(r))\|_{\Ltwo}^p\left(\sum_{j=1}^N\int_0^t\|\nabla G_{jn}(u_n(s))\|^2_{\Ltwo}\ud s\right)^\frac{p}{2}\right]\\
    &\leq&\eps\left[C_{\nabla \phi}m(D)\right]^{2p}+\frac{4}{\eps}(2t)^{p-1}a^pN^p\mathbb{E}\left(\int_0^t\|\nabla u_n(s)\|_{\Ltwo}^{2p}\ud s\right)+\frac{4}{\eps}2^{p-1}(bt)^pN^p.
  \end{eqnarray*}

Hence we infer that for $t\in[0,T]$,
  \begin{eqnarray}
    &&\mathbb{E}\sup_{r\in[0,t]}\left|\sum_{j=1}^N\int_0^r\llangle \nabla u_n(s),\nabla G_{jn}\big(u_n(s)\big)\rrangle_{\Ltwo}\ud W_s\right|^p\nonumber\\
    &\leq&K\eps\mathbb{E}\left(\sup_{r\in[0,t]}\|\nabla u_n(r)\|_{\Ltwo}^{2p}\right)+\frac{4K}{\eps}(2t)^{p-1}a^pN^p\mathbb{E}\left(\int_0^t\|\nabla u_n(s)\|_{\Ltwo}^{2p}\ud s\right)\label{eq:20.7}\\
    &&+\frac{4K}{\eps}2^{p-1}(bt)^pN^p.\nonumber
  \end{eqnarray}

and similarly for $t\in[0,T]$,
  \begin{eqnarray}
  &&\hspace{0.5cm}\mathbb{E}\sup_{r\in[0,t]}\left|\sum_{j=1}^N\int_0^r\llangle \nabla \phi\big( u_n(s)\big), G_{jn}\big(u_n(s)\big)\rrangle_{\Ltwo}\ud W_s\right|^p\label{eq:20.8}\\
  &\leq& K\eps\left[C_{\nabla \phi}m(D)\right]^{2p}+\frac{4K}{\eps}(2t)^{p-1}a^pN^p\mathbb{E}\left(\int_0^t\|\nabla u_n(s)\|_{\Ltwo}^{2p}\ud s\right)+\frac{4K}{\eps}2^{p-1}(bt)^pN^p.\nonumber
\end{eqnarray}

Hence for every $t\in [0,T]$,
  \begin{eqnarray}
    &&\mathbb{E}\sup_{r\in[0,t]}\Bigg\{\left\|\nabla u_n(r)\right\|^2_{\Ltwo}+\int_D\phi\big(u_n(r,x)\big)\ud x+2\lmd_2\int_0^r\|u_n(s)\times(\Delta u_n(s)-\pi_n\nabla \phi(u_n(s)))\|^2_{\mathbb{L}^2}\ud s\Bigg\}^p\nonumber\\
    &\leq&4^{p-1}C_2^pt^{p-1}\mathbb{E}\left(\int_0^t\|\nabla u_n(s)\|^{2p}_{\Ltwo}\ud s\right)+4^{p-1}K\eps\mathbb{E}\left(\sup_{r\in[0,t]}\|\nabla u_n(r)\|_{\Ltwo}^{2p}\right)+4^{p-1}K\eps\left[C_{\nabla \phi}m(D)\right]^{2p}\nonumber\\
    &&+\frac{8K}{\eps}(8t)^{p-1}a^pN^p\mathbb{E}\left(\int_0^t\|\nabla u_n(s)\|_{\Ltwo}^{2p}\ud s\right)+\frac{K}{\eps}8^{p}(bt)^pN^p\nonumber
    \end{eqnarray}

Set $\eps=\frac{1}{2K4^{p-1}}$ in the above inequality, we have:
  \begin{eqnarray}
    &&\hspace{-1.5cm}\mathbb{E}\sup_{r\in[0,t]}\Bigg\{\left\|\nabla u_n(r)\right\|^2_{\Ltwo}+\int_D\phi\big(u_n(r,x)\big)\ud x\nonumber\\
    &&+2\lmd_2\int_0^r\|u_n(s)\times(\Delta u_n(s)-\pi_n\nabla \phi(u_n(s)))\|^2_{\mathbb{L}^2}\ud s\Bigg\}^p\nonumber\\
    &\leq&\left[2\cdot4^{p-1}C_2^pt^{p-1}+32K^2(8t)^{p-1}a^pN^p\right]\mathbb{E}\left(\int_0^t\|\nabla u_n(s)\|^{2p}_{\Ltwo}\ud s\right)\label{eq:20.6}\\
    &&+\left[C_{\nabla \phi}m(D)\right]^{2p}+4K^28^{p}(bt)^pN^p\nonumber\\
    &=&C_3\mathbb{E}\left(\int_0^t\|\nabla u_n(s)\|^{2p}_{\Ltwo}\ud s\right)+C_4\nonumber.
\end{eqnarray}
where the constants $C_3$ and $C_4$ are defined by:
\[C_3=2\cdot4^{p-1}C_2^pt^{p-1}+32K^2(8t)^{p-1}a^pN^p,\]
\[C_4=\left[C_{\nabla \phi}m(D)\right]^{2p}+4K^28^{p}(bt)^pN^p,\]
note that they do not depend on $n$.

And since $\int_D\phi\big(u_n(r,x)\big)\ud x$ and $\lmd_2\int_0^t|u_n(s)\times(\Delta u_n(s)-\pi_n\nabla \phi(u_n(s)))|^2_{\mathbb{L}^2}\ud s$ are nonnegative, so by the inequality \eqref{eq:20.6}, we have
  \begin{eqnarray}
    &&\hspace{-1.5cm}\mathbb{E}\sup_{r\in[0,t]}\Bigg\{\left\|\nabla u_n(r)\right\|^2_{\Ltwo}+\int_D\phi\big(u_n(r,x)\big)\ud x\nonumber\\
    &&+2\lmd_2\int_0^t\|u_n(s)\times(\Delta u_n(s)-\pi_n\nabla \phi(u_n(s)))\|^2_{\mathbb{L}^2}\ud s\Bigg\}^p\nonumber\\
    &\leq&C_3\int_0^t\mathbb{E}\sup_{r\in[0,s]}\Bigg\{\|\nabla u_n(r)\|_{\Ltwo}^2+\int_D\phi\big(u_n(r,x)\big)\ud x\label{eq:20.61}\\
    &&+2\lmd_2\int_0^r\|u_n(\tau)\times(\Delta u_n(\tau)-\pi_n\nabla \phi(u_n(\tau)))\|^2_{\mathbb{L}^2}\ud \tau\Bigg\}^p\ud s+C_4.\nonumber
  \end{eqnarray}

Let us define a function $\psi_n$ by:
  \begin{eqnarray*}
  \psi_n(s)&=&\mathbb{E}\sup_{r\in[0,s]}\Bigg\{\big\|\nabla u_n(r)\big\|_{\Ltwo}^{2}+\int_D\phi\big(u_n(r,x)\big)\ud x\\
  &&+2\lmd_2\int_0^r\|u_n(\tau)\times(\Delta u_n(\tau)-\pi_n\nabla \phi(u_n(\tau)))\|^2_{\mathbb{L}^2}\ud \tau\Bigg\}^p,\qquad s\in [0,T].
  \end{eqnarray*}

  Then by the inequality \eqref{eq:20.61}, we deduce that:
  \[\psi_n(t)\leq C_3\int_0^t\psi_n(s)\ud s+C_4.\]

Observe that $\psi_n$ is a bounded Borel function. The boundedness is because
\[\|\nabla u_n(r)\|_{\Ltwo}\leq \|u_n(r)\|_{\mathbb{V}}\leq C_n\|u_n(r)\|_{\mathbb{L}^2}\leq C_n\|u_0\|_{\mathbb{L}^2},\quad r\in[0,T],\]
and
  \begin{eqnarray*}
    &&\hspace{-1.5cm}\left\|u_n(s)\times (\Delta u_n(s)-\pi_n\nabla \phi(u_n(s)))\right\|_{H_n}\\
    &\leq&\|u_n(s)\|_{\mathbb{L}^\infty}\left(\|\Delta u_n(s)\|_{\mathbb{L}^2}+\|\pi_n\nabla \phi(u_n(s))\|_{\mathbb{L}^2}\right)\\
    &\leq&C_n\|u_n(s)\|_{\mathbb{L}^2}\left(C_n\|u_n(s)\|_{\mathbb{L}^2}+C_{\nabla \phi}m(D)^\frac{1}{2}\right)\\
    &\leq&C_n\|u_0\|_{\mathbb{L}^2}\left(C_n\|u_0\|_{\mathbb{L}^2}+C_{\nabla \phi}m(D)^\frac{1}{2}\right).
  \end{eqnarray*}
where $C_n$ is from the norm equivalence in the $n$-dimensional space. Therefore
\[\|\psi_n(s)\|\leq \left(C_n^2\|u_0\|^2_{\mathbb{L}^2}+C_\phi m(D)+2\lmd_2TC_n^2\|u_0\|^2_{\mathbb{L}^2}\left(C_n\|u_0\|_{\mathbb{L}^2}+2C_{\nabla \phi}m(D)^\frac{1}{2}\right)^2\right)^p.\]
Therefore by the Gronwall inequality, we have
  \[\psi_n(t)\leq C_3e^{C_4t},\qquad t\in[0,T].\]
  Since $C_3$ and $C_4$ are independent of $n$, we have proved that for  $T\in(0,\infty)$,
  \begin{eqnarray*}
  &&\mathbb{E}\sup_{r\in[0,t]}\left\{\big\|\nabla u_n(r)\big\|_{\Ltwo}^{2}+\int_D\phi\big(u_n(r,x)\big)\ud x+2\lmd_2\int_0^r\|u_n(\tau)\times(\Delta u_n(\tau)-\pi_n\nabla \phi(u_n(\tau)))\|^2_{\mathbb{L}^2}\ud \tau\right\}^p\\
  &&\hspace{8.5cm}\leq C_3e^{C_4T}=C_T,\qquad t\in[0,T]
  \end{eqnarray*}
  where $C_T$ is independent of $n$.
  Therefore we infer that
  \[\mathbb{E}\sup_{r\in[0,t]}\left\{\big\|\nabla u_n(r)\big\|_{\Ltwo}^{2}+\int_D\phi\big(u_n(r,x)\big)\ud x\right\}^p\leq C_T,\]
  and
  \[\mathbb{E}\left(\int_0^T\|u_n(\tau)\times(\Delta u_n(\tau)-\pi_n\nabla \phi(u_n(\tau)))\|^2_{\mathbb{L}^2}\ud \tau\right)^p\leq C_T.\]
  This completes the proof of the inequalities \eqref{eq:3.11phi} and \eqref{eq:3.12phi}.
\end{proof}

\begin{proof}[Proof of \eqref{eq:3.13phi}]
By the H$\ddot{\textrm{o}}$lder inequality and the Sobolev imbedding $\mathbb{V}\hookrightarrow \mathbb{L}^6$, we have that for some constant $c>0$
  \begin{eqnarray*}
    \big\|u_n(t)\times\big(u_n(t)\times\left(\Delta u_n(t)-\pi_n\nabla \phi\big(u_n(t)\big)\right)\big)\big\|_{\mathbb{L}^\frac{3}{2}}&\leq&\big\|u_n(t)\big\|_{\mathbb{L}^6}\big\|u_n(t)\times\left(\Delta u_n(t)-\pi_n\nabla \phi\big(u_n(t)\big)\right)\big\|_{{\mathbb{L}^2}}\\
    &\leq& c\big\|u_n(t)\big\|_{{\mathbb{V}}}\big\|u_n(t)\times\left(\Delta u_n(t)-\pi_n\nabla \phi\big(u_n(t)\big)\right)\big\|_{{\mathbb{L}^2}}.
  \end{eqnarray*}
  Then by \eqref{eq:3.10phi},
  \eqref{eq:3.11phi} and \eqref{eq:3.12phi}, we have:
  {
  \begin{eqnarray*}
    &&\mathbb{E}\left[\left(\int_0^T\left\|u_n(t)\times\big(u_n(t)\times \left(\Delta u_n(t)-\pi_n\nabla \phi\big(u_n(t)\big)\right)\big)\right\|^2_{\mathbb{L}^\frac{3}{2}}\ud t\right)^p\right]\\
    &\leq&c\mathbb{E}\left[\sup_{r\in[0,T]}\big\|u_n(r)\big\|_{{\mathbb{V}}}^{2p}\left(\int_0^T\big\|u_n(t)\times \left(\Delta u_n(t)-\pi_n\nabla \phi\big(u_n(t)\big)\right)\big\|_{{\mathbb{L}^2}}^2\ud t\right)^p\right]\\
    &\leq&c\left(\mathbb{E}\left[\sup_{r\in[0,T]}\big\|u_n(r)\big\|_{{\mathbb{V}}}^{4p}\right]\right)^\frac{1}{2} \left(\mathbb{E}\left[\left(\int_0^T\big\|u_n(t)\times \left(\Delta u_n(t)-\pi_n\nabla \phi\big(u_n(t)\big)\right)\big\|_{{\mathbb{L}^2}}^2\ud t\right)^{2p}\right]\right)^\frac{1}{2}\leq C,
  \end{eqnarray*}}
  Note that $C$ is independent of $n$. This completes the proof of \eqref{eq:3.13phi}.
\end{proof}

\begin{proof}[Proof of \eqref{eq:3.14phi}]
By Sobolev imbedding theorems,  if $\beta>\frac{1}{4}$, $X^\beta\hookrightarrow \mathbb{H}^{2\beta}(D)$ continuously. And if $\beta>\frac{1}{4}$, $\mathbb{H}^{2\beta}(D)$ is continuously imbedded in $\mathbb{L}^3$. Therefore $\mathbb{L}^\frac{3}{2}(D)$ is continuously imbedded in $X^{-\beta}$. And since for $\xi\in {\mathbb{L}^2}$,
{
  \begin{eqnarray*}
    &&\|\pi_n\xi\|_{X^{-\beta}}=\sup_{\|\varphi\|_{X^\beta}\leq 1}\left|{}_{X^{-\beta}}\langle \pi_n\xi,\varphi\rangle_{X^\beta}\right| =\sup_{\|\varphi\|_{X^\beta}\leq 1}\left|\langle\pi_n\xi,\varphi\rangle_{\mathbb{L}^2}\right|\\
    &=&\sup_{\|\varphi\|_{X^\beta}\leq 1}\left|\langle\xi,\pi_n\varphi\rangle_{\mathbb{L}^2}\right|\leq\sup_{\|\pi_n\varphi\|_{X^\beta}\leq 1}\left|{}_{X^{-\beta}}\langle \xi,\pi_n\varphi\rangle_{X^\beta}\right|=\|\xi\|_{X^{-\beta}}.
  \end{eqnarray*}}
Therefore we infer that there exists some constant $c>0$ such that
{
  \begin{eqnarray*}
  &&\mathbb{E}\int_0^T\big\|\pi_n\big(u_n(t)\times\big(u_n(t)\times\left(\Delta u_n(t)-\pi_n\nabla \phi\big(u_n(t)\big)\right)\big)\big)\big\|^2_{X^{-\beta}}\ud t\\
  &\leq&c\mathbb{E}\int_0^T\big\|u_n(t)\times \big(u_n(t)\times\left(\Delta u_n(t)-\pi_n\nabla \phi\big(u_n(t)\big)\right)\big)\big\|_{\mathbb{L}^\frac{3}{2}}^2\ud t.
  \end{eqnarray*}}
Then \eqref{eq:3.14phi} follows from \eqref{eq:3.13phi}.
\end{proof}

\begin{prop}
Let $u_n$, for $n\in\mathbb{N}$, be the solution of the equation \eqref{eq:dun} and assume that $\alpha\in(0,\frac{1}{2})$, $\beta>\frac{1}{4}$, $p\geq2$. Then the following estimates holds:
  \begin{equation}\label{eq:4.2}
  \sup_{n\in\mathbb{N}}\mathbb{E}\big(\|u_n\|^2_{W^{\alpha,p}(0,T;X^{-\beta})}\big)<\infty.
\end{equation}
\end{prop}

We need the following Lemma to prove \eqref{eq:4.2}.
\begin{lem}[\cite{Flandoli}, Lem 2.1]\label{lem:LLG A.1}
 Assume that $E$ is a separable Hilbert space, $p\in[2,\infty)$ and $a\in (0,\frac{1}{2}).$ Then there exists a constant $C$ depending on $T$ and $a$, such
 that for any progressively measurable process $\xi=(\xi_j)_{j=1}^\infty$,
 \[\mathbb{E}\left\|I(\xi)\right\|^p_{W^{a,p}(0,T;E)}\leq C\mathbb{E}\int_0^T\left(\sum_{j=1}^\infty|\xi_j(r)|_E^2\right)^\frac{p}{2}\ud t,\]
  where $I(\xi_j)$ is defined by
  \[I(\xi):=\sum_{j=1}^\infty\int_0^t\xi_j(s)\ud W_j(s),\quad t\geq 0.\]
   In particular, $\mathbb{P}$-a.s. the trajectories of the process $I(\xi_j)$ belong to  $W^{a,2}(0,T;E)$.
\end{lem}

\begin{proof}[Proof of \eqref{eq:4.2}]
 Let us fix $\alpha\in(0,\frac{1}{2})$, $\beta>\frac{1}{4}$, $p\geq2$. By the equation \eqref{eq:dun2}, we get
 {
  \begin{eqnarray*}
   u_n(t)&=&u_{0,n}+ \lmd_1 \int_0^t\left(   F_n^1\big(u_n(s)\big) +   F_n^3 \big(u_n(s)\big)\right)\, \ud s
-\lmd_2 \int_0^t\left( F_n^2\big(u_n(s)\big) +   F_n^4 \big(u_n(s)\big)\right)\, \ud s\\
&&+\frac12\sum_{j=1}^N\int_0^t G^2_{jn}\big(u_n(s)\big)\, \ud s +\sum_{j=1}^N\int_0^tG_{jn}\big(u_n(s)\big)\ud W(s)\\
&=:&u_{0,n}+\sum_{i=1}^4u_n^i(t), \qquad t\in[0,T].
  \end{eqnarray*}}
 By Theorem \ref{thm:8.9}, we have the following results:\\
  There exists $C>0$ such that for all $n\in\mathbb{N}$,
\begin{trivlist}
  \item[(1)]
  \[\mathbb{E}\left[\|u_n^1\|^2_{W^{1,2}(0,T;{\mathbb{L}^2})}\right]\leq C.\]
  \item[(2)]
  \[\mathbb{E}\left[\|u_n^2\|^2_{W^{1,2}(0,T;X^{-\beta})}\right]\leq C.\]
  \item[(3)]
  \[\|u_n^3\|^2_{W^{1,2}(0,T;{\mathbb{L}^2})}\leq C,\qquad\mathbb{P}-a.s..\]
\end{trivlist}
Moreover, by the equality \eqref{eq:3.10phi},
\[\mathbb{E}\left[\sup_{t\in[0,T]}\|u_n(t)\|_{\mathbb{L}^2}^p\right]=\mathbb{E}\left[\|u_n(0)\|_{\mathbb{L}^2}^p\right]\leq C.\]
By the inequality \eqref{eq:Gnineq} and Lemma \ref{lem:LLG A.1}, we have:
\[\mathbb{E}\left[\|u_n^4\|^p_{W^{\alpha,p}(0,T;X^{-\beta})}\right]\leq C.\]
Therefore since $H^1(0,T;X^{-\beta})\hookrightarrow W^{\alpha,p}(0,T;X^{-\beta})$ continuously, we get
\[\sup_{n\in\mathbb{N}}\mathbb{E}\big(\|u_n\|^2_{W^{\alpha,p}(0,T;X^{-\beta})}\big)<\infty.\]
This completes the proof of the inequality \eqref{eq:4.2}.
\end{proof}

\section{Tightness results}
In this subsection we will use the {\em a priori} estimates \eqref{eq:3.10phi}-\eqref{eq:3.14phi} to show that the laws $\{\mathcal{L}(u_n):n\in\mathbb{N}\}$ are tight on a suitable path space.
\begin{lem}\label{lem:4.2}
  For any  $p\geq 2$, $q\in [2,6)$ and $\beta>\frac{1}{4}$ the set of laws $\{\mathcal{L}(u_n):n\in\mathbb{N}\}$ on the Banach space
  \[L^p(0,T;\mathbb{L}^q(D))\cap C(0,T; X^{-\beta})\]
  is tight.
\end{lem}
\begin{proof}
Let us choose and fix $p\geq 2$, $q\in [2,6)$ and $\beta>\frac{1}{4}$. Let us then choose auxiliary numbers   $\beta'\in (\frac{1}{4},\beta)$ and  $\alpha\in (\frac{1}{p},1)$. Since $q<6$ we can choose $\gamma\in (\frac{3}{4}-\frac{3}{2q},\frac{1}{2})$ such that $\mathbb{H}^{2\gamma}\hookrightarrow \mathbb{L}^q$ continuously.
Since, the operator $A_{1}^{-1}$ is compact in $\mathbb{L}^2$,   the embedding ${\mathbb{V}}=D(A_{1}^{\frac{1}{2}})\hookrightarrow X^\gamma=D(A_{1}^\gamma)$ is compact and thus also  the embedding
\[L^p(0,T;{\mathbb{V}})\cap W^{\alpha,p}(0,T;X^{-\beta'})\hookrightarrow L^p(0,T;X^\gamma)\]
is compact.
We note that for any positive real number $r$ and random variables $\xi$ and $\eta$, since
\[\left\{\omega:\xi(\omega)>\frac{r}{2}\right\}\cup\left\{\omega:\eta(\omega)>\frac{r}{2}\right\}\supset \Big\{\omega:\xi(\omega)+\eta(\omega)>r\Big\},\]
we have
{
  \begin{eqnarray*}
&&\hspace{-2truecm}\lefteqn{\mathbb{P}\Big(\|u_n\|_{L^p(0,T;H^1)\cap W^{\alpha,p}(0,T;X^{-\beta'})}>r\Big)}\\
&=&\mathbb{P}\Big(\|u_n\|_{L^p(0,T;H^1)}+\|u_n\|_{W^{\alpha,p}(0,T;X^{-\beta'})}>r\Big)\\
&\leq& \mathbb{P}\left(\|u_n\|_{L^p(0,T;H^1)}>\frac{r}{2}\right)+\mathbb{P} \left(\|u_n\|_{W^{\alpha,p}(0,T;X^{-\beta'})}>\frac{r}{2}\right)
\leq\ldots
\end{eqnarray*}}
then by the Chebyshev inequality,
\[\ldots\leq\frac{4}{r^2}\mathbb{E}\left(\|u_n\|^2_{L^p(0,T;{\mathbb{V}})}+\|u_n\|^2_{W^{\alpha,p}(0,T;X^{-\beta'})}\right).\]
By the estimates in \eqref{eq:4.2}, \eqref{eq:3.10phi} and \eqref{eq:3.11phi}, the expected value on the right hand side of the last inequality is uniformly bounded in $n$. Let $X_T:=L^p(0,T;{\mathbb{V}})\cap W^{\alpha,p}(0,T;X^{-\beta'})$.
There is a constant $C$, such that
\[\mathbb{P}\big(\|u_n\|_{X_T}>r\big)\leq\frac{C}{r^2},\qquad\forall r,n.\]
Since
\[\mathbb{E}(\|u_n\|_X)=\int_0^\infty\mathbb{P}(\|M_n\|>r)\ud r,\]
we can infer that
\[\mathbb{E}\big(\|u_n\|_{X_T}\big)\leq 1+\int_1^\infty \frac{C}{r^2}\ud r=1+C<\infty,\quad\forall n\in\mathbb{N}.\]
Therefore the family of laws $\big\{\mathcal{L}(u_n):n\in\mathbb{N}\big\}$ is tight on $L^p(0,T;X^\gamma)$.
 Since by Proposition \ref{prop:DAH}, $X^\gamma=\mathbb{H}^{2\gamma}$, as at the beginning of the proof,  $X^\gamma\hookrightarrow \mathbb{L}^q(D)$ continuously.
 Hence we infer that  $L^p(0,T;X^\gamma)\hookrightarrow L^p(0,T;\mathbb{L}^q(D))$ continuously and  therefore $\big\{\mathcal{L}(u_n):n\in\mathbb{N}\big\}$ is also tight on $L^p(0,T;\mathbb{L}^q(D))$.\\
Since $\beta'<\beta$,  $W^{\alpha,p}(0,T;X^{-\beta'})\hookrightarrow C(0,T;X^{-\beta})$ compactly. Therefore by estimates  \eqref{eq:4.2}, we can conclude that $\big\{\mathcal{L}(u_n):n\in\mathbb{N}\big\}$ is tight on $C(0,T;X^{-\beta})$.\\
Therefore  $\big\{\mathcal{L}(u_n):n\in\mathbb{N}\big\}$ is tight on $L^p(0,T;L^q)\cap C([0,T];X^{-\beta}).$
Hence the proof of Lemma \ref{lem:4.2} is completed.
\end{proof}

From now on we will always assume $\beta>\frac{1}{4}$.

\section{Construction of new Probability Space and Processes}
In this section we will use Skorohod's theorem to obtain another probability space and an almost surely convergent sequence defined on this space whose limit is a weak martingale solution of the equation \eqref{eq:7eq}.\\
By  Lemma \ref{lem:4.2} and Prokhorov's Theorem, we have the following property.
\begin{prop}\label{prop:conv to mu0}
  Let us assume that $W$ is a $N$-dimensional Wiener process and $p\in [2,\infty),$ $q\in[2,6)$ and $\beta>\frac{1}{4}$.
  Then there is a subsequence of $\{u_n\}$ which we will denote it in the same way as the full sequence, such that the laws $\mathcal{L}(u_n,W)$ converge weakly to a certain probability measure $\mu$ on $L^p(0,T;\mathbb{L}^q(D))\cap C([0,T];X^{-\beta})\times C(0,T;\R^N)$.
\end{prop}

Now by the Skorohod's theorem we have:
\begin{prop}\label{prop:u_n'phi}
  There exists a probability space $(\Omega^{\prime},\mathcal{F}^{\prime},\mathbb{P}^{\prime})$ and there exists a sequence $(u_n^{\prime},W_{n}^{\prime})$ of $[L^4(0,T;\mathbb{L}^4)\cap C([0,T];X^{-\beta})]\times C([0,T];\R^N)$-valued random variables defined on $(\Omega^{\prime},\mathcal{F}^{\prime},\mathbb{P}^{\prime})$ such that
  \begin{trivlist}
    \item[(a)] On $[L^4(0,T;\mathbb{L}^4)\cap C([0,T];X^{-\beta})]\times C([0,T];\R^N)$,
    $$\mathcal{L}(u_n,W)=\mathcal{L}(u_n^{\prime},W_{n}^{\prime}),\qquad\forall n\in\mathbb{N}$$
    \item[(b)] There exists a random variable
    \[(u^{\prime},W^{\prime}):(\Omega^{\prime},\mathcal{F}^{\prime},\mathbb{P}^{\prime})\longrightarrow [L^4(0,T;\mathbb{L}^4)\cap C(0,T;X^{-\beta})]\times C(0,T;\R^N)\]
    such that
    \begin{trivlist}
    \item[(i)] On $[L^4(0,T;\mathbb{L}^4)\cap C([0,T];X^{-\beta})]\times C([0,T];\R^N)$,
    \[\mathcal{L}(u^{\prime},W^{\prime})=\mu,\]
    where $\mu$ is same as in Proposition \ref{prop:conv to mu0}.
    \item[(ii)] $u_n^{\prime}\longrightarrow u^{\prime}$ in $L^4(0,T;\mathbb{L}^4)\cap C([0,T];X^{-\beta})$ almost surely,
    \item[(iii)] $W_{n}^{\prime}\longrightarrow W^{\prime}$ in $C([0,T];\R^N)$ almost surely.
    \end{trivlist}
  \end{trivlist}
\end{prop}

\begin{notation}
  We will use $\mathbb{F}^{\prime}$ to denote the filtration generated by $u^{\prime}$ and $W^{\prime}$ in the probability space $(\Omega^{\prime},\mathcal{F}^{\prime},\mathbb{P}^{\prime})$.
\end{notation}

From now on we will prove that $u^{\prime}$ is the weak solution of the equation \eqref{eq:7eq}. And we begin with showing that $\{u_n^{\prime}\}$ satisfies the same {\em a priori} estimates as the original sequence $\{u_n\}$. By the Kuratowski Theorem, we have
\begin{prop}\label{prop:sabor}
  The Borel subsets of $C(0,T;H_n)$ are Borel subsets of $L^4(0,T;\mathbb{L}^4)\cap C(0,T;X^{-\frac{1}{2}})$.
\end{prop}
So we have the following Corollary.
\begin{cor}\label{cor:4.4phi}
  $u_n^{\prime}$ takes values in $H_n$ and the laws on $C([0,T];H_n)$ of $u_n$ and $u_n^{\prime}$ are equal.
\end{cor}
By the Corollary \ref{cor:4.4phi}, we have
\begin{lem}\label{lem:estun'}
  The $\{u_n^{\prime}\}$ defined in Proposition \ref{prop:u_n'phi} satisfies the following estimates:
  \begin{equation}\label{eq:4.5phi}
    \sup_{t\in[0,T]}\big\|u_n^{\prime}(t)\big\|_{\mathbb{L}^2}\leq\big\|u_0\big\|_{\mathbb{L}^2},\qquad\mathbb{P}^{\prime}-a.s.,
  \end{equation}
  \begin{equation}\label{eq:4.6phi}
    \sup_{n\in\mathbb{N}}\mathbb{E}^{\prime}\left[\sup_{t\in[0,T]}\big\|u_n^{\prime}(t)\big\|_{\mathbb{V}}^{2r}\right]<\infty, \qquad \forall r\geq 1,
  \end{equation}
  \begin{equation}\label{eq:4.7phi}
    \sup_{n\in\mathbb{N}}\mathbb{E}^{\prime}\left[\left(\int_0^T\|u_n^{\prime}(t)\times\left(\Delta u_n^{\prime}(t)-\pi_n\nabla \phi\big(u_n^{\prime}(t)\big)\right)\|_{\mathbb{L}^2}^2\ud t\right)^r\right]<\infty,\qquad\forall r\geq 1,
  \end{equation}
  \begin{equation}\label{eq:4.8phi}
    \sup_{n\in\mathbb{N}}\mathbb{E}^{\prime}\int_0^T\left\|u_n^{\prime}(t)\times\big(u_n^{\prime}(t)\times\left(\Delta u_n^{\prime}(t)-\pi_n\nabla \phi\big(u_n^{\prime}(t)\big)\right)\big)\right\|^2_{\mathbb{L}^\frac{3}{2}}\ud t<\infty,
  \end{equation}
  \begin{equation}\label{eq:4.9phi}
     \sup_{n\in\mathbb{N}}\mathbb{E}^{\prime}\int_0^T\left\|\pi_n\left[u_n^{\prime}(t)\times\big(u_n^{\prime}(t)\times\left(\Delta u_n^{\prime}(t)-\pi_n\nabla \phi\big(u_n^{\prime}(t)\big)\right)\big)\right]\right\|^2_{X^{-\beta}}\ud t<\infty.
  \end{equation}
\end{lem}

Now we will study some inequalities satisfied by the limiting process $u^{\prime}$.
\begin{prop}
  Let $u^{\prime}$ be the process which is defined in Proposition \ref{prop:u_n'phi}. Then we have
  \begin{equation}\label{eq:4.10phi}
    \esup_{t\in[0,T]}\|u^{\prime}(t)\|_{\mathbb{L}^2}\leq \|u_0\|_{\mathbb{L}^2},\quad \mathbb{P}^{\prime}-a.s.
  \end{equation}
  \begin{equation}\label{eq:4.11phi}
    \sup_{t\in[0,T]}\|u^{\prime}(t)\|_{X^{-\beta}}\leq c\|u_0\|_{\mathbb{L}^2},\quad\mathbb{P}^{\prime}-a.s.
  \end{equation}
\end{prop}
\begin{proof}[Proof of \eqref{eq:4.10phi}]
  Since $u_n^{\prime}$ converges to $u^{\prime}$ in $L^4(0,T;\mathbb{L}^4)\cap C(0,T;X^{-\beta})$ $\mathbb{P}^{\prime}$ almost surely,
  \[\lim_{n\rightarrow\infty}\int_0^T\|u_n^{\prime}(t)-u^{\prime}(t)\|^4_{\mathbb{L}^4}\ud t=0,\quad\mathbb{P}^{\prime}-a.s.\]
  Since $\mathbb{L}^4\hookrightarrow {\mathbb{L}^2}$, we infer that
  \[\lim_{n\rightarrow\infty}\int_0^T\|u_n^{\prime}(t)-u^{\prime}(t)\|^2_{\mathbb{L}^2}\ud t=0.\]
  Hence $u_n^{\prime}$ converges to $u^{\prime}$  in $L^2(0,T;\mathbb{L}^2)$ $\mathbb{P}^{\prime}$ almost surely. Therefore by \eqref{eq:4.5phi},
  \[\esup_{t\in[0,T]}\|u^{\prime}(t)\|_{\mathbb{L}^2}\leq \|u_0\|_{\mathbb{L}^2},\quad \mathbb{P}^{\prime}-a.s.\]
\end{proof}
\begin{proof}[Proof of \eqref{eq:4.11phi}]
  Since ${\mathbb{L}^2}\hookrightarrow X^{-\beta}$,  there exists some constant $c>0$, such that $\|u_n^{\prime}(t)\|_{X^{-\beta}}\leq c\|u_n^{\prime}(t)\|_{\mathbb{L}^2}$ for all $n\in\mathbb{N}$. By \eqref{eq:4.5phi}, we have
  \[\sup_{t\in[0,T]}\|u_n^{\prime}(t)\|_{X^{-\beta}}\leq c\sup_{t\in[0,T]}\|u_n^{\prime}(t)\|_{\mathbb{L}^2}\leq c\|u_0\|_{\mathbb{L}^2},\quad\mathbb{P}^{\prime}-a.s.\]
  And by Proposition \ref{prop:u_n'phi} (ii) $u_n^{\prime}$ converges to $u^{\prime}$ in $C([0,T];X^{-\beta})$, we infer that
  \[\sup_{t\in[0,T]}\|u^{\prime}(t)\|_{X^{-\beta}}\leq c\|u_0\|_{\mathbb{L}^2},\quad\mathbb{P}^{\prime}-a.s.\]
\end{proof}

We continue investigating properties of the process $u^{\prime}$. The next result and it's proof are related to the estimate \eqref{eq:4.6phi}.
\begin{prop}
  Let $u^{\prime}$ be the process which was defined in Proposition \ref{prop:u_n'phi}. Then we have
  \begin{equation}\label{eq:4.12phi}
    \mathbb{E}^{\prime}[\esup_{t\in[0,T]}\|u^{\prime}(t)\|^{2r}_{{\mathbb{V}}}]<\infty,\quad r\geq 2.
  \end{equation}
\end{prop}
\begin{proof}

Since $L^{2r}(\Omega^{\prime};L^\infty(0,T;{\mathbb{V}}))$ is isomorphic to $\left(L^\frac{2r}{2r-1}(\Omega^{\prime};L^1(0,T;X^{-\frac{1}{2}}))\right)^\ast$, by the Banach-Alaoglu Theorem we infer that the sequence $\{u_n^{\prime}\}$ contains a subsequence, denoted in the same way as the full sequence, and there exists an element
$v \in L^{2r}(\Omega^{\prime};L^\infty(0,T;{\mathbb{V}}))$ such that $u_n^{\prime} \to v$ weakly$^\ast$ in $L^{2r}(\Omega^{\prime};L^\infty(0,T;{\mathbb{V}}))$. In particular,  we have
\[
\lb u_n^{\prime} , \varphi \rb \to \lb v,  \varphi \rb, \qquad  \varphi \in L^\frac{2r}{2r-1}(\Omega^{\prime};(L^1(0,T;X^{-\frac{1}{2}}))).
\]
This means that
\[ \int_{\Omega^{\prime}} \int_0^T \lb u_n^{\prime}(t,\omega) , \varphi(t,\omega) \rb \ud t \ud\mathbb{P}^{\prime}(\omega) \to
\int_{\Omega^{\prime}} \int_0^T \lb v(t,\omega) , \phi(t,\omega) \rb \ud t\ud\mathbb{P}^{\prime}(\omega).
\]

On the other hand, if we fix $\varphi \in L^4(\Omega^{\prime};L^\frac{4}{3}(0,T;\mathbb{L}^\frac{4}{3}))$, by the inequality \eqref{eq:4.6phi} we have (to avoid too long formulations, we omit some parameters $t$ in the following equations)
 {
 \begin{eqnarray*}
   &&\sup_n\int_{\Omega^{\prime}}\left|\int_0^T{}_{\mathbb{L}^4}\langle u_n^{\prime}(t),\varphi(t)\rangle_{\mathbb{L}^\frac{4}{3}}\ud t\right|^2\ud \mathbb{P}^{\prime}(\omega)\leq\sup_n\int_{\Omega^{\prime}}\left|\int_0^T\|u_n^{\prime}\|_{\mathbb{L}^4}\|\varphi\|_{\mathbb{L}^\frac{4}{3}}\ud t\right|^2\ud \mathbb{P}^{\prime}(\omega)\\
   &\leq&\sup_n\int_{\Omega^{\prime}}\|u_n^{\prime}\|_{L^\infty(0,T;\mathbb{L}^4)}^2\|\varphi\|_{L^1(0,T;\mathbb{L}^\frac{4}{3})}^2\ud \mathbb{P}^{\prime}(\omega)\leq \sup_n\|u_n^{\prime}\|^2_{L^4(\Omega^{\prime};L^\infty(0,T;\mathbb{L}^4))}\|\varphi\|^2_{L^4(\Omega^{\prime};L^1(0,T;\mathbb{L}^\frac{4}{3}))}<\infty.
 \end{eqnarray*}}
So the sequence
$\int_0^T{}_{\mathbb{L}^4}\langle u_n^{\prime}(t),\varphi(t)\rangle_{\mathbb{L}^\frac{4}{3}}\ud t$ is uniformly integrable on $\Omega^{\prime}$. Moreover, by the $\mathbb{P}^{\prime}$ almost surely convergence of $u_n^{\prime}$ to $u^{\prime}$ in $L^4(0,T;\mathbb{L}^4)$, we get $\mathbb{P}^{\prime}$-a.s.
  {
 \begin{eqnarray*}
   &&\left|\int_0^T{}_{\mathbb{L}^4}\langle u_n^{\prime}(t),\varphi(t)\rangle_{\mathbb{L}^\frac{4}{3}}\ud t-\int_0^T{}_{\mathbb{L}^4}\langle u^{\prime}(t),\varphi(t)\rangle_{\mathbb{L}^\frac{4}{3}}\ud t\right|\\
   &\leq&\int_0^T\left|{}_{\mathbb{L}^4}\langle u_n^{\prime}(t)-u^{\prime}(t),\varphi(t)\rangle_{\mathbb{L}^\frac{4}{3}}\right|\ud t\leq\int_0^T\|u_n^{\prime}(t)-u^{\prime}(t)\|_{\mathbb{L}^4}\|\varphi(t)\|_{\mathbb{L}^\frac{4}{3}}\ud t\\
   &\leq&\|u_n^{\prime}-u^{\prime}\|_{L^4(0,T;\mathbb{L}^4)}\|\varphi\|_{L^\frac{4}{3}(0,T;\mathbb{L}^\frac{4}{3})}\rightarrow 0.
 \end{eqnarray*}}

Therefore we infer that $\int_0^T{}_{\mathbb{L}^4}\langle u_n^{\prime}(t),\varphi(t)\rangle_{\mathbb{L}^\frac{4}{3}}\ud t$ converges to $\int_0^T{}_{\mathbb{L}^4}\langle u^{\prime}(t),\varphi(t)\rangle_{\mathbb{L}^\frac{4}{3}}\ud t$ $\mathbb{P}^{\prime}$ almost surely. Thus,
\[ \int_{\Omega^{\prime}} \int_0^T{}_{\mathbb{L}^4}\langle u_n^{\prime}(t,\omega),\varphi(t,\omega)\rangle_{\mathbb{L}^\frac{4}{3}}\ud t \ud\mathbb{P}^{\prime}(\omega) \to
\int_{\Omega^{\prime}} \int_0^T{}_{\mathbb{L}^4}\langle u^{\prime}(t,\omega),\varphi(t,\omega)\rangle_{\mathbb{L}^\frac{4}{3}}\ud t \ud\mathbb{P}^{\prime}(\omega).
\]
Hence we deduce that
\[
\int_{\Omega^{\prime}} \int_0^T{}_{\mathbb{L}^4}\langle v(t,\omega),\varphi(t,\omega)\rangle_{\mathbb{L}^\frac{4}{3}}\ud t\ud\mathbb{P}^{\prime}(\omega)
=\int_{\Omega^{\prime}} \int_0^T{}_{\mathbb{L}^4}\langle u^{\prime}(t,\omega),\varphi(t,\omega)\rangle_{\mathbb{L}^\frac{4}{3}}\ud t\ud\mathbb{P}^{\prime}(\omega)
\]

By the arbitrariness of $\varphi$ and density of $L^4(\Omega^{\prime};L^\frac{4}{3}(0,T;\mathbb{L}^\frac{4}{3}))$ in $L^\frac{2r}{2r-1}(\Omega^{\prime};L^1(0,T;X^{-\frac{1}{2}}))$, we infer that $u^{\prime}=v$ and since $v$ satisfies \eqref{eq:4.12phi} we infer that $u^{\prime}$ also satisfies \eqref{eq:4.12phi}. In this way the proof of \eqref{eq:4.12phi} is complete.
\end{proof}

Now we will strengthen part (ii) of Proposition \ref{prop:u_n'phi} about the convergence of $u_n^{\prime}$ to $u^{\prime}$.
\begin{prop}
  \begin{equation}\label{eq:4.13phi}
  \lim_{n\rightarrow\infty}\mathbb{E}^{\prime}\int_0^T\|u_n^{\prime}(t)-u^{\prime}(t)\|_{\mathbb{L}^4}^4\ud t=0.
  \end{equation}
\end{prop}
\begin{proof}
Since $u_n^{\prime}\longrightarrow u^{\prime}$ in $L^4(0,T;\mathbb{L}^4)\cap C(0,T;X^{-\beta})$ $\mathbb{P}^{\prime}$-almost surely, $u_n^{\prime}\longrightarrow u^{\prime}$ in $L^4(0,T;\mathbb{L}^4)$ $\mathbb{P}^{\prime}$-almost surely, i.e.
\[\lim_{n\rightarrow\infty}\int_0^T\|u_n^{\prime}(t)-u^{\prime}(t)\|_{\mathbb{L}^4}^4\ud t=0,\qquad \mathbb{P}^{\prime}-a.s.,\]
and by \eqref{eq:4.6phi} and \eqref{eq:4.12phi},
\[\sup_n\mathbb{E}^{\prime}\left(\int_0^T\|u_n^{\prime}(t)-u^{\prime}(t)\|_{\mathbb{L}^4}^4\ud t\right)^2\leq 2^7\sup_n\left(\|u_n^{\prime}\|^8_{L^4(0,T;\mathbb{L}^4)}+\|u^{\prime}\|^8_{L^4(0,T;\mathbb{L}^4)}\right)<\infty,\]
Hence we infer that
\[\lim_{n\rightarrow\infty}\mathbb{E}^{\prime}\int_0^T\|u_n^{\prime}(t)-u^{\prime}(t)\|_{\mathbb{L}^4}^4\ud t=0.\]
This completes the proof.
\end{proof}

By the estimate \eqref{eq:4.6phi}, $\{u_n^{\prime}\}_{n=1}^\infty$ is bounded in $L^2(\Omega^{\prime};L^2(0,T;\mathbb{H}^1))$. And since $u_n^{\prime}\longrightarrow u^{\prime}$ in $L^2(\Omega^{\prime};L^2(0,T;\mathbb{L}^2))$, we have:
\begin{equation}\label{eq:4.17phi}
  D_iu_n^{\prime}\longrightarrow D_iu^{\prime}\textrm{ weakly in }L^2(\Omega^{\prime};L^2(0,T;\mathbb{L}^2)),\;i=1,2,3.
\end{equation}

\begin{lem}\label{lem:uxDeltauinH}
  There exists a unique $\Lambda\in L^2(\Omega^{\prime};L^2(0,T;{\mathbb{L}^2}))$ such that
  \begin{equation}\label{eq:uxDeltauinH}
  \mathbb{E}^{\prime}\int_0^T\langle \Lambda(t), v(t)\rangle_{\mathbb{L}^2}\ud t=\sum_{i=1}^3\mathbb{E}^{\prime}\int_0^T\langle D_iu^{\prime}(t),u^{\prime}(t)\times D_iv(t)\rangle_{\mathbb{L}^2}\ud t
  \end{equation}
  for every  $v\in L^2(\Omega^{\prime};L^2(0,T;\mathbb{W}^{1,4}))$.
\end{lem}
\begin{proof}
  We will omit``$(t)$'' in this proof. Let us denote $\Lambda_n:=u_n^{\prime}\times Au_n^{\prime}$. By the estimate \eqref{eq:4.7phi}, there exists a constant $C$ such that
  \[\|\Lambda_n\|_{L^2(\Omega^{\prime};L^2(0,T;{\mathbb{L}^2}))}\leq C,\quad n\in\mathbb{N}.\]
  Hence by the Banach-Alaoglu Theorem, there exists $\Lambda\in L^2(\Omega^{\prime};L^2(0,T;{\mathbb{L}^2}))$ such that $\Lambda_n\rightarrow \Lambda$ weakly in $ L^2(\Omega^{\prime};L^2(0,T;{\mathbb{L}^2}))$.\\
   Let us fix $v\in L^2(\Omega^{\prime};L^2(0,T;\mathbb{W}^{1,4}))$. Since $u_n^{\prime}(t)\in D(A)$ for almost every $t\in[0,T]$ and $\mathbb{P}^{\prime}$-almost surely, by the Proposition \ref{prop:uxAuweak} and estimate \eqref{eq:4.7phi} again, we have
  \[\mathbb{E}^{\prime}\int_0^T\langle \Lambda_n,v\rangle_{\mathbb{L}^2}\ud t=\sum_{i=1}^3\mathbb{E}^{\prime}\int_0^T\langle D_iu_n^{\prime},u_n^{\prime}\times D_iv\rangle_{\mathbb{L}^2}\ud t.\]
  Moreover, by the results: \eqref{eq:4.17phi}, \eqref{eq:4.6phi} and \eqref{eq:4.13phi}, we have for $i=1,2,3$,
  {
  \begin{eqnarray*}
    &&\hspace{-0.5truecm}\lefteqn{\left|\mathbb{E}^{\prime}\int_0^T\langle D_iu^{\prime},u^{\prime}\times D_iv\rangle_{\mathbb{L}^2}\ud t-\mathbb{E}^{\prime}\int_0^T\langle D_iu_n^{\prime},u_n^{\prime}\times D_iv\rangle_{\mathbb{L}^2}\ud t\right|}\\
    &\leq&\left|\mathbb{E}^{\prime}\int_0^T\langle D_iu^{\prime}-D_iu_n^{\prime},u^{\prime}\times D_iv\rangle_{\mathbb{L}^2}\ud t\right|+\left|\mathbb{E}^{\prime}\int_0^T\langle D_iu_n^{\prime},(u^{\prime}-u_n^{\prime})\times D_iv\rangle_{\mathbb{L}^2}\ud t\right|\\
    &\leq&\left|\mathbb{E}^{\prime}\int_0^T\langle D_iu^{\prime}-D_iu_n^{\prime},u^{\prime}\times D_iv\rangle_{\mathbb{L}^2}\ud t\right|+\Big(\mathbb{E}^{\prime}\int_0^T\|D_iu_n^{\prime}\|_{\mathbb{L}^2}^2\ud t\Big)^\frac{1}{2}\\
    &&\hspace{+1.5truecm}\lefteqn{\times \Big(\mathbb{E}^{\prime}\int_0^T\|u^{\prime}-u_n^{\prime}\|_{\mathbb{L}^4}^4\ud t\Big)^\frac{1}{4}\left(\mathbb{E}^{\prime}\int_0^T\|D_iv\|_{\mathbb{L}^4}^4\ud t\right)^\frac{1}{4}\rightarrow0.}
  \end{eqnarray*}}
  Therefore we infer that
  \[\lim_{n\rightarrow\infty}\mathbb{E}'\int_0^T\langle \Lambda_n,v\rangle_{\mathbb{L}^2}\ud t=\sum_{i=1}^3\mathbb{E}'\int_0^T\langle D_iu^{\prime},u^{\prime}\times D_iv\rangle\ud t.\]
  Since on the other hand we have proved $\Lambda_n\rightarrow \Lambda$ weakly in $L^2(\Omega^{\prime};L^2(0,T;{\mathbb{L}^2}))$ the equality \eqref{eq:uxDeltauinH} follows. \\
  It remains to prove the uniqueness of $\Lambda$, but this follows from the fact that\\ $L^2(\Omega^{\prime};L^2(0,T;\mathbb{W}^{1,4}))$ is dense in $ L^2(\Omega^{\prime};L^2(0,T;{\mathbb{L}^2}))$ and \eqref{eq:uxDeltauinH}. This complete the proof of Lemma \ref{lem:uxDeltauinH}.
\end{proof}

\begin{notation}\label{notation:uxDeltau}
   The process $\Lambda$ introduced in  Lemma \ref{lem:uxDeltauinH} will be denoted by $u^{\prime}\times \Delta u^{\prime}$ (as explained in the Appendix). Note that $u^{\prime}\times \Delta u^{\prime}$ is an  element of $L^2(\Omega^{\prime};L^2(0,T;{\mathbb{L}^2}))$ such that
   for all test functions $v \in L^2(\Omega^{\prime};L^2(0,T;\mathbb{W}^{1,4}))$ the following identity holds
  \[\mathbb{E}^{\prime}\int_0^T\langle (u^{\prime}\times\Delta u^{\prime})(t), v(t)\rangle_{\mathbb{L}^2}\ud t=\sum_{i=1}^3\mathbb{E}^{\prime}\int_0^T\langle D_iu^{\prime}(t),u^{\prime}(t)\times D_iv(t)\rangle_{\mathbb{L}^2}\ud t.\]
\end{notation}

\begin{notation}\label{notation:uxuxDeltau}
Since by the estimate \eqref{eq:4.12phi}, $u'\in L^2(\Omega',L^\infty(0,T;{\mathbb{V}}))$ and by Notation \ref{notation:uxDeltau}, $\Lambda\in L^2(\Omega^{\prime};L^2(0,T;{\mathbb{L}^2}))$,
   the process $u'\times \Lambda\in L^\frac{4}{3}(\Omega';L^2(0,T;\mathbb{L}^\frac{3}{2}(D)))$. And $u'\times \Lambda$ will be denoted by $u^{\prime}\times (u^{\prime}\times \Delta u^{\prime})$.
\end{notation}

\begin{notation}
  $\Lambda-u'\times\nabla \phi\big(u^{\prime}\big)$ will be denoted by $u^{\prime}\times\left(\Delta u^{\prime}-\nabla \phi\big(u^{\prime}\big)\right)$.
\end{notation}

 Next we will show that the limits of the following three sequences

  \begin{eqnarray*}
  &&\big\{u_n^{\prime}\times\left(\Delta u_n^{\prime}-\pi_n\nabla \phi\big(u_n^{\prime}\big)\right)\big\}_n,\\
   &&\big\{u_n^{\prime}\times(u_n^{\prime}\times\left(\Delta u_n^{\prime}-\pi_n\nabla \phi\big(u_n^{\prime}\big)\right))\big\}_n,\\
    &&\big\{\pi_n\big(u_n^{\prime}\times(u_n^{\prime}\times\left(\Delta u_n^{\prime}-\pi_n\nabla \phi\big(u_n^{\prime}\big)\right))\big)\big\}_n,
     \end{eqnarray*}
  exist and
  are equal respectively to

\begin{eqnarray*}
  &&u^{\prime}\times\left(\Delta u^{\prime}-\nabla \phi\big(u^{\prime}\big)\right),\\
  &&u^{\prime}\times(u^{\prime}\times\left(\Delta u^{\prime}-\nabla \phi\big(u^{\prime}\big)\right)),\\
   &&u^{\prime}\times(u^{\prime}\times\left(\Delta u^{\prime}-\nabla \phi\big(u^{\prime}\big)\right)).
  \end{eqnarray*}

By \eqref{eq:4.7phi}-\eqref{eq:4.9phi}, the first sequence  is bounded in $L^{2r}(\Omega^{\prime};L^2(0,T;\mathbb{L}^2))$ for $r\geq1$, the second sequence is bounded in $L^2(\Omega^{\prime};L^2(0,T;\mathbb{L}^\frac{3}{2}))$ and the third sequence  is bounded in $L^2(\Omega^{\prime};L^2(0,T;X^{-\beta}))$. And since the Banach spaces $L^{2r}(\Omega^{\prime};L^2(0,T;\mathbb{L}^2))$, $L^2(\Omega^{\prime};L^2(0,T;\mathbb{L}^\frac{3}{2}))$ and $L^2(\Omega^{\prime};L^2(0,T;X^{-\beta}))$ are all reflexive, by the Banach-Alaoglu Theorem, there exist subsequences weakly convergent. So we can assume that there exist
  \begin{eqnarray*}
  &&Y\in L^{2r}(\Omega^{\prime};L^2(0,T;\mathbb{L}^2)),\\
  &&Z\in L^2(\Omega^{\prime};L^2(0,T;\mathbb{L}^\frac{3}{2})), \\
  &&Z_1\in L^2(\Omega^{\prime};L^2(0,T;X^{-\beta})),
  \end{eqnarray*}

     such that
 \begin{equation}\label{eq:4.14phi}
   u_n^{\prime}\times\left(\Delta u_n^{\prime}-\pi_n\nabla \phi\big(u_n^{\prime}\big)\right)\longrightarrow Y\quad\textrm{weakly in } L^{2r}(\Omega^{\prime};L^2(0,T;\mathbb{L}^2)),
 \end{equation}
 \begin{equation}\label{eq:4.15phi}
   u_n^{\prime}\times\left(u_n^{\prime}\times\left(\Delta u_n^{\prime}-\pi_n\nabla \phi\big(u_n^{\prime}\big)\right)\right)\longrightarrow Z\quad\textrm{weakly in }L^2(\Omega^{\prime};L^2(0,T;\mathbb{L}^\frac{3}{2})),
 \end{equation}
 \begin{equation}\label{eq:4.16phi}
   \pi_n\left(u_n^{\prime}\times\left(u_n^{\prime}\times\left(\Delta u_n^{\prime}-\pi_n\nabla \phi\big(u_n^{\prime}\big)\right)\right)\right)\longrightarrow Z_1\quad\textrm{weakly in } L^2(\Omega^{\prime};L^2(0,T;X^{-\beta})).
 \end{equation}

\begin{rem*}
  Similar argument has been done in \cite{ZB&TJ} for terms not involving $\nabla \phi$. Our main contribution here is to show the validity of such an argument for term containing $\nabla \phi$ (and to be more precise). This works because earlier we have been able to prove generalized estimates as in \cite{ZB&TJ} as in Lemma \ref{lem:estun'}.
\end{rem*}

\begin{prop}\label{prop:Z=Z1}
  If $Z$ and $Z_1$ defined as above, then $Z=Z_1\in L^2(\Omega^{\prime};L^2(0,T;X^{-\beta}))$.
\end{prop}
\begin{proof}
  Notice that $(L^\frac{3}{2})^*=L^3$, and by Proposition \ref{prop:DAH}, $X^\beta=\mathbb{H}^{2\beta}$. By $X^\beta\subset L^3$ for $\beta>\frac{1}{4}$, we deduce that $L^\frac{3}{2}\subset X^{-\beta}$.
  So $$L^2(\Omega^{\prime};L^2(0,T;\mathbb{L}^\frac{3}{2}))\subset L^2(\Omega^{\prime};L^2(0,T;X^{-\beta})).$$
  Therefore $Z\in L^2(\Omega^{\prime};L^2(0,T;X^{-\beta}))$ and also $Z_1\in L^2(\Omega^{\prime};L^2(0,T;X^{-\beta}))$.\\
  Since $X^{\beta}=D(A_1^\beta)$ and $A_1$ is self-adjoint, we can define
  \[X^{\beta}_n=\left\{\pi_nx=\sum_{j=1}^nx_je_j:\sum_{j=1}^\infty\lmd_j^{2\beta}x_j^2<\infty\right\}.\]
  Then $X^{\beta}=\bigcup_{n=1}^\infty X_n^{\beta}$ and $L^2(\Omega^{\prime};L^2(0,T;X^{\beta}))=\bigcup_{n=1}^\infty L^2(\Omega^{\prime};L^2(0,T;X^{\beta}_n))$.
  We have for $\psi_n\in L^2(\Omega^{\prime};L^2(0,T;X^{\beta}_n))$,
  {
  \begin{eqnarray*}
  &&_{L^2(\Omega^{\prime};L^2(0,T;X^{-\beta}))}\langle \pi_n(u_n^{\prime}\times(u_n^{\prime}\times\left(\Delta u_n^{\prime}-\pi_n\nabla \phi\big(u_n^{\prime}\big)\right))),\psi_n\rangle_{L^2(\Omega^{\prime};L^2(0,T;X^{\beta}))}\\
  &=&\mathbb{E}^{\prime}\int_0^T \,_{X^{-\beta}}\langle \pi_n(u_n^{\prime}(t)\times(u_n^{\prime}(t)\times\left(\Delta u_n^{\prime}(t)-\pi_n\nabla \phi\big(u_n^{\prime}(t)\big)\right))),\psi_n(t)\rangle_{X^\beta}\ud t\\
  &=&\mathbb{E}^{\prime}\int_0^T \,_{{\mathbb{L}^2}}\langle \pi_n(u_n^{\prime}(t)\times(u_n^{\prime}(t)\times\left(\Delta u_n^{\prime}(t)-\pi_n\nabla \phi\big(u_n^{\prime}(t)\big)\right))),\psi_n(t)\rangle_{{\mathbb{L}^2}}\ud t\\
  &=&\mathbb{E}^{\prime}\int_0^T \,_{{\mathbb{L}^2}}\langle u_n^{\prime}(t)\times(u_n^{\prime}(t)\times\left(\Delta u_n^{\prime}(t)-\pi_n\nabla \phi\big(u_n^{\prime}(t)\big)\right)),\psi_n(t)\rangle_{{\mathbb{L}^2}}\ud t\\
  &=&\mathbb{E}^{\prime}\int_0^T \,_{X^{-\beta}}\langle u_n^{\prime}(t)\times(u_n^{\prime}(t)\times\left(\Delta u_n^{\prime}(t)-\pi_n\nabla \phi\big(u_n^{\prime}(t)\big)\right)),\psi_n(t)\rangle_{X^\beta}\ud t\\
  &=&_{L^2(\Omega^{\prime};L^2(0,T;X^{-\beta}))}\langle u_n^{\prime}\times(u_n^{\prime}\times\left(\Delta u_n^{\prime}(t)-\pi_n\nabla \phi\big(u_n^{\prime}(t)\big)\right)),\psi_n\rangle_{L^2(\Omega^{\prime};L^2(0,T;X^{\beta}))}.
  \end{eqnarray*}}
  Hence
  \[_{L^2(\Omega^{\prime};L^2(0,T;X^{-\beta}))}\langle Z_1,\psi_n\rangle_{L^2(\Omega^{\prime};L^2(0,T;X^{\beta}))}=_{L^2(\Omega^{\prime};L^2(0,T;X^{-\beta}))}\langle Z,\psi_n\rangle_{L^2(\Omega^{\prime};L^2(0,T;X^{\beta}))},\]
  $\forall\psi_n\in L^2(\Omega^{\prime};L^2(0,T;X^{\beta}_n))$.
  For any $\psi\in L^2(\Omega^{\prime};L^2(0,T;X^{\beta})$, there exists $L^2(\Omega^{\prime};L^2(0,T;X^{\beta}_n))\ni \psi_n\longrightarrow\psi$ as $n\longrightarrow\infty$, hence for all $\psi\in L^2(\Omega^{\prime};L^2(0,T;X^{\beta})$,
  {
  \begin{eqnarray*}
    _{L^2(\Omega^{\prime};L^2(0,T;X^{-\beta}))}\langle Z_1,\psi\rangle_{L^2(\Omega^{\prime};L^2(0,T;X^{\beta}))}&=& \lim_{n\rightarrow\infty}\,_{L^2(\Omega^{\prime};L^2(0,T;X^{-\beta}))}\langle Z_1,\psi_n\rangle_{L^2(\Omega^{\prime};L^2(0,T;X^{\beta}))}\\
    &=&\lim_{n\rightarrow\infty}\,_{L^2(\Omega^{\prime};L^2(0,T;X^{-\beta}))}\langle Z,\psi_n\rangle_{L^2(\Omega^{\prime};L^2(0,T;X^{\beta}))}\\
    &=&_{L^2(\Omega^{\prime};L^2(0,T;X^{-\beta}))}\langle Z,\psi\rangle_{L^2(\Omega^{\prime};L^2(0,T;X^{\beta}))}
  \end{eqnarray*}}
  Therefore $Z=Z_1\in L^2(\Omega^{\prime};L^2(0,T;X^{-\beta}))$ and this concludes the proof of Proposition \ref{prop:Z=Z1}.
\end{proof}

\begin{lem}\label{lem:4.5phi}
  For any measurable process $\psi\in L^4(\Omega^{\prime};L^4(0,T;\mathbb{W}^{1,4}))$, we have the equality
  {
  \begin{eqnarray*}
    &&\lim_{n\rightarrow\infty}\mathbb{E}^{\prime}\int_0^T\langle u_n^{\prime}(t)\times\left(\Delta u_n^{\prime}-\pi_n\nabla \phi\big(u_n^{\prime}(t)\big)\right),\psi(t)\rangle_{\Ltwo}\ud t\\
    &=&\mathbb{E}^{\prime}\int_0^T\langle Y(t),\psi(t)\rangle_{\Ltwo}\ud t\\
    &=&\mathbb{E}^{\prime}\int_0^T\sum_{i=1}^3\langle\frac{\partial u^{\prime}(t)}{\partial x_i},u^{\prime}(t)\times\frac{\partial\psi(t)}{\partial x_i}\rangle_{\mathbb{L}^2}\ud s+\mathbb{E}^{\prime}\int_0^T\left\langle u^{\prime}(t)\times \nabla \phi\big(u^{\prime}(t)\big),\psi\right\rangle_{\mathbb{L}^2}\ud t.
  \end{eqnarray*}}
\end{lem}
\begin{proof}
 Let us fix $\psi\in L^4(\Omega^{\prime};L^4(0,T;\mathbb{W}^{1,4}))$.
 Firstly, we will prove that
 \[\lim_{n\rightarrow\infty}\mathbb{E}^{\prime}\int_0^T\langle u_n^{\prime}(t)\times\Delta u_n^{\prime}(t),\psi(t)\rangle_{\mathbb{L}^2}\ud t=\mathbb{E}^{\prime}\int_0^T\sum_{i=1}^3\left\langle\frac{ \partial u^{\prime}(t)}{\partial x_i},u^{\prime}(t)\times\frac{\partial \psi(t)}{\partial x_i}\right\rangle_{\mathbb{L}^2}\ud t.\]
  For each $n\in \mathbb{N}$ we have
  \begin{equation}\label{eq:4.19phi}
    \langle u_n^{\prime}(t)\times \Delta u_n^{\prime}(t),\psi\rangle_{\mathbb{L}^2}=\sum_{i=1}^3\left\langle\frac{\partial u_n^{\prime}(t)}{\partial x_i},u_n^{\prime}(t)\times \frac{\partial \psi(t)}{\partial x_i}\right\rangle_{\mathbb{L}^2}
  \end{equation}
  for almost every $t\in [0,T]$ and $\mathbb{P}^{\prime}$ almost surely. By Corollary \ref{cor:4.4phi}, $\mathbb{P}(u_n^{\prime}\in C(0,T;H_n))=1$. For each $i\in\{1,2,3\}$ we may write
  {
  \begin{eqnarray}
    &&\left\langle\frac{\partial u_n^{\prime}(t)}{\partial x_i},u_n^{\prime}(t)\times\frac{\partial \psi(t)}{\partial x_i}\right\rangle_{\mathbb{L}^2}-\left\langle\frac{\partial u^{\prime}(t)}{\partial x_i},u^{\prime}(t)\times\frac{\partial \psi(t)}{\partial x_i}\right\rangle_{\mathbb{L}^2}\label{eq:4.20phi}\\
    &=&\left\langle\frac{\partial u_n^{\prime}(t)}{\partial x_i}-\frac{\partial u^{\prime}(t)}{\partial x_i},u^{\prime}(t)\times\frac{\partial \psi(t)}{\partial x_i}\right\rangle_{\mathbb{L}^2}+\left\langle\frac{\partial u_n^{\prime}(t)}{\partial x_i},(u^{\prime}_n(t)-u^{\prime}(t))\times\frac{\partial \psi(t)}{\partial x_i}\right\rangle_{\mathbb{L}^2}
    \nonumber
  \end{eqnarray}}
  Since $\mathbb{L}^4\hookrightarrow \mathbb{L}^2$ and $\mathbb{W}^{1,4}\hookrightarrow \mathbb{L}^2$, so there are constants $C_1$ and $C_2<\infty$ such that
  {
  \begin{eqnarray*}
    &&\left\langle\frac{\partial u_n^{\prime}(t)}{\partial x_i},(u_n^{\prime}(t)-u^{\prime}(t))\times\frac{\partial \psi(t)}{\partial x_i}\right\rangle_{\mathbb{L}^2}\leq\left\|\frac{\partial u_n^{\prime}(t)}{\partial x_i}\right\|_{\mathbb{L}^2}\left\|(u_n^{\prime}(t)-u^{\prime}(t))\times\frac{\partial \psi(t)}{\partial x_i}\right\|_{\mathbb{L}^2}\\
    &\leq&\left\|u_n^{\prime}(t)\right\|_{\mathbb{H}^1}C_1\|u_n^{\prime}(t)-u^{\prime}(t)\|_{\mathbb{L}^4}C_2\|\psi(t)\|_{\mathbb{W}^{1,4}}.
  \end{eqnarray*}}
  Hence
  {
  \begin{eqnarray*}
    &&\mathbb{E}^{\prime}\int_0^T\left|\left\langle\frac{\partial u_n^{\prime}(t)}{\partial x_i},(u_n^{\prime}(t)-u^{\prime}(t))\times\frac{\partial \psi(t)}{\partial x_i}\right\rangle_{\mathbb{L}^2}\right|\ud t\\
    &\leq& C_1C_2\mathbb{E}^{\prime}\int_0^T\|u_n^{\prime}(t)\|_{\mathbb{H}^1}\|u_n^{\prime}(t)-u^{\prime}(t)\|_{\mathbb{L}^4}\|\psi(t)\|_{\mathbb{W}^{1,4}}\ud t.
  \end{eqnarray*}}
  Moreover by the H$\ddot{\textrm{o}}$lder's inequality,
  {
  \begin{eqnarray*}
    &&\mathbb{E}^{\prime}\int_0^T\|u_n^{\prime}(t)\|_{\mathbb{H}^1}\|u_n^{\prime}(t)-u^{\prime}(t)\|_{\mathbb{L}^4}\|\psi(t)\|_{\mathbb{W}^{1,4}}\ud t\\
    &\leq&\left(\mathbb{E}^{\prime}\int_0^T\|u_n^{\prime}(t)\|_{\mathbb{H}^1}^2\ud t\right)^\frac{1}{2} \left(\mathbb{E}^{\prime}\int_0^T\|u_n^{\prime}(t)-u^{\prime}(t)\|_{\mathbb{L}^4}^4\ud t\right)^\frac{1}{4} \left(\mathbb{E}^{\prime}\int_0^T\|\psi(t)\|_{\mathbb{W}^{1,4}}^4\ud t \right)^\frac{1}{4}\\
    &\leq&T^\frac{1}{2}\left(\mathbb{E}^{\prime}\sup_{t\in[0,T]}\|u_n^{\prime}(t)\|_{\mathbb{H}^1}^2\right)^\frac{1}{2} \left(\mathbb{E}^{\prime}\int_0^T\|u_n^{\prime}(t)-u^{\prime}(t)\|_{\mathbb{L}^4}^4\ud t\right)^\frac{1}{4} \left(\mathbb{E}^{\prime}\int_0^T\|\phi(t)\|_{\mathbb{W}^{1,4}}^4\ud t \right)^\frac{1}{4}
  \end{eqnarray*}}
  By \eqref{eq:4.6phi}, \eqref{eq:4.13phi} and since $\psi\in L^4(\Omega^{\prime};L^4(0,T;\mathbb{W}^{1,4}))$, we have
  \[
    \lim_{n\rightarrow\infty}\left(\mathbb{E}^{\prime}\sup_{t\in[0,T]}\|u_n^{\prime}(t)\|_{\mathbb{H}^1}^2\right)^\frac{1}{2} \left(\mathbb{E}^{\prime}\int_0^T\|u_n^{\prime}(t)-u^{\prime}(t)\|_{\mathbb{L}^4}^4\ud t\right)^\frac{1}{4} \left(\mathbb{E}^{\prime}\int_0^T\|\psi(t)\|_{\mathbb{W}^{1,4}}^4\ud t \right)^\frac{1}{4}
    =0
  \]
  Hence
  \begin{equation}\label{eq:4.21phi}
    \lim_{n\rightarrow\infty}\mathbb{E}^{\prime}\int_0^T\left|\left\langle\frac{\partial u_n^{\prime}(t)}{\partial x_i},(u_n^{\prime}(t)-u^{\prime}(t))\times\frac{\partial \psi(t)}{\partial x_i}\right\rangle_{\mathbb{L}^2}\right|\ud t=0
  \end{equation}
  Both $u^{\prime}$ and $\frac{\partial \psi}{\partial x_i}$ are in $L^2(\Omega^{\prime};L^2(0,T;\mathbb{L}^2))$, so $u^{\prime}\times \frac{\partial \psi}{\partial x_i}\in L^2(\Omega^{\prime};L^2(0,T;\mathbb{L}^2))$. Hence by \eqref{eq:4.17phi}, we have
  \begin{equation}\label{eq:4.22phi}
    \lim_{n\rightarrow\infty}\mathbb{E}^{\prime}\int_0^T\left\langle\frac{\partial u_n^{\prime}(t)}{\partial x_i}-\frac{\partial u^{\prime}(t)}{\partial x_i},u^{\prime}(t)\times\frac{\partial \psi(t)}{\partial x_i}\right\rangle_{\mathbb{L}^2}\ud t=0.
  \end{equation}
  Therefore by \eqref{eq:4.20phi}, \eqref{eq:4.21phi}, \eqref{eq:4.22phi},
  \begin{equation}\label{eq:4.23phi}
    \lim_{n\rightarrow\infty}\mathbb{E}^{\prime}\int_0^T\left\langle\frac{\partial u_n^{\prime}(t)}{\partial x_i},u_n^{\prime}(t)\times\frac{\partial \psi(t)}{\partial x_i}\right\rangle_{\mathbb{L}^2}\ud t=\mathbb{E}^{\prime}\int_0^T\left\langle\frac{\partial u^{\prime}(t)}{\partial x_i},u^{\prime}(t)\times\frac{\partial \psi(t)}{\partial x_i}\right\rangle_{\mathbb{L}^2}\ud t
  \end{equation}
  Then by \eqref{eq:4.19phi}, we have
  \begin{equation}\label{eq:unxdeltaunwc}
    \lim_{n\rightarrow\infty}\mathbb{E}^{\prime}\int_0^T\langle u_n^{\prime}(t)\times\Delta u_n^{\prime}(t),\psi(t)\rangle_{\mathbb{L}^2}\ud t=\mathbb{E}^{\prime}\int_0^T\sum_{i=1}^3\left\langle\frac{ \partial u^{\prime}(t)}{\partial x_i},u^{\prime}(t)\times\frac{\partial \psi(t)}{\partial x_i}\right\rangle_{\mathbb{L}^2}\ud t
  \end{equation}

Secondly, we will show that
 \[\lim_{n\rightarrow\infty}\mathbb{E}^{\prime}\int_0^T\left\langle u_n^{\prime}(t)\times\pi_n\nabla \phi\big(u_n^{\prime}(t)\big),\psi\right\rangle_{\mathbb{L}^2}\ud t=\mathbb{E}^{\prime}\int_0^T\left\langle u^{\prime}(t)\times \nabla \phi\big(u^{\prime}(t)\big),\psi\right\rangle_{\mathbb{L}^2}\ud t.\]
Since
  {
  \begin{eqnarray*}
    &&\left|\left\langle u_n^{\prime}(t)\times\pi_n\nabla \phi\big(u_n^{\prime}(t)\big),\psi\rangle_{\mathbb{L}^2}-\langle u^{\prime}(t)\times\nabla \phi\big(u^{\prime}(t)\big),\psi\right\rangle_{\mathbb{L}^2}\right|\\
    &\leq&\left|\left\langle \left(u_n^{\prime}(t)-u^{\prime}(t)\right)\times\pi_n\nabla \phi\big(u_n^{\prime}(t)\big),\psi\right\rangle_{\mathbb{L}^2}\right|+\left| \left\langle u^{\prime}(t)\times\left(\pi_n\nabla \phi\big(u_n^{\prime}(t)\big)-\nabla \phi\big(u^{\prime}(t)\big)\right),\psi\right\rangle_{\mathbb{L}^2}\right|\\
    &\leq&\big\|\psi\big\|_{\mathbb{L}^2}\big\|u_n^{\prime}(t)-u^{\prime}(t)\big\|_{\mathbb{L}^2}\big\|\nabla \phi\big(u_n^{\prime}(t)\big)\big\|_{\mathbb{L}^2}+\big\|\psi\big\|_{\mathbb{L}^2}\big\|u^{\prime}(t)\big\|_{\mathbb{L}^2} \big\|\pi_n\nabla \phi\big(u_n^{\prime}(t)\big)-\nabla \phi\big(u^{\prime}(t)\big)\big\|_{\mathbb{L}^2},
     \end{eqnarray*}}
we have
{
  \begin{eqnarray*}
    &&\left|\mathbb{E}^{\prime}\int_0^T\left\langle u_n^{\prime}(t)\times\pi_n\nabla \phi\big(u_n^{\prime}(t)\big),\psi\right\rangle_{\mathbb{L}^2}\ud t-\mathbb{E}^{\prime}\int_0^T\left\langle u^{\prime}(t)\times \nabla \phi\big(u^{\prime}(t)\big),\psi\right\rangle_{\mathbb{L}^2}\ud t\right|\\
    &\leq&\mathbb{E}^{\prime}\int_0^T\left(\big\|\psi\big\|_{\mathbb{L}^2}\big\|u_n^{\prime}(t)-u^{\prime}(t)\big\|_{\mathbb{L}^2}\big\|\nabla \phi\big(u_n^{\prime}(t)\big)\big\|_{\mathbb{L}^2}+\big\|\psi\big\|_{\mathbb{L}^2}\big\|u^{\prime}(t)\big\|_{\mathbb{L}^2} \big\|\pi_n\nabla \phi\big(u_n^{\prime}(t)\big)-\nabla \phi\big(u^{\prime}(t)\big)\big\|_{\mathbb{L}^2}\right)\ud t\\
    &\leq&\left(\mathbb{E}^{\prime}\int_0^T\big\|\psi\big\|_{\mathbb{L}^{1,4}}^4\ud t\right)^\frac{1}{4}\left(\mathbb{E}^{\prime}\int_0^T\big\|u_n^{\prime}(t)-u^{\prime}(t)\big\|_{\mathbb{L}^4}^4\ud t\right)^\frac{1}{4}\left(\mathbb{E}^{\prime}\int_0^T\big\|\nabla \phi\big(u_n^{\prime}(t)\big)\big\|_{\mathbb{L}^2}^2\ud t\right)^\frac{1}{2}\\
    &&+\left(\mathbb{E}^{\prime}\int_0^T\big\|\psi\big\|_{\mathbb{W}^{1,4}}^4\ud t\right)^\frac{1}{4}\left(\mathbb{E}^{\prime}\int_0^T\big\|u^{\prime}(t)\big\|_{\mathbb{L}^4}^4\ud t\right)^\frac{1}{4} \left(\mathbb{E}^{\prime}\int_0^T\big\|\pi_n\nabla \phi\big(u_n^{\prime}(t)\big)-\nabla \phi\big(u^{\prime}(t)\big)\big\|_{{\mathbb{L}^2}}^2\ud t\right)^\frac{1}{4}\rightarrow0.
  \end{eqnarray*}}

We need to prove why
\begin{equation}\label{eq:piphiun'conv}
 \mathbb{E}^{\prime}\int_0^T\big\|\pi_n\nabla \phi\big(u_n^{\prime}(t)\big)-\nabla \phi\big(u^{\prime}(t)\big)\big\|_{{\mathbb{L}^2}}^2\ud t \to 0
 \end{equation}
This is because
{
  \begin{eqnarray*}
    &&\left(\mathbb{E}^{\prime}\int_0^T\big\|\pi_n\nabla \phi\big(u_n^{\prime}(t)\big)-\nabla \phi\big(u^{\prime}(t)\big)\big\|_{{\mathbb{L}^2}}^2\ud t\right)^\frac{1}{2}\\
    &\leq&\left(\mathbb{E}^{\prime}\int_0^T\left\|\pi_n\nabla \phi\big(u_n^{\prime}(t)\big)-\pi_n\nabla \phi\big(u^{\prime}(t)\big)\right\|_{\mathbb{L}^2}^2\ud t\right)^\frac{1}{2}+\left(\mathbb{E}^{\prime}\int_0^T\left\|\pi_n\nabla \phi\big(u^{\prime}(t)\big)-\nabla \phi\big(u^{\prime}(t)\big)\right\|_{\mathbb{L}^2}^2\ud t\right)^\frac{1}{2}\leq\cdots
  \end{eqnarray*}}
Since $\nabla \phi$ is global Lipschitz, there exists a constant $C$ such that
{
  \begin{eqnarray*}
    \cdots&\leq&C\left(\mathbb{E}^{\prime}\int_0^T\left\|u_n^{\prime}(t)-u^{\prime}(t)\right\|_{\mathbb{L}^2}^2\ud t\right)^\frac{1}{2}+\left(\mathbb{E}^{\prime}\int_0^T\left\|\pi_n\nabla \phi\big(u^{\prime}(t)\big)-\nabla \phi\big(u^{\prime}(t)\big)\right\|_{\mathbb{L}^2}^2\ud t\right)^\frac{1}{2}.
  \end{eqnarray*}}
  By \eqref{eq:4.13phi}, the first term on the right hand side of above inequality converges to $0$. And since $\left\|\pi_n\nabla \phi\big(u^{\prime}(t)\big)-\nabla \phi\big(u^{\prime}(t)\big)\right\|_{\mathbb{L}^2}^2\rightarrow0$ for almost every $(t,\omega)\in [0,T]\times \Omega$, and since $\nabla \phi$ is bounded, $\left\|\pi_n\nabla \phi\big(u^{\prime}(t)\big)-\nabla \phi\big(u^{\prime}(t)\big)\right\|_{\mathbb{L}^2}^2$ is uniformly integrable, hence the second term of right hand side also converges to $0$ as $n\rightarrow\infty$. Therefore we have proved \eqref{eq:piphiun'conv}.

  Hence we have
  \begin{equation}\label{eq:unxphi'wc}
  \lim_{n\rightarrow\infty}\mathbb{E}^{\prime}\int_0^T\left\langle u_n^{\prime}(t)\times\pi_n\nabla \phi\big(u_n^{\prime}(t)\big),\psi\right\rangle_{\mathbb{L}^2}\ud t=\mathbb{E}^{\prime}\int_0^T\left\langle u^{\prime}(t)\times \nabla \phi\big(u^{\prime}(t)\big),\psi\right\rangle_{\mathbb{L}^2}\ud t.
  \end{equation}

  Therefore by the equalities \eqref{eq:unxdeltaunwc} and \eqref{eq:unxphi'wc}, we have
{
  \begin{eqnarray}
    &&\lim_{n\rightarrow\infty}\mathbb{E}^{\prime}\int_0^T\langle u_n^{\prime}(t)\times[\Delta u_n^{\prime}(t)-\pi_n\nabla \phi(u_n^{\prime}(t))],\psi(t)\rangle_{\mathbb{L}^2}\ud t\label{eq:4.231phi}\\
    &=&\mathbb{E}^{\prime}\int_0^T\sum_{i=1}^3\left\langle\frac{ \partial u^{\prime}(t)}{\partial x_i},u^{\prime}(t)\times\frac{\partial \psi(t)}{\partial x_i}\right\rangle_{\mathbb{L}^2}\ud t+\mathbb{E}^{\prime}\int_0^T\left\langle u^{\prime}(t)\times \nabla \phi\big(u^{\prime}(t)\big),\psi\right\rangle_{\mathbb{L}^2}\ud t.\nonumber
   \end{eqnarray}}

 Moreover, by \eqref{eq:4.14phi}, for every $\psi\in L^2(\Omega^{\prime};L^2(0,T;\mathbb{L}^2))$,
  \begin{equation}\label{eq:4.24phi}
    \lim_{n\rightarrow\infty}\mathbb{E}^{\prime}\int_0^T\langle u_n^{\prime}(t)\times\left(\Delta u_n^{\prime}(t)-\pi_n\nabla \phi\big(u_n^{\prime}(t)\big)\right),\psi\rangle_{\mathbb{L}^2}\ud t=\mathbb{E}^{\prime}\int_0^T\langle Y(t),\psi\rangle_{\mathbb{L}^2}\ud t.
  \end{equation}
  Hence by \eqref{eq:4.231phi} and \eqref{eq:4.24phi},
  {
  \begin{eqnarray*}
    &&\lim_{n\rightarrow\infty}\mathbb{E}^{\prime}\int_0^T\langle u_n^{\prime}(t)\times\left(\Delta u_n^{\prime}(t)-\pi_n\nabla \phi\big(u_n^{\prime}(t)\big)\right),\psi(s)\rangle_{\mathbb{L}^2}\ud t\\
    &=&\mathbb{E}^{\prime}\int_0^T\langle Y(t),\psi(t)\rangle_{\mathbb{L}^2}\ud t\\
    &=&\mathbb{E}^{\prime}\int_0^T\sum_{i=1}^3\langle\frac{\partial u^{\prime}(t)}{\partial x_i},u^{\prime}(t)\times\frac{\partial\psi(t)}{\partial x_i}\rangle_{\mathbb{L}^2}\ud t+\mathbb{E}^{\prime}\int_0^T\left\langle u^{\prime}(t)\times \nabla \phi\big(u^{\prime}(t)\big),\psi(t)\right\rangle_{\mathbb{L}^2}\ud t.
  \end{eqnarray*}}
  This completes the proof of Lemma \ref{lem:4.5phi}.
\end{proof}

\begin{lem}\label{lem:4.8phi}
  For any process $\psi\in L^4(\Omega^{\prime};L^4(0,T;\mathbb{L}^4))$ we have
  {
  \begin{eqnarray}
    &&\hspace{-2truecm}\lefteqn{\lim_{n\rightarrow\infty}\mathbb{E}^{\prime}\int_0^T\ \! _{\mathbb{L}^\frac{3}{2}}\langle u_n^{\prime}(s)\times\big(u_n^{\prime}(s)\times\left(\Delta u_n^{\prime}-\pi_n\nabla \phi\big(u_n^{\prime}(t)\big)\right)\big),\psi(s)\rangle_{\mathbb{L}^3}\ud s}\nonumber\\
    &=&\mathbb{E}^{\prime}\int_0^T \ \!_{\mathbb{L}^\frac{3}{2}}\langle Z(s),\psi(s)\rangle_{\mathbb{L}^3}\ud s\label{eq:4.25phi}\\
    &=&\mathbb{E}^{\prime}\int_0^T\ \! _{\mathbb{L}^\frac{3}{2}}\langle u^{\prime}(s)\times Y(s),\psi(s)\rangle_{\mathbb{L}^3}\ud s.\label{eq:4.26phi}
  \end{eqnarray}}
\end{lem}
\begin{proof}
 Let us take $\psi\in L^4(\Omega^{\prime};L^4(0,T;\mathbb{L}^4))$.
  For $n\in \mathbb{N}$, put $Y_n:=u_n^{\prime}\times\left(\Delta u_n^{\prime}+\nabla \phi\big(u_n^{\prime}\big)\right)$.
  Since $L^4(\Omega^{\prime};L^4(0,T;\mathbb{L}^4))\subset L^2(\Omega^{\prime};L^2(0,T;\mathbb{L}^3))=\left[L^2(\Omega^{\prime};L^2(0,T;\mathbb{L}^\frac{3}{2}))\right]'$,  we deduce that  \eqref{eq:4.15phi} implies that \eqref{eq:4.25phi} holds.\\
  So it remains to prove equality \eqref{eq:4.26phi}. Since by the H$\ddot{\textrm{o}}$lder's inequality
  {
  \begin{eqnarray*}
    \|\psi\times u^{\prime}\|_{\mathbb{L}^2}^2&=&\int_D|\psi(x)\times u^{\prime}(x)|^2\ud x\leq\int_D|\psi(x)|^2|u^{\prime}(x)|^2\ud x
    \\ &\leq& \|\psi\|_{\mathbb{L}^4}^2\|u^{\prime}\|_{\mathbb{L}^4}^2
    \leq \|\psi\|_{\mathbb{L}^4}^4+\|u^{\prime}\|_{\mathbb{L}^4}^4.
  \end{eqnarray*}}
  And since by \eqref{eq:4.13phi}, $u^{\prime}\in L^4(\Omega^{\prime};L^4(0,T;\mathbb{L}^4))$, we infer that
  {
  \begin{eqnarray*}
    &&\mathbb{E}^{\prime}\int_0^T\|\psi\times u^{\prime}\|_{\mathbb{L}^2}^2\ud t\leq\mathbb{E}^{\prime}\int_0^T\|\psi\|_{\mathbb{L}^4}^4\ud t+\mathbb{E}^{\prime}\int_0^T\|u^{\prime}\|_{\mathbb{L}^4}^4\ud t<\infty.
  \end{eqnarray*}}
   This proves that $\psi\times u^{\prime}\in L^2(\Omega^{\prime};L^2(0,T;\mathbb{L}^2))$ and similarly $\psi\times u_n^{\prime}\in L^2(\Omega^{\prime};L^2(0,T;\mathbb{L}^2))$.\\
   Thus since by \eqref{eq:4.14phi}, $Y_n\in L^2(\Omega^{\prime};L^2(0,T;\mathbb{L}^2))$, we infer that
   {
  \begin{eqnarray}
    \,\!_{\mathbb{L}^\frac{3}{2}}\langle u_n^{\prime}\times Y_n,\psi\rangle_{\mathbb{L}^3}&=&\int_D\langle u^{\prime}_n(x)\times Y_n(x),\psi(x)\rangle\ud x\nonumber\\
    &=&\int_D\langle Y_n(x),\psi(x)\times u_n^{\prime}(x)\rangle\ud x\label{eq:4.261phi}=\langle Y_n,\psi\times u^{\prime}_n\rangle_{\mathbb{L}^2}.
  \end{eqnarray}}
   Similarly, since by \eqref{eq:4.14phi}, $Y\in L^2(\Omega^{\prime};L^2(0,T;\mathbb{L}^2))$, we have
   {
  \begin{eqnarray}
    &&\,\!_{\mathbb{L}^\frac{3}{2}}\langle u^{\prime}\times Y,\psi\rangle_{\mathbb{L}^3}=\int_D\langle u^{\prime}(x)\times Y(x),\psi(x)\rangle\ud x\nonumber\\
    &=&\int_D\langle Y(x),\psi(x)\times u^{\prime}(x)\rangle\ud x\label{eq:4.262phi}=\langle Y,\psi\times u^{\prime}\rangle_{\mathbb{L}^2}.
  \end{eqnarray}}
  Thus by \eqref{eq:4.261phi} and \eqref{eq:4.262phi}, we get
  {
  \begin{eqnarray*}
    \,\!_{\mathbb{L}^\frac{3}{2}}\langle u_n^{\prime}\times Y_n,\psi\rangle_{\mathbb{L}^3}-\,\!_{\mathbb{L}^\frac{3}{2}}\langle u^{\prime}\times Y,\psi\rangle_{\mathbb{L}^3}&=&\langle Y_n,\psi\times u_n^{\prime}\rangle_{\mathbb{L}^2}-\langle Y,\psi\times u^{\prime}\rangle_{\mathbb{L}^2}\\
    &=&\langle Y_n-Y,\psi\times u^{\prime}\rangle_{\mathbb{L}^2}+\langle Y_n,\psi\times(u_n^{\prime}-u^{\prime})\rangle_{\mathbb{L}^2}.
  \end{eqnarray*}}
  In order to prove \eqref{eq:4.26phi}, we are aiming to prove that the expectation of the left hand side of the above equality goes to $0$ as $n\rightarrow\infty$.
  By \eqref{eq:4.14phi},  since $\psi\times u^{\prime}\in L^2(\Omega^{\prime};L^2(0,T;\mathbb{L}^2))$,
  \[\lim_{n\rightarrow\infty}\mathbb{E}^{\prime}\int_0^T\langle Y_n(s)-Y(s),\psi(s)\times u^{\prime}(s)\rangle_{\mathbb{L}^2}\ud s=0.\]
  By the Cauchy-Schwartz inequality and the equation \eqref{eq:4.13phi}, we have
  {
  \begin{eqnarray*}
    &&\hspace{-1truecm}\lefteqn{\mathbb{E}^{\prime}\int_0^T\langle Y_n(s),\psi(s)\times(u_n^{\prime}(s)-u^{\prime}(s))\rangle_{\mathbb{L}^2}^2\ud s\leq\mathbb{E}^{\prime}\int_0^T\|Y_n(s)\|^2_{\mathbb{L}^2}\|\psi(s)\times (u_n^{\prime}(s)-u^{\prime}(s))\|^2_{\mathbb{L}^2}\ud s}\\
  &\leq&\mathbb{E}^{\prime}\int_0^T\big\|Y_n(s)\big\|_{\mathbb{L}^2}\big\|\psi(s)\big\|_{\mathbb{L}^4} \big\|u_n^{\prime}(s)-u^{\prime}(s)\big\|_{\mathbb{L}^4}\ud s\\
  &\leq&\left(\mathbb{E}^{\prime}\int_0^T\big\|Y_n(s)\big\|_{\mathbb{L}^2}^2\ud s\right)^\frac{1}{2}\left(\mathbb{E}^{\prime}\int_0^T\big\|\psi(s)\big\|_{\mathbb{L}^4}^4\ud s\right)^\frac{1}{4}\left(\mathbb{E}^{\prime}\int_0^T\big\|u_n^{\prime}(s)-u^{\prime}(s)\big\|_{\mathbb{L}^4}^4\ud s\right)^\frac{1}{4}\rightarrow0.
  \end{eqnarray*}}

  Therefore, we infer that
  \[\lim_{n\rightarrow\infty}\mathbb{E}^{\prime}\int_0^T\ \! _{\mathbb{L}^\frac{3}{2}}\langle u_n^{\prime}(s)\times(u_n^{\prime}(s)\times\Delta u_n^{\prime}(s)),\psi(s)\rangle_{\mathbb{L}^3}\ud s=\mathbb{E}^{\prime}\int_0^T\ \! _{\mathbb{L}^\frac{3}{2}}\langle u^{\prime}(s)\times Y(s),\psi(s)\rangle_{\mathbb{L}^3}\ud s.\]
  This completes the proof of the Lemma \ref{lem:4.8phi}.
\end{proof}

The next result will be used, see Theorem \ref{thm:u'=1},  to show that the process $u^{\prime}$ satisfies the condition $|u^{\prime}(t,x)|_{\R^3}=1$ for all $t\in [0,T]$, $x\in D$ and $\mathbb{P}^{\prime}$-almost surely.

\begin{lem}\label{lem:4.7phi}
  For any bounded measurable function $\psi:D\longrightarrow\R$ we have
  \[\langle Y(s,\omega),\psi u^{\prime}(s,\omega)\rangle_{\mathbb{L}^2}=0,\]
  for almost every $(s,\omega)\in[0,T]\times\Omega^{\prime}$.
\end{lem}
\begin{proof}
Let $B\subset[0,T]\times\Omega^{\prime}$ be an arbitrary progressively measurable set.
  {
  \begin{eqnarray*}
    &&\left|\mathbb{E}^{\prime}\int_0^T1_B(s)\left\langle u_n^{\prime}(s)\times\left(\Delta u_n^{\prime}(s)-\pi_n\nabla \phi\big(u_n^{\prime}(s)\big)\right),\psi u_n^{\prime}(s)\right\rangle_{\mathbb{L}^2}\ud s-\mathbb{E}^{\prime}\int_0^T1_B(s)\langle Y(s),\psi u^{\prime}(s)\rangle_{\mathbb{L}^2}\ud s\right|\\
    &\leq&\left|\mathbb{E}^{\prime}\int_0^T1_B(s)\left\langle u_n^{\prime}(s)\times\left(\Delta u_n^{\prime}(s)-\pi_n\nabla \phi\big(u_n^{\prime}(s)\big)\right),\psi (u_n^{\prime}(s)-u'(s))\right\rangle_{\mathbb{L}^2}\ud s\right|\\
    &&+\left|\mathbb{E}^{\prime}\int_0^T1_B(s)\llangle u_n^{\prime}(s)\times\left(\Delta u_n^{\prime}(s)-\pi_n\nabla \phi\big(u_n^{\prime}(s)\big)\right)-Y(s),\psi u^{\prime}(s)\rrangle_{\mathbb{L}^2}\ud s\right|.
  \end{eqnarray*}}
  Next we will show that both terms in the right hand side of the above inequality will converge to $0$. \\
  For the first term, by the boundness of $\psi$, \eqref{eq:4.7phi} and \eqref{eq:4.13phi}, we have
  {
  \begin{eqnarray*}
  &&\left|\mathbb{E}^{\prime}\int_0^T1_B(s)\left\langle u_n^{\prime}(s)\times\left(\Delta u_n^{\prime}(s)-\pi_n\nabla \phi\big(u_n^{\prime}(s)\big)\right),\psi (u_n^{\prime}(s)-u'(s))\right\rangle_{\mathbb{L}^2}\ud s\right|\\
  &\leq& \mathbb{E}'\int_0^T\left\|u_n^{\prime}(s)\times\left(\Delta u_n^{\prime}(s)-\pi_n\nabla \phi\big(u_n^{\prime}(s)\big)\right)\psi\right\|_{\mathbb{L}^2}\left\|u_n'(s)-u'(s)\right\|_{\mathbb{L}^2}\ud s\\
  &\leq&\left\|u_n^{\prime}(s)\times\left(\Delta u_n^{\prime}(s)-\pi_n\nabla \phi\big(u_n^{\prime}(s)\big)\right)\psi\right\|_{L^2(\Omega';L^2(0,T;\mathbb{L}^2))}\left\|u_n'-u'\right\|_{L^2(\Omega';L^2(0,T;\mathbb{L}^2))}\longrightarrow 0.
  \end{eqnarray*}}
  For the second term, since $1_B\psi u'\in L^2(\Omega';L^2(0,T;\mathbb{L}^2))$, by \eqref{eq:4.13phi} and \eqref{eq:4.14phi}, we have
  {
  \begin{eqnarray*}
    &&\left|\mathbb{E}^{\prime}\int_0^T1_B(s)\llangle u_n^{\prime}(s)\times\left(\Delta u_n^{\prime}(s)-\pi_n\nabla \phi\big(u_n^{\prime}(s)\big)\right)-Y(s),\psi u^{\prime}(s)\rrangle_{\mathbb{L}^2}\ud s\right|\rightarrow0.
  \end{eqnarray*}}
  Therefore
  {
  \begin{eqnarray*}
    0&=&\lim_{n\rightarrow\infty}\mathbb{E}^{\prime}\int_0^T1_B(s)\left\langle u_n^{\prime}(s)\times\left(\Delta u_n^{\prime}(s)-\pi_n\nabla \phi\big(u_n^{\prime}(s)\big)\right),\psi u_n^{\prime}(s)\right\rangle_{\mathbb{L}^2}\ud s\\
    &=&\mathbb{E}^{\prime}\int_0^T1_B(s)\langle Y(s),\psi u^{\prime}(s)\rangle_{\mathbb{L}^2}\ud s,
  \end{eqnarray*}}
  where the first equality from the fact that $\langle a\times b, a\rangle=0$.
  By the arbitrariness of $B$, this concludes the proof of Lemma \ref{lem:4.7phi}.
\end{proof}

\section{Conclusion of the proof of the existence of a weak solution}
Our aim in this section is to prove that the process $u^{\prime}$ from Proposition \ref{prop:u_n'phi} is a weak solution of the equation \eqref{eq:7eq} according to the definition \ref{defn:sol}.\\
We define a sequence of ${\mathbb{L}^2}$-valued process $(M_n(t))_{t\in [0,T]}$ on the original probability space $(\Omega,\mathcal{F},\mathbb{P})$ by
{
  \begin{eqnarray}
M_n(t)&:=&u_n(t)-u_n(0)-\lmd_1\int_0^t\pi_n\left(u_n(s)\times\left(\Delta u_n(s)-\pi_n\nabla \phi\big(u_n(s)\big)\right)\right)\ud s\nonumber\\
&&+\lmd_2\int_0^t\pi_n\left(u_n(s)\times\left(u_n(s)\times\left(\Delta u_n(s)-\pi_n\nabla \phi\big(u_n(s)\big)\right)\right)\right)\ud s\label{eq:Mnt4}\\
&&-\frac{1}{2}\sum_{j=1}^N\int_0^t\pi_n\left((\pi_n (u_n(s)\times h_j))\times h_j\right)\ud s.\nonumber
 \end{eqnarray}}

Since $u_n$ is the solution of the Equation \eqref{eq:dun}, we have
{
  \begin{eqnarray*}
    u_n(t)&=&u_n(0)+\lmd_1\int_0^t\pi_n\left(u_n(s)\times\left(\Delta u_n(s)-\pi_n\nabla \phi\big(u_n(s)\big)\right)\right)\ud s\\
    &&-\lmd_2\int_0^t \pi_n\left(u_n(s)\times\left(u_n(s)\times\left(\Delta u_n(s)-\pi_n\nabla \phi\big(u_n(s)\big)\right)\right)\right)\ud s\\
    &&+\frac{1}{2}\sum_{j=1}^N\int_0^t\pi_n(\pi_n(u_n(s)\times h_j)\times h_j)\ud s+\sum_{j=1}^N\int_0^t\pi_n(u_n(s)\times h_j)\ud W_j(s).
  \end{eqnarray*}}
Hence we have
\begin{equation}\label{eq:MnItointe}
M_n(t)=\sum_{j=1}^N\int_0^t\pi_n(u_n(s)\times h_j)\ud W_j(s),\quad t\in[0,T].
\end{equation}

It will be $2$ steps to prove $u^{\prime}$ is a weak solution of the Equation \eqref{eq:7eq}:
\begin{trivlist}
\item[Step 1]:
 we are going to find some $M'(t)$ defined similar as in \eqref{eq:Mnt4}, but with $u^{\prime}$ instead of $u_n$.
\item[Step 2]:
  We will show the similar result as in \eqref{eq:MnItointe} but with $u^{\prime}$ instead of $u_n$ and $W_j^{\prime}$ instead of $W_j$.
 \end{trivlist}
\subsection{Step 1}
We define a sequence of ${\mathbb{L}^2}$-valued process $\big(M_n'(t)\big)_{t\in[0,T]}$ on the new probability space $(\Omega^{\prime},\mathcal{F}^{\prime},\mathbb{P}^{\prime})$ by a formula similar as \eqref{eq:Mnt4}.
{
  \begin{eqnarray}
M_n'(t)&:=&u_n^{\prime}(t)-u_n^{\prime}(0)-\lmd_1\int_0^t\pi_n\left(u_n^{\prime}(s)\times\left(\Delta u_n^{\prime}(s)-\pi_n\nabla \phi\big(u_n^{\prime}(s)\big)\right)\right)\ud s\nonumber\\
&&+\lmd_2\int_0^t\pi_n\left(u_n^{\prime}(s)\times\left(u_n^{\prime}(s)\times\left(\Delta u_n^{\prime}(s)-\pi_n\nabla \phi\big(u_n^{\prime}(s)\big)\right)\right)\right)\ud s\label{eq:Mn't4}\\
&&-\frac{1}{2}\sum_{j=1}^N\int_0^t\pi_n[(\pi_n (u_n^{\prime}(s)\times h_j))\times h_j]\ud s.\nonumber
 \end{eqnarray}}

It will be natural to ask if $\{M_n'\}$ has limit and if yes, what is the limit. The next result answers this question.
\begin{lem}\label{lem:limMn'4}
  For each $t\in[0,T]$ the sequence of random variables $M_n'(t)$ converges weakly as $n$ goes to infinity in $L^2(\Omega^{\prime};X^{-\beta})$ to the limit
  {
  \begin{eqnarray*}
M'(t)&:=&u^{\prime}(t)-u_0-\lmd_1\int_0^t\left(u^{\prime}(s)\times\left(\Delta u^{\prime}(s)-\nabla \phi\big(u^{\prime}(s)\big)\right)\right)\ud s\\
&&+\lmd_2\int_0^t\left(u^{\prime}(s)\times\left(u^{\prime}(s)\times\left(\Delta u^{\prime}(s)-\nabla \phi\big(u^{\prime}(s)\big)\right)\right)\right)\ud s\\
&&-\frac{1}{2}\sum_{j=1}^N\int_0^t(u^{\prime}(s)\times h_j)\times h\ud s.
 \end{eqnarray*}}
\end{lem}
\begin{proof}
The dual space of $L^2(\Omega^{\prime};X^{-\beta})$ is $L^2(\Omega^{\prime};X^\beta)$. Let $t\in(0,T]$ and $U\in L^2(\Omega^{\prime};X^\beta)$. We have
{
  \begin{eqnarray*}
    &&\,\!_{L^2(\Omega^{\prime};X^{-\beta})}\langle M_n'(t),U\rangle_{L^2(\Omega^{\prime};X^\beta)}=\mathbb{E}^{\prime}\left[_{X^{-\beta}}\langle M_n'(t),U\rangle_{X^\beta}\right]\\
    &=&\mathbb{E}^{\prime}\Big[_{X^{-\beta}}\langle u_n^{\prime}(t),U\rangle_{X^\beta}-_{X^{-\beta}}\langle u_n(0),U\rangle_{X^\beta}\\
    &&-\lmd_1\int_0^t\left\langle u_n^{\prime}(s)\times\left(\Delta u_n^{\prime}(s)-\pi_n\nabla \phi\big(u_n^{\prime}(s)\big)\right),\pi_n U
    \right\rangle_{\mathbb{L}^2}\ud s\\
    &&+\lmd_2\int_0^t \,\!_{X^{-\beta}}\left\langle \left(u_n^{\prime}(s)\times\left(u_n^{\prime}(s)\times\left(\Delta u_n^{\prime}(s)-\pi_n\nabla \phi\big(u_n^{\prime}(s)\big)\right)\right)\right),U\right\rangle_{X^\beta}\ud s\\
    &&-\frac{1}{2}\sum_{j=1}^N\int_0^t\langle \pi_n(u_n^{\prime}(s)\times h_j)\times h_j,\pi_n U\rangle_{\mathbb{L}^2}\ud s\Big].
  \end{eqnarray*}}
  We know that $u_n^{\prime}\longrightarrow u^{\prime}$ in $C(0,T;X^{-\beta})$ $\mathbb{P}^{\prime}$-a.s., so
   \[\sup_{t\in[0,T]}\|u_n(t)-u(t)\|_{X^{-\beta}}\longrightarrow0,\quad \mathbb{P}^{\prime}-a.s.\]
   so $u_n^{\prime}(t)\longrightarrow u^{\prime}(t)$ in $X^{-\beta}$ $\mathbb{P}^{\prime}$-almost surely for any $t\in[0,T]$.
   And $\,\!_{X^{-\beta}}\langle \cdot,U\rangle_{X^\beta}$ is a continuous function on $X^{-\beta}$, hence
   \[\lim_{n\rightarrow\infty}\,\!_{X^{-\beta}}\langle u_n^{\prime}(t),U\rangle_{X^\beta}=\,\!_{X^{-\beta}}\langle u^{\prime}(t),U\rangle_{X^\beta},\quad \mathbb{P}^{\prime}-a.s.\]
   By \eqref{eq:4.5phi}, $\sup_{t\in[0,T]}\|u_n^{\prime}(t)\|_{{\mathbb{L}^2}}\leq\|u_0\|_{\mathbb{L}^2}$, since ${\mathbb{L}^2}\hookrightarrow X^{-\beta}$ continuously, we can find a constant $C$ such that
   {
  \begin{eqnarray*}
    &&\sup_n\mathbb{E}^{\prime}\left[\left|{}_{X^{-\beta}}\langle u_n^{\prime}(t),U\rangle_{X^\beta}\right|^2\right]\leq\sup_n\mathbb{E}^{\prime}\|U\|^2_{X^\beta}\mathbb{E}^{\prime}\|u_n^{\prime}(t)\|^2_{X^{-\beta}}\\
    &\leq&C\mathbb{E}^{\prime}\|U\|_{X^\beta}^2\mathbb{E}^{\prime}\sup_n\|u_n^{\prime}(t)\|^2_{\mathbb{L}^2}\leq C\mathbb{E}^{\prime}\|U\|_{X^\beta}^2\mathbb{E}^{\prime}\|u_0\|_{\mathbb{L}^2}^2<\infty.
  \end{eqnarray*}}
  Hence the sequence ${}_{X^{-\beta}}\langle u_n^{\prime}(t),U\rangle_{X^{\beta}}$ is uniformly integrable.
   So the almost surely convergence and uniform integrability implies that
    $$\lim_{n\rightarrow \infty}\mathbb{E}^{\prime}[\,\!_{X^{-\beta}}\langle u_n^{\prime}(t),U\rangle_{X^\beta}]=\mathbb{E}^{\prime}[\,\!_{X^{-\beta}}\langle u^{\prime}(t),U\rangle_{X^\beta}].$$
   By \eqref{eq:4.14phi},
   \[\lim_{n\rightarrow\infty}\mathbb{E}^{\prime}\int_0^t\left\langle u_n^{\prime}(s)\times\left(\Delta u_n^{\prime}(s)-\pi_n\nabla \phi\big(u_n^{\prime}(s)\big)\right),\pi_nU\right\rangle_{\mathbb{L}^2}\ud s=\mathbb{E}^{\prime}\int_0^t\langle Y(s), U\rangle_{\mathbb{L}^2}.\]
   By \eqref{eq:4.16phi}
   \[\lim_{n\rightarrow\infty}\mathbb{E}^{\prime}\int_0^t\,\!_{X^{-\beta}}\left\langle \pi_n\left(u_n^{\prime}(s)\times\left(u_n^{\prime}(s)\times\left(\Delta u_n^{\prime}(s)-\pi_n\nabla \phi\big(u_n^{\prime}(s)\big)\right)\right)\right),U\right\rangle_{X^\beta}\ud s=\mathbb{E}^{\prime}\int_0^t\,\!\langle Z(s),U\rangle_{X^\beta}\ud s.\]
   By the H$\ddot{\textrm{o}}$lder's inequality,
   \[\|u_n^{\prime}(t)-u^{\prime}(t)\|_{\mathbb{L}^2}^2\leq m(D)^\frac{1}{2}\|u_n^{\prime}(t)-u^{\prime}(t)\|_{\mathbb{L}^4}^2.\]
   Hence by \eqref{eq:4.13phi},
   {
  \begin{eqnarray*}
    &&\hspace{-2truecm}\lefteqn{\mathbb{E}^{\prime}\int_0^t\langle \pi_n((u_n^{\prime}(s)-u^{\prime}(s))\times h_j)\times h_j,\pi_nU\rangle_{\mathbb{L}^2}\ud s}\\
    &\leq&\|U\|_{L^2(\Omega^{\prime};L^2(0,T;\mathbb{L}^2))}\left(\mathbb{E}^{\prime}\int_0^t(\|(u_n^{\prime}(s)-u^{\prime}(s))\times h_j)\times h_j\|_{\mathbb{L}^2}^2\ud s\right)^\frac{1}{2}\\
    &\leq&\|h_j\|_{\mathbb{L}^\infty}^2\|U\|_{L^2(\Omega^{\prime};L^2(0,T;\mathbb{L}^2))}\left(\mathbb{E}^{\prime}\int_0^t\| u_n^{\prime}(s)-u^{\prime}(s)\|_{\mathbb{L}^2}^2\ud s\right)^\frac{1}{2}\\
    &\leq&\|h_j\|_{\mathbb{L}^\infty}^2\|U\|_{L^2(\Omega^{\prime};L^2(0,T;\mathbb{L}^2))}\left(\mathbb{E}^{\prime}\int_0^t\| u_n^{\prime}(s)-u^{\prime}(s)\|_{\mathbb{L}^4}^2\ud s\right)^\frac{1}{2}m(D)^\frac{1}{4}\\
    &\leq&\|h_j\|_{\mathbb{L}^\infty}^2\|U\|_{L^2(\Omega^{\prime};L^2(0,T;\mathbb{L}^2))}\left(\mathbb{E}^{\prime}\int_0^t\| u_n^{\prime}(s)-u^{\prime}(s)\|_{\mathbb{L}^4}^4\ud s\right)^\frac{1}{4}t^\frac{1}{4}m(D)^\frac{1}{4}\\
    &\longrightarrow&0.
  \end{eqnarray*}}
  Hence
  {
  \begin{eqnarray*}
    &&\hspace{-1truecm}\lefteqn{\lim_{n\rightarrow\infty}\,\!_{L^2(\Omega^{\prime};X^{-\beta})}\langle M_n'(t),U\rangle_{L^2(\Omega^{\prime};X^\beta)}
    }\\
    &=&\mathbb{E}^{\prime}\Big[\,\!_{X^{-\beta}}\langle u^{\prime}(t), U\rangle_{X^\beta}-\,\!_{X^{-\beta}}\langle u_0,U\rangle_{X^\beta}-\lmd_1\int_0^t\langle Y(s),U\rangle_{\mathbb{L}^2}\ud s\\
    &+&\lmd_2\int_0^t\,\!_{X^{-\beta}}\langle Z(s),U\rangle_{X^\beta}\ud s-\frac{1}{2}\sum_{j=1}^N\int_0^t\langle(u^{\prime}(s)\times h_j)\times h_j,U\rangle_{\mathbb{L}^2}\ud s\Big].
  \end{eqnarray*}}
  Since by Lemma \ref{lem:4.5phi} and Lemma \ref{lem:4.8phi}, we have $Y=u^{\prime}\times \Delta u^{\prime}$ and $Z=u^{\prime}\times(u^{\prime}\times\Delta u^{\prime})$. Therefore for any $U\in L^2(\Omega^{\prime};X^\beta)$,
  \[\lim_{n\rightarrow\infty}\,\!_{L^2(\Omega^{\prime};X^{-\beta})}\langle M_n'(t),U\rangle_{L^2(\Omega^{\prime};X^\beta)}=\,\!_{L^2(\Omega^{\prime};X^{-\beta})}\langle M'(t),U\rangle_{L^2(\Omega^{\prime};X^\beta)}.\]
  This concludes the proof of Lemma \ref{lem:limMn'4}.
\end{proof}

Before we can continue to prove $u^{\prime}$ is the weak solution of equation \eqref{eq:7eq}, we need to show that the $W^{\prime}$ and $W_{n}^{\prime}$ in Proposition \ref{prop:u_n'phi} are Brownian Motions. This  will be done in Lemmata \ref{lem: Wn BMphi} and \ref{lem:W'BMphi}, which can be proved by considering the characteristic functions. We will only show the proof of Lemma \ref{lem:W'BMphi}.

\begin{lem}\label{lem: Wn BMphi}
  Suppose the $W_n'$ defined in $(\Omega',\mathcal{F}',\mathbb{P}')$ has the same distribution as the Brownian Motion $W$ defined in $(\Omega,\mathcal{F},\mathbb{P})$ as in Proposition \ref{prop:u_n'phi}. Then $W_n'$ is also a Brownian Motion.
\end{lem}

\begin{lem}\label{lem:W'BMphi}
  The process $(W^{\prime}(t))_{t\in[0,T]}$ is a real-valued Brownian Motion on $(\Omega^{\prime},\mathcal{F}^{\prime},\mathbb{P}^{\prime})$ and if $0\leq s<t\leq T$ then the increment $W^{\prime}(t)-W^{\prime}(s)$ is independent of the $\sigma$-algebra generated by $u^{\prime}(r)$ and $W^{\prime}(r)$ for $r\in [0,s]$.
\end{lem}
\begin{proof}
We consider the characteristic functions of $W^{\prime}$.
  Let $k\in\mathbb{N}$ and $0=s_0<s_1<\cdots<s_k\leq T$. For $(t_1,\ldots,t_k)\in\R^k$, we have for each $n\in\mathbb{N}$:
  \[\mathbb{E}^{\prime}\left[e^{i\sum_{j=1}^kt_j(W^{\prime}_n(s_j)-W_n'(s_{j-1}))}\right]= e^{-\frac{1}{2}\sum_{j=1}^kt_j^2(s_j-s_{j-1})}.\]
  Notice that $\left|e^{i\sum_{j=1}^kt_j(W^{\prime}_n(s_j)-W_n'(s_{j-1}))}\right|\leq 1$, by the Lebesgue's dominated convergence theorem,
  \[\mathbb{E}^{\prime}\left[e^{i\sum_{j=1}^kt_j(W^{\prime}(s_j)-W^{\prime}(s_{j-1}))}\right]=\lim_{n\rightarrow \infty}\mathbb{E}^{\prime}\left[e^{i\sum_{j=1}^kt_j(W^{\prime}_n(s_j)-W_n'(s_{j-1}))}\right]= e^{-\frac{1}{2}\sum_{j=1}^kt_j^2(s_j-s_{j-1})}.\]
  Hence $W^{\prime}(t)$ has the same distribution with $W_n'(t)$ for $t\in[0,T]$. Since random variables are independent if and only if the characteristic function of the sum of them equals to the multiplication of their characteristic functions, and here we have
    \[\mathbb{E}^{\prime}\left[e^{i\sum_{j=1}^kt_j(W^{\prime}(s_j)-W^{\prime}(s_{j-1}))}\right]=\prod_{j=1}^k\mathbb{E}^{\prime}\left[e^{it_j(W^{\prime}(s_j)-W^{\prime}(s_{j-1}))}\right].\]
    Hence $W^{\prime}$ has independent increments.\\
    And
    \[W^{\prime}(0)=\lim_{n\rightarrow\infty}W^{\prime}_n(0)=0,\quad \mathbb{P}^{\prime}-a.s.,\]
    so $(W^{\prime}(t))_{t\in[0,T]}$ is a real-valued Brownian motion on $(\Omega^{\prime},\mathcal{F}^{\prime},\mathbb{P}^{\prime})$.\\
    The law of $(u_n,W)$ is same as $(u^{\prime}_n,W_n')$ and if $t>s\geq r$, $W(t)-W(s)$ is independent of $u_n(r)$, so as the same method as before we can see $W_n'(t)-W_n'(s)$ is independent of $u^{\prime}_n(r)$ for all $n$. By Proposition \ref{prop:u_n'phi}, $\lim_{n\rightarrow\infty}\|u_n^{\prime}(r)\|_{{\mathbb{V}}'}=\|u^{\prime}(r)\|_{{\mathbb{V}}'}$ and $\lim_{n\rightarrow\infty}(W_n'(t)-W_n'(s))=W^{\prime}(t)-W^{\prime}(s)$, hence by the Lebesgue's dominated convergence theorem we have
     {
  \begin{eqnarray*}
    &&\mathbb{E}^{\prime}\left(e^{it\left(\|u^{\prime}(r)\|_{{\mathbb{V}}'}+W^{\prime}(t)-W^{\prime}(s)\right)}\right)=\lim_{n\rightarrow\infty}\mathbb{E}^{\prime}\left(e^{it\left(\|u_n^{\prime}(r)\|_{{\mathbb{V}}'}+W_n'(t)-W_n'(s)\right)}\right)\\
    &=&\lim_{n\rightarrow\infty}\mathbb{E}^{\prime}\left(e^{it\left(\|u_n^{\prime}(r)\|_{{\mathbb{V}}'}\right)}\right)\mathbb{E}^{\prime}\left(e^{it(W_n'(t)-W_n'(s))}\right) =\mathbb{E}^{\prime}\left(e^{it\left(\|u^{\prime}(r)\|_{{\mathbb{V}}'}\right)}\right)\mathbb{E}^{\prime}\left(e^{it(W^{\prime}(t)-W^{\prime}(s))}\right).
  \end{eqnarray*}}
  So $W^{\prime}(t)-W^{\prime}(s)$ is independent of $u^{\prime}(r)$.
    Hence this completes the proof of Lemma \ref{lem:W'BMphi}.
    \end{proof}
    \begin{rem}
      We will denote $\mathbb{F}^{\prime}$ the filtration generated by $(u^{\prime},W^{\prime})$ and $\mathbb{F}_n'$ the filtration generated by $(u_n^{\prime},W_n')$. Then by Lemma \ref{lem:W'BMphi}, $u^{\prime}$ is progressively measurable with respect to $\mathbb{F}^{\prime}$ and by Lemma \ref{lem: Wn BMphi}, $u_n^{\prime}$ is progressively measurable with respect to $\mathbb{F}_n'$.
    \end{rem}
\subsection{Step 2}
Let us summarize what we have achieved so far. We have got our process $M'$ and have shown $W^{\prime}$ is a Wiener process. Next we will show a similar result as in equation \eqref{eq:MnItointe} to prove $u^{\prime}$ is a weak solution of the Equation \eqref{eq:7eq}. \\
But before that we still need some preparation.
The following result is needed to prove Lemma \ref{lem:4 ineqphi}.
  \begin{prop}\label{prop:uxhbd} Suppose that $\beta>\frac{1}{4}$ and  $h\in\mathbb{L}^\infty\cap\mathbb{W}^{1,3}$. Then there exists $c_h>0$, such that for every $u\in X^{\beta}$, $u\times h \in X^{-\beta}$ and
    \begin{equation}\label{eq:uxhbd}
    \|u\times h\|_{X^{-\beta}}\leq c_h\|u\|_{X^{-\beta}}<\infty.
    \end{equation}
  \end{prop}
  \begin{proof}
      Let us fix $\beta>\frac{1}{4}$,  $h\in \mathbb{L}^\infty\cap\mathbb{W}^{1,3}$. Then for every   $z\in\mathbb{H}^1$ we have
      {
  \begin{eqnarray*}
    \|z\times h\|_{\mathbb{H}^1}^2&=&\|\nabla(z\times h)\|^2_{\mathbb{L}^2}+\|z\times h\|^2_{\mathbb{L}^2}
    \leq2(\|\nabla z\times h\|^2_{\mathbb{L}^2}+\|z\times\nabla h\|_{\mathbb{L}^2}^2)+\|z\times h\|_{\mathbb{L}^2}^2
    \\
    &\leq&2(\|h\|_{\mathbb{L}^\infty}^2\|\nabla z\|_{\mathbb{L}^2}^2+\|\nabla h\|_{\mathbb{L}^3}^2\|z\|_{\mathbb{L}^6}^2)+\|h\|_{\mathbb{L}^\infty}^2\|z\|_{\mathbb{L}^2}^2
    \leq2(\|h\|_{\mathbb{L}^\infty}^2+c^2\|\nabla h\|_{\mathbb{L}^3}^2)\|z\|_{\mathbb{H}^1}^2,
  \end{eqnarray*}}
  where the constant $c$ is from $\|\cdot\|_{\mathbb{L}^6}\leq c\|\cdot\|_{\mathbb{H}^1}.$
  So the linear  map $
   M_h: \mathbb{H}^1\ni z\longmapsto z\times h\in \mathbb{H}^1
  $
  is  bounded. Obviously, $M_h : \mathbb{L}^2 \to \mathbb{L}^2$ is also linear and bounded. Hence by the interpolation theorem, as $X^\beta=[\mathbb{L}^2, \mathbb{H}^1]_{\beta}$, we infer that $M_h : X^\beta \to X^\beta$ is  linear and bounded as well.

Next, let us fix $u\in \mathbb{L}^2 \subset X^{-\beta}$ (since $\beta>\frac{1}{4}$) and $z \in X^\beta$. Since $X^{-\beta}$ is equal to the dual space of $X^\beta$   we have
\begin{eqnarray*}
\vert  \langle u\times h,z\rangle \vert  &=& \vert  \langle u,z\times h\rangle \vert
\\
&\leq & \|u\|_{X^{-\beta}}  \|z\times h \|_{X^{\beta}}=\|u\|_{X^{-\beta}}  \|M_h(z) \|_{X^{\beta}} \leq c_h \|u\|_{X^{-\beta}} \|z \|_{X^{\beta}}.
\end{eqnarray*}
Since $\mathbb{L}^2$ is a dense subspace  of $X^{-\beta}$, we infer that the above inequality holds for every $u\in   X^{-\beta}$.
In particular, for every $u\in   X^{-\beta}$, $u\times h \in X^{-\beta}$ and inequality \eqref{eq:uxhbd} holds. This completes the proof
of Proposition \ref{prop:uxhbd}.
\end{proof}

The proof of next Lemma is omitted because it is similar as part of the proof of Lemma 5.2 in Brze{\'z}niak, Goldys and Jegaraj \cite{ZB&TJ}.

    \begin{lem}\label{lem:4 ineqphi}
    For each $m\in\mathbb{N}$, we define the partition $\Big\{s_i^m:=\frac{iT}{m},\;i=0,\ldots,m\Big\}$ of $[0,T]$.
      Then for any $\eps>0$, there exists $m_0(\eps)\in\mathbb{N}$ such that for all $m\geq m_0(\eps)$, we have:
      \begin{trivlist}
        \item[(i)]
        \[\lim_{n\rightarrow\infty}\left(\mathbb{E}^{\prime}\left[\left\|\sum_{j=1}^N\int_0^t\left(\pi_n(u_n^{\prime}(s)\times h_j)-\sum_{i=0}^{m-1}\pi_n(u_n^{\prime}(s_i^m)\times h_j)1_{(s_i^m,s_{i+1}^m]}(s)\right)\ud W_{jn}^{\prime}(s)\right\|_{X^{-\beta}}^2\right]\right)^\frac{1}{2}<\frac{\eps}{2};\]
        \item[(ii)]
        {
  \begin{eqnarray*}
  &&\hspace{-1truecm}\lefteqn{\lim_{n\rightarrow\infty}\mathbb{E}^{\prime}\Bigg[\Bigg\|\sum_{i=0}^{m-1}\sum_{j=1}^N\pi_n(u_n^{\prime}(s_i^m)\times h_j)(W^{\prime}_{jn}(t\wedge s_{i+1}^m)-W^{\prime}_{jn}(t\wedge s_i^m)) }\\
  &&-\sum_{i=0}^{m-1}\sum_{j=1}^N\pi_n(u^{\prime}(s_i^m)\times h_j)(W_j^{\prime}(t\wedge s_{i+1}^m)-W_j^{\prime}(t\wedge s_i^m))\Bigg\|^2_{X^{-\beta}}\Bigg]=0;
        \end{eqnarray*}}
        \item[(iii)]
        \[\lim_{n\rightarrow\infty}\left(\mathbb{E}^{\prime}\left[\left\|\sum_{j=1}^N\int_0^t(\pi_n(u^{\prime}(s)\times h_j)-\sum_{i=0}^{m-1}\pi_n(u^{\prime}(s_i^m)\times h_j)1_{(s_i^m,s_{i+1}^m]}(s))\ud W_j^{\prime}(s)\right\|^2_{X^{-\beta}}\right]\right)^\frac{1}{2}<\frac{\eps}{2};\]
        \item[(iv)]
        \[\lim_{n\rightarrow\infty}\mathbb{E}^{\prime}\left[\left\|\sum_{j=1}^N\int_0^t(\pi_n(u^{\prime}(s)\times h_j)-(u^{\prime}(s)\times h_j))\ud W_j^{\prime}(s)\right\|_{X^{-\beta}}^2\right]=0.\]
      \end{trivlist}
    \end{lem}

Now we are ready to state the Theorem which means that $u^{\prime}$ is the weak solution of the equation \eqref{eq:7eq}.
\begin{thm}\label{thm:M'tint4}
   For each $t\in[0,T]$ we have $M'(t)=\sum_{j=1}^N\int_0^t(u^{\prime}(s)\times h_j)\ud W_j^{\prime}(s)$.
\end{thm}
    \begin{proof}
     \begin{trivlist}
       \item{Step 1:}
     We will show that
     \begin{equation}\label{eq:Mn'ito}
       M'_n(t)=\sum_{j=1}^N\int_0^t\pi_n(u_n^{\prime}(s)\times h_j)\ud W^{\prime}_{jn}(s)
     \end{equation}
     $\mathbb{P}^{\prime}$ almost surely for each $t\in[0,T]$ and $n\in\mathbb{N}$.\\
      Let us fix that $t\in[0,T]$ and $n\in\mathbb{N}$. Let us also fix $m\in\mathbb{N}$ and define the partition $\Big\{s_i^m:=\frac{iT}{m},i=0,\ldots,m\Big\}$ of $[0,T]$. Let us recall that $(u_n^{\prime},W^{\prime}_n)$ and $(u_n,W)$ have the same laws on the separable Banach space $C([0,T];H_n)\times C([0,T];\R^N)$. Since the map
      {
  \begin{eqnarray*}
    \Psi: &&C([0,T];H_n)\times C([0,T];\R^N)\longrightarrow H_n\\
    &&(u_n,W)\longmapsto M_n(t)-\sum_{i=0}^{m-1}\sum_{j=1}^N\pi_n(u_n(s_i^m)\times h_j)(W_j(t\wedge s_{i+1}^m)-W_j(t\wedge s_i^m))
  \end{eqnarray*}}
      is continuous and so measurable.
      By involving the Kuratowski Theorem we infer that the ${\mathbb{L}^2}$-valued random variables:
      \[M_n(t)-\sum_{i=0}^{m-1}\sum_{j=1}^N\pi_n(u_n(s_i^m)\times h_j)(W_j(t\wedge s_{i+1}^m)-W_j(t\wedge s_i^m))\]
      and
      \[M_n'(t)-\sum_{j=0}^{m-1}\sum_{j=1}^N\pi_n(u_n^{\prime}(s_i^m)\times h_j)(W_{jn}^{\prime}(t\wedge s_{i+1}^m)-W_{jn}^{\prime}(t\wedge s_j^m))\]
      have the same laws. Let us denote $u_{n,m}:=\sum_{i=0}^{m-1}u_n(s_i^m)1_{[s_i^m,s_{i+1}^m)}$. By the It\^{o} isometry, we have
      {
  \begin{eqnarray}
    &&\left\|\sum_{i=0}^{m-1}\pi_n(u_n(s_i^m)\times h_j)(W_j(t\wedge s_{i+1}^m)-W_j(t\wedge s_i^m))-\int_0^t\pi_n(u_n(s)\times h_j)\ud W_j(s)\right\|^2_{L^2(\Omega;{\mathbb{L}^2})}\label{eq:7.6}\\
    &=&\mathbb{E}\left\|\int_0^t\left(\pi_n(u_{n,m}\times h_j)-\pi_n(u_n(s)\times h_j)\right)\ud W_j(s)\right\|_{\mathbb{L}^2}^2\leq\|h_j\|_{\mathbb{L}^\infty}^2\mathbb{E}\int_0^t\|u_{n,m}(s)-u_n(s)\|_{\mathbb{L}^2}^2\ud s.\nonumber
  \end{eqnarray}}
  Since $u_n\in C([0,T];H_n)$ $\mathbb{P}$-almost surely, we have
  \begin{equation}\label{eq:tilde{u}_n(s)-u_n(s)0}
    \lim_{m\rightarrow\infty}\int_0^t\|u_{n,m}(s)-u_n(s)\|_{\mathbb{L}^2}^2\ud s=0,\quad\mathbb{P}-a.s..
  \end{equation}
  Moreover by the equality \eqref{eq:3.10phi}, we infer that
  {
  \begin{eqnarray}
   &&\hspace{-3truecm}\lefteqn{ \sup_m\mathbb{E}\left|\int_0^t\|u_{n,m}(s)-u_n(s)\|_{\mathbb{L}^2}^2\ud s\right|^2\leq \sup_m\mathbb{E}\left|\int_0^t\left(2\|u_{n,m}(s)\|_{\mathbb{L}^2}^2+2\|u_n(s)\|^2_{\mathbb{L}^2}\right)\ud s\right|^2}
   \label{eq:7.8}\\
   &\leq&\mathbb{E}\left|4\|u_0\|_{\mathbb{L}^2}^2T\right|^2=16\|u_0\|_{\mathbb{L}^2}^4T^2<\infty.\nonumber
  \end{eqnarray}}
 By \eqref{eq:7.8}, we have  $\int_0^t\|u_{n,m}(s)-u_n(s)\|_{\mathbb{L}^2}^2\ud s$ is uniformly (with respect to $m$) integrable.
  Therefore by the integrability and \eqref{eq:tilde{u}_n(s)-u_n(s)0}, we have
  \[\lim_{m\rightarrow\infty}\mathbb{E}\int_0^t\|u_{n,m}(s)-u_n(s)\|_{\mathbb{L}^2}^2\ud s=0.\]
  Then by above equality and \eqref{eq:7.6}, we have
      \[\lim_{m\rightarrow\infty}\left\|\sum_{i=0}^{m-1}\pi_n(u_n(s_i^m)\times h_j)(W_j(t\wedge s_{i+1}^m)-W_j(t\wedge s_i^m))-\int_0^t\pi_n(u_n(s)\times h_j)\ud W_j(s)\right\|^2_{L^2(\Omega;{\mathbb{L}^2})}\!\!\!\!\!\!=0.\]
      Similarly, because $u_n^{\prime}$ satisfies the same conditions as $u_n$, we also get
      \[\lim_{m\rightarrow\infty}\left\|\sum_{i=0}^{m-1}\pi_n(u^{\prime}_n(s_i^m)\times h)(W_{jn}^{\prime}(t\wedge s_{i+1}^m)-W_{jn}^{\prime}(t\wedge s_i^m))-\int_0^t\pi_n(u^{\prime}_n(s)\times h_j)\ud W_{jn}^{\prime}(s)\right\|^2_{L^2(\Omega;{\mathbb{L}^2})}\!\!\!\!\!\!=0.\]
  Hence, since the $L^2$ convergence implies the weak convergence, we infer that the random variables $M_n(t)-\sum_{j=1}^N\int_0^t\pi_n(u_n(s)\times h_j)\ud W_j(s)$ and $M_n'(t)-\sum_{j=1}^N\int_0^t\pi_n(u^{\prime}_n(s)\times h_j)\ud W^{\prime}_{jn}(s)$ have same laws. But $M_n(t)-\sum_{j=1}^N\int_0^t\pi_n(u_n(s)\times h_j)\ud W_j(s)=0$ $\mathbb{P}$-almost surely, so \eqref{eq:Mn'ito} follows.
  \item{Step 2:} From Lemma \ref{lem:4 ineqphi} and the Step 1, we infer that $M_n'(t)$ converges in $L^2(\Omega^{\prime};X^{-\beta})$ to $\sum_{j=1}^N\int_0^t(u^{\prime}(s)\times h_j)\ud W^{\prime}_j(s)$ as $n\rightarrow\infty$.
\end{trivlist}
This completes the proof of Theorem \ref{thm:M'tint4}.
\end{proof}

Summarizing, it follows from Theorem \ref{thm:M'tint4} that for every  $t\in [0,T]$ the following equality is  satisfied  in $L^2(\Omega^{\prime};X^{-\beta})$:
  {
  \begin{eqnarray}
    u^{\prime}(t)=u_0&+&\lmd_1\int_0^t\left(u^{\prime}\times\left(\Delta u^{\prime}-\nabla \phi(u^{\prime})\right)\right)(s)\ud s
    \label{eq:5.1phi}\\
    &-&\lmd_2\int_0^tu^{\prime}(s)\times\left(u^{\prime}\times\left(\Delta u^{\prime}-\nabla \phi(u^{\prime})\right)\right)(s)\ud s\nonumber\\
    &+&\sum_{j=1}^N\int_0^t(u^{\prime}(s)\times h_j)\circ \ud W_j^{\prime}(s).\nonumber
  \end{eqnarray}}

Hence by Definition \ref{defn:sol}, $u^{\prime}$ is a weak solution of Equation \eqref{eq:7eq}.\\

\section{Regularity of a  weak solution}
Now we will start to show some regularity of $u^{\prime}$.
\begin{thm}\label{thm:u'=1}
  The process $u^{\prime}$ from Proposition \ref{prop:u_n'phi} satisfies:
  \begin{equation}\label{eq:2.12phi}
  |u^{\prime}(t,x)|_{\R^3}=1, \quad\textrm{for Lebesgue a.e. } (t,x)\in [0,T]\times D \textrm{ and }\mathbb{P}^{\prime}-a.s..
  \end{equation}
\end{thm}

To prove Theorem \ref{thm:u'=1}, we need the following Lemma:
\begin{lem}\cite{Pardoux}(Th. 1.2)\label{lem:Pardouxphi}
  Let $(\Omega,(\mathcal{F}_t),\mathbb{P})$ be a filtered probability space and let ${\mathbb{V}}$ and ${\mathbb{L}^2}$ be two separable Hilbert spaces, such that ${\mathbb{V}}\hookrightarrow {\mathbb{L}^2}$ continuously and densely. We identify ${\mathbb{L}^2}$ with it's dual space and have a Gelfand triple: ${\mathbb{V}}\hookrightarrow {\mathbb{L}^2}\cong {\mathbb{L}^2}'\hookrightarrow {\mathbb{V}}'$.
  We assume that
  \[u\in M^2(0,T;{\mathbb{V}}),\;u_0\in {\mathbb{L}^2},\;v\in M^2(0,T;{\mathbb{V}}'),\;z_j\in M^2(0,T;{\mathbb{L}^2}),\]
  for every $t\in[0,T]$,
  \[u(t)=u_0+\int_0^tv(s)\ud s+\sum_{j=1}^N\int_0^tz_j(s)\ud W_j(s),\quad\mathbb{P}-a.s..\]
  Let $\psi$ be a twice differentiable functional on ${\mathbb{L}^2}$, which satisfies:
  \begin{trivlist}
    \item[(i)] $\psi$, $\psi'$ and $\psi^\bis$ are locally bounded.
    \item[(ii)] $\psi$ and $\psi'$ are continuous on ${\mathbb{L}^2}$.
    \item[(iii)] Let $\mathscr{L}^1({\mathbb{L}^2})$ be the Banach space of all the trace class operators on ${\mathbb{L}^2}$. Then $\forall Q\in \mathscr{L}^1({\mathbb{L}^2})$, $Tr[Q\circ \psi^\bis]$ is a continuous functional on ${\mathbb{L}^2}$.
    \item[(iv)] If $u\in {\mathbb{V}}$, $\psi'(u)\in {\mathbb{V}}$; $u\mapsto\psi'(u)$ is continuous from ${\mathbb{V}}$ (with the strong topology) into ${\mathbb{V}}$ endowed with the weak topology.
    \item[(v)] $\exists k$ such that $\|\psi'(u)\|_{\mathbb{V}}\leq k(1+\|u\|_{\mathbb{V}})$, $\forall u\in {\mathbb{V}}$.
  \end{trivlist}
  Then for every $t\in[0,T]$,
  {
  \begin{eqnarray*}
    \psi(u(t))&=&\psi(u_0)+\int_0^t{}_{{\mathbb{V}}^*}\langle v(s),\psi'(u(s))\rangle_{\mathbb{V}} \ud s+\sum_{j=1}^N\int_0^t{}{}_{\mathbb{L}^2}\langle\psi'(u(s)),z_j(s)\rangle_{\mathbb{L}^2}\ud W_s\\
    &&+\frac{1}{2}\sum_{j=1}^N\int_0^t{}_{\mathbb{L}^2}\langle\psi^\bis(u(s))z_j(s),z_j(s)\rangle_{\mathbb{L}^2}\ud s,\quad\mathbb{P}-a.s..
  \end{eqnarray*}}
\end{lem}

\begin{proof}[Proof of Theorem \ref{thm:u'=1}]
  Let $\xi\in C_0^\infty(D,\R)$. Then we consider a function
  \[\psi:{\mathbb{L}^2}\ni u\longmapsto \langle u,\xi u\rangle_{\mathbb{L}^2}\in \R.\]
  It's easy to see that $\psi$ is of $C^2$ class and $\psi'(u)=2\xi u$, $\psi^\bis(u)(v)=2\xi v$, $u,v\in {\mathbb{L}^2}$.
  Next we will check the assumptions of Lemma \ref{lem:Pardouxphi}. By previous work (see details below), $u^{\prime}$ satisfies:
  \[\mathbb{E}^{\prime}\int_0^T\|u^{\prime}(t)\|_{\mathbb{V}}^2\ud t<\infty,\quad\textrm{by \eqref{eq:4.12phi}},\]
  \[\mathbb{E}^{\prime}\int_0^T\|\left(u^{\prime}\times \left(\Delta u^{\prime}-\nabla \phi(u^{\prime})\right)\right)(t)\|_{X^{-\beta}}^2\ud t<\infty,\quad\textrm{by \eqref{eq:4.14phi}},\]
  \[\mathbb{E}^{\prime}\int_0^T\|u^{\prime}(t)\times\left(u^{\prime}\times \left(\Delta u^{\prime}-\nabla \phi(u^{\prime})\right)\right)(t)\|_{X^{-\beta}}^2\ud t<\infty,\quad\textrm{by \eqref{eq:4.16phi}},\]
  \[\mathbb{E}^{\prime}\int_0^T\|(u^{\prime}(s)\times h_j)\times h_j\|_{X^{-\beta}}^2\ud t<\infty,\quad\textrm{by \eqref{eq:4.10phi}},\]
  \[\mathbb{E}^{\prime}\int_0^T\|u^{\prime}(s)\times h_j\|_{\mathbb{L}^2}^2\ud t<\infty,\quad\textrm{by \eqref{eq:4.10phi}}.\]
  And $\psi$ satisfies:
  \begin{trivlist}
    \item[(i)] $\psi$, $\psi'$, $\psi^\bis$ are locally bounded.
    \item[(ii)] Since $\psi'$, $\psi^\bis$ exist, $\psi$, $\psi'$ are continuous on ${\mathbb{L}^2}$.
    \item[(iii)]$\forall Q\in \mathscr{L}^1({\mathbb{L}^2})$,
     \[Tr[Q\circ \psi''(a)]=\sum_{j=1}^\infty\langle Q\circ\psi''(a)e_j,e_j\rangle_{\mathbb{L}^2}=2\sum_{j=1}^\infty\langle Q(\eta e_j),e_j\rangle_{\mathbb{L}^2},\]
    which is a constant in $\R$, so
     the map ${\mathbb{L}^2}\ni a\longmapsto Tr[Q\circ \psi^\bis(a)]\in {\mathbb{L}^2}$ is a continuous functional on ${\mathbb{L}^2}$.
    \item[(iv)] If $u\in {\mathbb{V}}$, $\psi'(u)\in {\mathbb{V}}$; $u\mapsto\psi'(u)$ is continuous from ${\mathbb{V}}$ (with the strong topology) into ${\mathbb{V}}$ endowed with the weak topology.\\
        This is because: For any $v^*\in X^{-\beta}$, we have
        {
  \begin{eqnarray*}
    {}_{X^\beta}\langle \psi'(u+v)-\psi'(u), v^*\rangle_{X^{-\beta}}&=&{}_{X^\beta}\langle 2\phi v,v^*\rangle_{X^{-\beta}}\\
    &\leq&2|\xi|_{C(D,\R)}{}_{X^\beta}\langle v,v^*\rangle_{X^{-\beta}},
  \end{eqnarray*}}
  hence $\psi'$ is weakly continuous. Let us denote $\tau$ as the strong topology of ${\mathbb{V}}$ and $\tau_w$ the weak topology of ${\mathbb{V}}$. Take $B\in \tau_w$, by the weak continuity $(\psi')^{-1}(B)\in \tau_w$, but $\tau_w\subset \tau$. Hence $(\psi')^{-1}(B)\in \tau$, which implies (iv).
    \item[(v)] $\exists k=2\|\xi\|_{C_0(D)}$ such that
    $$\|\psi'(u)\|_{\mathbb{V}}=2\|\xi u\|_{\mathbb{V}}\leq k(1+\|u\|_{\mathbb{V}}),\quad \forall u\in {\mathbb{V}}.$$
  \end{trivlist}
  Hence by Lemma \ref{lem:Pardouxphi}, we have that for $t\in[0,T]$, $\mathbb{P}^{\prime}$ almost surely,
  {
  \begin{eqnarray*}
    &&\langle u^{\prime}(t),\xi u^{\prime}(t)\rangle_{\mathbb{L}^2}-\langle u_0,\xi u_0\rangle_{\mathbb{L}^2}\\
    &=&\sum_{j=1}^N\int_0^t{}_{X^{-\beta}}\langle \lmd_1\left(u^{\prime}\times \left(\Delta u^{\prime}-\nabla \phi(u^{\prime})\right)\right)(s)-\lmd_2 u^{\prime}(s)\times(u^{\prime}\times (\Delta u^{\prime}\\
    &&+\nabla \phi(u^{\prime})))(s)+\frac{1}{2}(u^{\prime}(s)\times h_j)\times h_j,2\xi u^{\prime}(s)\rangle_{X^\beta}\ud s\\
    &&+\sum_{j=1}^N\int_0^t\langle 2\xi u^{\prime}(s),u^{\prime}(s)\times h_j\rangle_{\mathbb{L}^2}\ud W^{\prime}(s)+\sum_{j=1}^N\int_0^t\langle \xi u^{\prime}(s)\times h_j,u^{\prime}(s)\times h_j\rangle_{\mathbb{L}^2}\ud s.
  \end{eqnarray*}}
  By Lemma \ref{lem:4.7phi},
  $${}_{X^{-\beta}}\langle \lmd_1\left(u^{\prime}\times \left(\Delta u^{\prime}-\nabla \phi(u^{\prime})\right)\right)(s),2\xi u^{\prime}(s)\rangle_{X^\beta}=0.$$
  And since
  \[{}_{X^{-\beta}}\langle\lmd_2 u^{\prime}(s)\times\left(u^{\prime}\times \left(\Delta u^{\prime}-\nabla \phi(u^{\prime})\right)\right)(s),2\xi u^{\prime}(s)\rangle_{X^\beta}=0,\]
  \[{}_{X^{-\beta}}\langle (u^{\prime}(s)\times h_j)\times h_j,\xi u^{\prime}(s)\rangle_{X^\beta}=-{}_{X^{-\beta}}\langle u^{\prime}(s)\times h_j,\xi u^{\prime}(s)\times h_j\rangle_{X^\beta},\]
  \[\langle 2\xi u^{\prime}(s),u^{\prime}(s)\times h_j\rangle_{\mathbb{L}^2}=0,\]
  we have
  \[\langle u^{\prime}(t),\xi u^{\prime}(t)\rangle_{\mathbb{L}^2}-\langle u_0,\xi u_0\rangle_{\mathbb{L}^2}=0,\quad \mathbb{P}^{\prime}-a.s.\]
  Since $\xi$ is arbitrary and $|u_0(x)|=1$ for almost every $x\in D$, we infer that $|u^{\prime}(t,x)|=1$ for almost every $x\in D$ as well. This completes the proof of Theorem \ref{thm:u'=1}.
\end{proof}

From  Theorem \ref{thm:u'=1} we can deduce the following result.
\begin{thm}\label{thm:u'teqH}
The process $u^{\prime}$ from Proposition \ref{prop:u_n'phi} satisfies: for every $t\in[0,T]$,
  {
  \begin{eqnarray}
    u^{\prime}(t)&=&u_0+\lmd_1\int_0^t\left(u^{\prime}\times\left(\Delta u^{\prime}-\nabla \phi(u^{\prime})\right)\right)(s)\ud s\label{eq:5.1phiH}\\
    &&-\lmd_2\int_0^tu^{\prime}(s)\times\left(u^{\prime}\times\left(\Delta u^{\prime}-\nabla \phi(u^{\prime})\right)\right)(s)\ud s\nonumber\\
    &&+\sum_{j=1}^N\int_0^t(u^{\prime}(s)\times h_j)\circ \ud W^{\prime}_j(s)\nonumber
  \end{eqnarray}}
  in $L^2(\Omega^{\prime};{\mathbb{L}^2})$.
\end{thm}
\begin{proof}
By \eqref{eq:4.14phi} and Lemma \ref{lem:4.5phi},
\begin{equation}\label{eq:2.15phi}
\mathbb{E}^{\prime}\left(\int_0^T\left\|\left(u^{\prime}\times \left(\Delta u^{\prime}-\nabla \phi(u^{\prime})\right)\right)(t)\right\|_{{\mathbb{L}^2}}^2\ud t\right)^r<\infty,\quad r\geq 1.
\end{equation}
And then by \eqref{eq:2.12phi}, we see that
\begin{equation}\label{eq:5.2phi}
  \|u^{\prime}(t,\omega)\times\left(\left(u^{\prime}\times \left(\Delta u^{\prime}-\nabla \phi(u^{\prime})\right)\right)(t,\omega)\right)\|_{\mathbb{L}^2}\leq\left\|\left(u^{\prime}\times \left(\Delta u^{\prime}-\nabla \phi(u^{\prime})\right)\right)(t,\omega)\right\|_{\mathbb{L}^2}
\end{equation}
for almost every $(t,\omega)\in [0,T]\times\Omega^{\prime}$. And so
\[\mathbb{E}^{\prime}\int_0^T\left\|u^{\prime}(t)\times\left(u^{\prime}\times \left(\Delta u^{\prime}-\nabla \phi(u^{\prime})\right)\right)(t)\right\|_{\mathbb{L}^2}^2\ud t<\infty.\]
Therefore all the terms in the equation \eqref{eq:5.1phiH} are in the space $L^2(\Omega^{\prime};{\mathbb{L}^2})$. This completes the proof of Theorem \ref{thm:u'teqH}.
\end{proof}

\begin{thm}\label{thm:u'holder12c}
  The process $u^{\prime}$ defined in Proposition \ref{prop:u_n'phi} satisfies: for every $\alpha\in (0,\frac{1}{2})$,
  \begin{equation}\label{eq:2.17phi}
  u^{\prime}\in C^\alpha([0,T];{\mathbb{L}^2}),\qquad \mathbb{P}^{\prime}-a.s..
  \end{equation}
\end{thm}

We need the following Lemma to prove Theorem \ref{thm:u'holder12c}.

\begin{lem}[Kolmogorov test]\cite{key-14}\label{lem:Kol testphi}
  Let $\{u(t)\}_{t\in[0,T]}$ be a stochastic process with values in a separable Banach space $X$, such that for some $C>0$, $\eps>0$, $\delta>1$ and all $t,s\in[0,T]$,
  \[\mathbb{E}\big\|u(t)-u(s)\big\|_X^\delta\leq C|t-s|^{1+\eps}.\]
  Then there exists a version of $u$ with $\mathbb{P}$ almost surely trajectories being H$\ddot{\textrm{o}}$lder continuous functions with an arbitrary exponent smaller than $\frac{\eps}{\delta}$.
\end{lem}

\begin{proof}[Proof of Theorem \ref{thm:u'holder12c}]
  By \eqref{eq:5.1phi}, we have
  {
  \begin{eqnarray*}
    &&u^{\prime}(t)-u^{\prime}(s)\\
    &=&\lmd_1\int_s^t\left(u^{\prime}\times \left(\Delta u^{\prime}-\nabla \phi(u^{\prime})\right)\right)(\tau)\ud \tau-\lmd_2\int_s^tu^{\prime}(\tau)\times\left(u^{\prime}\times \left(\Delta u^{\prime}-\nabla \phi(u^{\prime})\right)\right)(\tau)\ud \tau\\
    &&+\sum_{j=1}^N\int_s^t(u^{\prime}(\tau)\times h_j\circ\ud W^{\prime}_j(\tau)\\
    &=&\lmd_1\int_s^t\left(u^{\prime}\times \left(\Delta u^{\prime}-\nabla \phi(u^{\prime})\right)\right)(\tau)\ud \tau-\lmd_2\int_s^tu^{\prime}(\tau)\times\left(u^{\prime}\times \left(\Delta u^{\prime}-\nabla \phi(u^{\prime})\right)\right)(\tau)\ud \tau\\
    &&+\frac{1}{2}\sum_{j=1}^N\int_s^t\big(u^{\prime}(\tau)\times h_j\big)\times h_j\ud \tau+\sum_{j=1}^N\int_s^t u^{\prime}(\tau)\times h_j\ud W_j^{\prime}(\tau),\quad 0\leq s<t\leq T.
  \end{eqnarray*}}
  Hence by Jensen's inequality, for $q>1$,
  {
  \begin{eqnarray*}
    &&\mathbb{E}^{\prime}\left[\big\|u^{\prime}(t)-u^{\prime}(s)\big\|_{\mathbb{L}^2}^{2q}\right]\\
    &\leq&\mathbb{E}^{\prime}\Bigg\{|\lmd_1|\int_s^t\left\|\left(u^{\prime}\times \left(\Delta u^{\prime}-\nabla \phi(u^{\prime})\right)\right)(\tau)\right\|_{\mathbb{L}^2}\ud \tau+|\lmd_2|\int_s^t\left\|u^{\prime}(\tau)\times\left(u^{\prime}\times \left(\Delta u^{\prime}-\nabla \phi(u^{\prime})\right)\right)(\tau)\right\|_{\mathbb{L}^2}\ud\tau\\
    &&+\frac{1}{2}\sum_{j=1}^N\int_s^t\big\|u^{\prime}(\tau)\times h_j\times h_j\big\|_{\mathbb{L}^2}\ud \tau+\sum_{j=1}^N\left\|\int_s^tu^{\prime}(\tau)\times h_j\ud W_j^{\prime}(\tau)\right\|_{\mathbb{L}^2}\Bigg\}^{2q}\\
    &\leq&4^{2q-1}\mathbb{E}^{\prime}\Bigg\{|\lmd_1|^{2q}\left(\int_s^t\big\|\left(u^{\prime}\times \left(\Delta u^{\prime}-\nabla \phi(u^{\prime})\right)\right)(\tau)\big\|_{\mathbb{L}^2}\ud \tau\right)^{2q}\\
    &&+|\lmd_2|^{2q}\left(\int_s^t\big\|u^{\prime}(\tau)\times\left(u^{\prime}\times \left(\Delta u^{\prime}-\nabla \phi(u^{\prime})\right)\right)(\tau)\big\|_{\mathbb{L}^2}\ud\tau\right)^{2q}\\
    &&+\left(\sum_{j=1}^N\int_s^t\big\|u^{\prime}(\tau)\times h_j\times h_j\big\|_{\mathbb{L}^2}\ud \tau\right)^{2q}+\left\|\sum_{j=1}^N\int_s^tu^{\prime}(\tau)\times h_j\ud W_j^{\prime}(\tau)\right\|^{2q}_{\mathbb{L}^2}\Bigg\}.
  \end{eqnarray*}}
  By \eqref{eq:2.15phi}, there exists $C_1>0$, such that
  {
  \begin{eqnarray*}
  \mathbb{E}^{\prime}\left(\int_s^t\big\|\left(u^{\prime}\times \left(\Delta u^{\prime}-\nabla \phi(u^{\prime})\right)\right)(\tau)\big\|_{\mathbb{L}^2}\ud \tau\right)^{2q}&\leq&(t-s)^q\mathbb{E}^{\prime}\left(\int_s^t\big\|\left(u^{\prime}\times \left(\Delta u^{\prime}-\nabla \phi(u^{\prime})\right)\right)(\tau)\big\|_{\mathbb{L}^2}^2\ud \tau\right)^q\\
  &\leq&C_1^q(t-s)^q.
  \end{eqnarray*}}

    By \eqref{eq:5.2phi}
  {
  \begin{eqnarray*}
    &&\mathbb{E}^{\prime}\left(\int_s^t\big\|u^{\prime}(\tau)\times\left(u^{\prime}\times \left(\Delta u^{\prime}-\nabla \phi(u^{\prime})\right)\right)(\tau)\big\|_{\mathbb{L}^2}\ud\tau\right)^{2q}\leq\mathbb{E}^{\prime}\left(\int_s^t\big\|\left(u^{\prime}\times \left(\Delta u^{\prime}-\nabla \phi(u^{\prime})\right)\right)(\tau)\big\|_{\mathbb{L}^2}\ud \tau\right)^{2q}\\
    &\leq&(t-s)^q\mathbb{E}^{\prime}\left(\int_s^t\big\|\left(u^{\prime}\times \left(\Delta u^{\prime}-\nabla \phi(u^{\prime})\right)\right)(\tau)\big\|_{\mathbb{L}^2}^2\ud \tau\right)^q\leq C_1^q(t-s)^q.
  \end{eqnarray*}}

  And by \eqref{eq:4.10phi},
  {
  \begin{eqnarray*}
  \mathbb{E}^{\prime}\left(\sum_{j=1}^N\int_s^t\big\|u^{\prime}(\tau)\times h_j\times h_j\big\|_{\mathbb{L}^2}\ud \tau\right)^{2q}\leq\|u_0\|^{2q}_{\mathbb{L}^2} T^q\sum_{j=1}^N\|h_j\|_{\mathbb{L}^\infty}^{4q}(t-s)^q.
  \end{eqnarray*}}

By the Burkholder-Davis-Gundy Inequality,
  {
  \begin{eqnarray*}
    \mathbb{E}^{\prime}\left\|\sum_{j=1}^N\int_s^t u^{\prime}(\tau)\times h_j\ud W_j^{\prime}(\tau)\right\|_{\mathbb{L}^2}^{2q}&\leq&K_q\mathbb{E}^{\prime} \left(\sum_{j=1}^N\int_s^t\big\|u^{\prime}(\tau)\times h_j\big\|_{\mathbb{L}^2}^2\ud \tau\right)^q\leq K_q\|u_0\|_{\mathbb{L}^2}^{2q}\sum_{j=1}^N\|h_j\|_{\mathbb{L}^\infty}^{2q}(t-s)^q.
  \end{eqnarray*}}

  Therefore, let $C=4^{2q-1}\left((|\lmd_1|^{2q}+|\lmd_2|^{2q})C_1^q+\|u_0\|^{2q}_{\mathbb{L}^2} T^q\sum_{j=1}^N\|h_j\|_{\mathbb{L}^\infty}^{4q}+K_q\|u_0\|_{\mathbb{L}^2}^{2q}\sum_{j=1}^N\|h_j\|_{\mathbb{L}^\infty}^{2q}\right)$, we have
  \[\mathbb{E}^{\prime}\left[\big\|u^{\prime}(t)-u^{\prime}(s)\big\|_{\mathbb{L}^2}^{2q}\right]\leq C(t-s)^q,\quad q\geq 1.\]
  Then by Lemma \ref{lem:Kol testphi},
  \[u\in C^\alpha([0,T];{\mathbb{L}^2}), \quad \alpha\in \left(0,\frac{1}{2}\right).\]
  This completes the proof of Theorem \ref{thm:u'holder12c}.
\end{proof}
\vfill

\appendix

\section{Some explanation}
This Appendix aims to clarify the meaning of the process $\Lambda$ from Notation \ref{notation:uxDeltau} and Lemma \ref{lem:uxDeltauinH}. And the explanation present here goes back to Visintin \cite{Visintin}.
\begin{defn}\label{def-MtimesDeltaM}
Assume that $D\subset \R^d$, $d\leq 3$.
Suppose that $M\in H^1(D)$. We say that $M\times \Delta M$ exists in the $L^2(D)$ sense (and write $M\times \Delta M \in L^2(D)$) iff there exists $B\in L^2(D)$ such that for every $u\in W^{1,3}(D)$,
\begin{equation}\label{def:B1}
\lb B,u\rb_{\Ltwo}= \sum_{i=1}^3 \lb D_iM, M\times D_i u\rb_{\Ltwo},
\end{equation}
where $\lb \cdot,\cdot\rb=\lb \cdot,\cdot\rb_{\Ltwo}$.
\end{defn}
\begin{rem*}
  Since $H^1(D) \subset L^6(D)$ and $D_i u\in L^3(D)$, the integral on the RHS above is convergent.
\end{rem*}

\begin{rem*}
  If $M\in D(A)$, then $B=M\times \Delta M$ can be defined pointwise as an element of $L^2(D)$. Moreover by  Proposition \ref{prop:uxAuweak}, \eqref{def:B1} holds, so $M\times \Delta M$ in the sense of Definition \ref{def-MtimesDeltaM}. The next result shows that this can happen also for less regular $M$.
\end{rem*}

\begin{prop}
Suppose that $M_n\in H^1(D)$ so that $\Lambda_n:=M_n\times \Delta M_n \in L^2(D)$ and
\[
\vert \Lambda_n \vert_{\Ltwo} \leq C.
\]
Suppose that
\[
\vert M_n \vert_{H^1} \leq C.
\]
Suppose that
\[M_{n} \to M \mbox{ weakly in } H^1(D).\]
Then $M\times \Delta M\in L^2(D)$.
\end{prop}

\begin{proof}
By the assumptions there exists a subsequence $(n_j)$ and $\Lambda\in L^2(D)$ such that for any $q<6$ (in particular $q=4$)
\begin{eqnarray*}
&&\Lambda_{n_j} \to \Lambda \mbox{ weakly in } L^2(D)\\
&& M_{n_j} \to M \mbox{ strongly in } L^q(D)\\
&& \nabla M_{n_j} \to \nabla M \mbox{ weakly in } L^2(D)
\end{eqnarray*}
We will prove that $M\times \Delta M=\Lambda\in L^2$. Let us fix $u\in W^{1,4}(D)$.

First we will show that

\begin{equation}\label{eqn-Lambda=MtimesDeltaM}
\lb \Lambda ,u\rb= \sum_{i=1}^3 \lb D_iM, M\times D_i u\rb,
\end{equation}

Since
\[\lb \Lambda_n ,u\rb= \sum_{i=1}^3 \lb D_iM_n, M_n\times D_i u\rb,\]
we have

\begin{eqnarray*}
-\lb \Lambda_n ,u\rb&+&\sum_{i=1}^3 \lb D_iM, M\times D_i u\rb\\
&=&-\sum_{i=1}^3 \lb D_iM_n, M_n\times D_i u\rb+\sum_{i=1}^3 \lb D_iM, M\times D_i u\rb
\\
&=&\sum_{i=1}^3 \lb D_iM-D_iM_n, M\times D_i u\rb+ \sum_{i=1}^3 \lb D_iM_n, M\times D_i-M_n\times D_i u\rb\\
&=&I_n+II_n
\end{eqnarray*}
Since $M\times D_i u \in L^2$ and $ D_iM-D_iM_n \to 0$ weakly in $L^2$ we infer that $I_n \to 0$.
Moreover, by the H\"older inequality we have
\begin{eqnarray*}
\vert II_n\vert&\leq & \sum_{i=1}^3 \vert  D_iM_n\vert_{\Ltwo} \vert  M- M_n \vert_{L^4}  \vert D_i u \vert_{L^4} \to 0.
\end{eqnarray*}
Thus, $\lb\Lambda_n ,u\rb \to \sum_{i=1}^3 \lb D_iM, M\times D_i u\rb$. On the other hand, $\lb\Lambda_n ,u\rb \to \lb\Lambda ,u\rb $, what concludes the proof of equality \eqref{eqn-Lambda=MtimesDeltaM} for $u\in W^{1,4}(D)$.\\
Since both sides of equality \eqref{eqn-Lambda=MtimesDeltaM} are continuous with respect to  $ W^{1,3}(D)$ norm of $u$ and  the space $ W^{1,4}(D)$  is dense in $ W^{1,3}(D)$, the result follows.

\end{proof}

\end{document}